\documentclass[12pt]{article}

\usepackage{amsmath,amscd,amsthm,makeidx,graphicx,url,bm,bbm,mathrsfs,palatino,bm,mathtools,latexsym}
\usepackage[bbgreekl]{mathbbol}

\usepackage{amssymb}
\usepackage{a4wide,xy,graphicx}
\usepackage{tikz-cd}
\usepackage{enumerate,comment}
\usepackage{titlesec}
\usepackage{tikz-cd}

\numberwithin{equation}{section}

\titleformat{\paragraph}[runin]
{\normalfont\normalsize\bfseries}{\theparagraph}{1em}{}
\titlespacing*{\paragraph}
{0pt}{3.25ex plus 1ex minus .2ex}{1.5ex plus .2ex}

\usepackage[
breaklinks=true,
pdftitle={Very stable Higgs bundles, the nilpotent cone and mirror symmetry},
pdfauthor={Tam\'as Hausel, Nigel Hitchin}]{hyperref}
\usepackage{color}
\definecolor{tocolor}{rgb}{.1,.1,.5}
\definecolor{urlcolor}{rgb}{.2,.2,.6}
\definecolor{linkcolor}{rgb}{.1,.4,.6}
\definecolor{citecolor}{rgb}{.6,.3,.1}
\hypersetup{colorlinks=true, urlcolor=urlcolor, linkcolor=linkcolor, citecolor=citecolor}

\let\oldmarginpar\marginpar
\renewcommand\marginpar[1]{\-\oldmarginpar[\raggedleft\footnotesize #1]
	{\raggedright\footnotesize #1}}

\makeindex

\xyoption{all}
\input{xypic}

\newcommand{\hookuparrow}{\mathrel{\rotatebox[origin=c]{90}{$\hookrightarrow$}}}

\renewcommand{\div}{{\rm div}}
\newcommand{\im}{{\rm im}}

\newtheorem{theorem}{Theorem}
\newtheorem{proposition}[theorem]{Proposition}
\newtheorem{corollary}[theorem]{Corollary}
\newtheorem{conjecture}[theorem]{Conjecture}
\newtheorem{lemma}[theorem]{Lemma}

\theoremstyle{remark}
\newtheorem{remark}[theorem]{Remark}
\newtheorem{example}[theorem]{Example}

\theoremstyle{definition}
\newtheorem{definition}[theorem]{Definition}

\numberwithin{theorem}{section}

\newcounter{margin}
	{\end{itshape}  \bigskip}

\newcommand{\lie}[1]{\mathfrak{#1}}

\def\beq{\begin{eqnarray}}
	\def\eeq{\end{eqnarray}}
\def\bes{\begin{eqnarray*}}
	\def\ees{\end{eqnarray*}}

\def\bP{{\mathbb P}}

\def\muhat{{\bbmu}}
\def\ddelta{{{\delta}}}

\def\calE{{\mathcal E}}

\def\Nm{{\rm Norm}}

\DeclareMathOperator{\Pic}{Pic}

\DeclareMathOperator{\Spec}{Spec} 
\DeclareMathOperator{\Hom}{Hom}
\DeclareMathOperator{\Tor}{Tor}
\DeclareMathOperator{\Ext}{Ext }

\DeclareMathOperator{\rank}{rank}

\DeclareMathOperator{\coker}{coker}

\def\bE{{\mathbb{E}}}
\def\bPhi{{\mathbb{\Phi}}}

\def\A{{\mathbb A}}

\def\E{\mathbb{E}}
\def\L{\mathbb{L}}
\def\C{\mathbb{C}}
\def\calN{{\mathcal{N}}}
\def\M{{\mathcal{M}}}

\def\calA{{\mathcal{A}}}

\def\calF{{\mathcal{F}}}

\def\calH{\mathcal{H}}

\def\P{\mathbb{P}}
\def\N{\mathbb{Z}_{\geq 0}}

\def\R{\mathbb{R}}

\newcommand{\T}{{\mathbb{T}}}

\def\calC{{\mathcal C}}

\def\calO{{\mathcal O}}

\def\Z{\mathbb{Z}}

\def\m{{\mathfrak m}}

\newcommand{\bino}[2]{\mbox{ $\binom{#1}{#2}$}}
\newcommand{\nc}{\newcommand}

\def\d{{{\delta}}}
\def\mm{{{\ell,m}}}

\usepackage{leftidx}

\newcommand{\calR}{\mathcal{R}}

\newcommand{\ad}{\textnormal{ad}}
\newcommand{\Gr}{\textnormal{Gr}}

\newcommand{\End}{\textnormal{End}}
\newcommand{\tr}{\textnormal{tr}}

\nc{\op}[1]{\mathop{\mathchoice{\mbox{\rm #1}}{\mbox{\rm #1}}
		{\mbox{\rm \scriptsize #1}}{\mbox{\rm \tiny #1}}}\nolimits}
\nc{\al}{\alpha}

\nc{\ep}{\varepsilon} 
\nc{\ga}{\gamma} 
\nc{\Ga}{\Gamma}
\nc{\la}{\lambda} 
\nc{\La}{\Lambda} 
\nc{\si}{\sigma}
\nc{\Sig}{{\Gamma}} 
\nc{\Om}{\Omega} 
\nc{\om}{\omega}

\nc{\SL}{\mathrm{SL}} 
\nc{\GL}{\mathrm{GL}} 
\nc{\SO}{\mathrm{SO}} 
\nc{\PGL}{\mathrm{PGL}}
\nc{\G}{\mathrm{G}}

\nc{\W}{\mathrm{W}}
\nc{\Lg}{\mathrm{L}}
\nc{\Pg}{\mathrm{P}}
\nc{\calL}{{\mathcal L}}
\nc{\Sym}{{\rm Sym}}

\renewcommand{\H}{\mathbb{H}}

\nc{\Frob}{\mathrm{Frob}}
\nc{\spec}{{\rm Spec}}
\def\PP {\mathbb{P}}

\nc{\rN}{\mathrm{N}}

\usepackage{longtable}

\nc{\cpt}{{\op{cpt}}} \nc{\Dol}{{\op{Dol}}} \nc{\DR}{{\op{DR}}}
\nc{\B}{{\op{B}}} \nc{\Triv}{\op{Triv}} \nc{\Hod}{{\op{Hod}}}
\nc{\Log}{{\op{Log}}} \nc{\Exp}{{\op{Exp}}} \nc{\Est}{E_{\op{st}}}
\nc{\Hst}{H_{\op{st}}} \nc{\Left}[1]{\hbox{$\left#1\vbox to
		10.5pt{}\right.\nulldelimiterspace=0pt \mathsurround=0pt$}}
\nc{\Right}[1]{\hbox{$\left.\vbox to
		10.5pt{}\right#1\nulldelimiterspace=0pt \mathsurround=0pt$}}
\nc{\LEFT}[1]{\hbox{$\left#1\vbox to
		15.5pt{}\right.\nulldelimiterspace=0pt \mathsurround=0pt$}}
\nc{\RIGHT}[1]{\hbox{$\left.\vbox to
		15.5pt{}\right#1\nulldelimiterspace=0pt \mathsurround=0pt$}}

\nc{\bee}{{\bf E}} 

\begin{document}


\title{Very stable Higgs bundles, equivariant multiplicity and mirror symmetry}

\author{ Tam\'as Hausel
\\ {\it IST Austria} 
\\{\tt tamas.hausel@ist.ac.at} \and Nigel Hitchin\\ {\it University of Oxford} 
\\{\tt hitchin@maths.ox.ac.uk }   }
\pagestyle{myheadings}

\maketitle

\begin{abstract} We define and study the existence of very stable Higgs bundles on Riemann surfaces, how it implies a precise formula for the  multiplicity of the very stable components of the global nilpotent cone and its relationship to mirror symmetry. The main ingredients are the Bialynicki-Birula theory of $\C^*$-actions on semiprojective varieties, $\C^*$ characters of indices of $\C^*$-equivariant coherent sheaves,   Hecke transformation for Higgs bundles, relative Fourier-Mukai transform along the Hitchin fibration, hyperholomorphic structures on universal bundles and cominuscule Higgs bundles. \end{abstract}

\tableofcontents

\section{Introduction}

\stepcounter{subsection}
Drinfeld \cite{drinfeld}  and Laumon \cite{laumon} call a vector bundle on a smooth complex projective curve $C$ of genus $g$ a {\em very stable bundle} if there is no non-zero nilpotent Higgs field $\Phi\in H^0(C;\End(E)\otimes K)$ on $E$. Laumon \cite[Proposition 3.5]{laumon} proved that a very stable bundle
is semi-stable for $g>0$ and stable for $g>1$ in the sense of Mumford and that very stable bundles form a non-empty open subset in the moduli space of stable bundles on $C$. 

We will assume $g>1$, $d\in \Z$ and $n\in \Z_{>1}$. Let $\M:=\M^d_n$ denote the moduli space of semi-stable rank $n$ degree $d$ Higgs bundles on $C$ i.e. pairs $(E,\Phi)$ with $\Phi\in H^0(C;\End(E)\otimes K)$. It was constructed in Hitchin \cite{hitchin1} by gauge theoretic and in \cite{ nitsure,simpson} by algebraic geometric techniques. Pauly and P\'eon-Nieto \cite{pauly-peon-nieto} proved that $E$ is very stable if and only if \beq \label{stableup}U_{E}=\{(E,\Phi)|\Phi\in H^0(C;\End(E)\otimes K)\} \subset \M\eeq is closed in $\M$. 

There is a generalization of $U_{E}\subset \M$ for more general Higgs bundles. Recall that there is a $\T$-action on $\M$ by $(E,\Phi)\mapsto (E,\lambda\Phi)$ for $\lambda\in \T$. For a stable Higgs bundle
$\calE=(E,\Phi)\in \M^{s\T}$ fixed by the $\T$-action we define its {\em upward flow} as \beq\label{upwardflow}W^+_\calE:=\{(E^\prime,\Phi^\prime)\in \M | \lim_{\lambda \to 0} (E^\prime,\lambda\Phi^\prime)=(E,\Phi)\}.\eeq
In particular, for a rank $n$ degree $d$ vector bundle $E$ we have $W^+_{(E,0)}=U_E$ from \eqref{stableup}. 

We define a $\T$-fixed stable Higgs bundle $\calE\in \M^{s\T}$ to be a {\em very stable Higgs bundle} if the only Higgs bundle $(E^\prime,\Phi^\prime)\in W^+_\calE$ with nilpotent Higgs field is $\calE$.  Our first result is a generalization of \cite[Theorem 1.1]{pauly-peon-nieto}. \begin{theorem}\label{mainclosed}  $\calE\in \M^{s\T}$ is very stable if and only if $W^+_{\calE}\subset \M$ is closed. 
\end{theorem}

From Laumon's result we know that a generic rank $n$ vector bundle is very stable, thus there always exists a very stable Higgs bundle of the form  $(E,0)$ with $E$ very stable, that is,  very stable Higgs  bundles in  the subvariety $\calN\in \pi_0(\M^\T)$ parametrizing semi-stable vector bundles on $C$. 

There are however many other fixed point components. Namely, when  $E=E_1\oplus\dots \oplus E_j$ is a direct sum of vector bundles and $\Phi|_{E_i}\to E_{i+1}K\subset EK$. We call $$(\rank(E_1),\rank(E_2),\dots,\rank(E_j))$$ -- an ordered partition of $n$ -- the {\em type} of the $\T$-invariant semistable Higgs bundle $(E,\Phi)$. For example, for a rank $n$ semistable bundle $E$ the semistable Higgs bundle $(E,0)\in \M^\T$ is of type $(n)$. 

In this paper we  classify all type $(1,\dots,1)$ very stable Higgs bundles:
\begin{theorem} \label{maintype111} Let $$E=L_0\oplus \dots \oplus L_{n-1}$$ be a direct sum of line bundles $L_i$ on $C$. Furthermore, let the Higgs field be given by $$\Phi|_{L_{i-1}}=b_i : L_{i-1}\to L_{i}K\subset EK$$ where $$b_i\in \Hom(L_{i-1},L_{i}K)=H^0(C;L_{i-1}^*L_{i}K).$$ Assume that $(E,\Phi)$ is stable. Then $(E,\Phi)$ is very stable if and only if $$b=b_{n-1}\circ \dots \circ b_1\in H^0(C;L_0^*L_{n-1}K^{n-1})$$ has no repeated zero. 
\end{theorem}

The proof of this proceeds by first verifying it in the case when $b$ has no zeroes. The corresponding Higgs bundle is called a uniformising Higgs bundle, the upward flow from which is nothing but a Hitchin section \cite{hitchin-lie}, which was already observed to be closed  in \cite{hitchin-lie}. In particular when $L_0\cong \calO_C$ then we get the {\em canonical uniformising Higgs bundle} \beq\label{canuni} \calE_0:=(\calO_C\oplus K^{-1}\oplus \dots \oplus K^{-n+1},\Phi|_{K^{-i+1}}=1:K^{-i+1}\to K^{-i}\otimes K).\eeq Then we apply Hecke transformations  at points of the curve to these Hitchin sections, and observe that we can reach all type $(1,\dots,1)$ upward flows with subsequent Hecke transformations, and can infer Theorem~\ref{maintype111} in turn. This argument is detailed in subsection~\ref{heckeverystable}.

From Theorem~\ref{maintype111} we  deduce that in every type $(1,\dots,1)$ component $F\in\pi_0(\M^\T)$ there exists a very stable Higgs bundle. Thus  when $n=2$ every component, being either of type $(2)$ or type $(1,1)$, contains very stable Higgs bundles. However, already for rank $n=3$ there are type $(1,2)$ components $F\in\pi_0(\M^\T)$ which contain no very stable Higgs bundles. This can be seen by computing the multiplicity of a very stable component of the nilpotent cone, which we now explain. 

We recall the Hitchin map \cite{hitchin-stable} $$\begin{array}{cccc} h:&\M&\to &\calA=\oplus_{i=1}^n  H^0(C;K^i)\\ &(E,\Phi)&\mapsto & \det(\Phi-x)\end{array} $$ given by the characteristic polynomial. It is a proper completely integrable system \cite{nitsure,hitchin-stable} with respect to a natural symplectic structure on $\M$. In particular, $\dim(\calA)=\dim(\M)/2$ and the fibres of $h$ are Lagrangian subvarieties. The generic fibre is isomorphic to an Abelian variety -- the Jacobian of a  spectral curve -- while the most singular fibre $N:=h^{-1}(0)$, the so-called {\em nilpotent cone}, is non-reduced as a subscheme. Its components can be understood from the {\em core} of $\M$. Namely, for  $F\in {\pi_0(\M^\T)}$  one defines the {\em downward flow} $$ W^-_F:=\{(E^\prime,\Phi^\prime)\in \M | \lim_{\lambda \to \infty} (E^\prime,\lambda\Phi^\prime)\in F\}.$$ Note that the Hitchin map is proper and $\T$-equivariant with an appropriate weighted $\T$-action on $\calA$. It follows that the  core  $$\calC:=\bigcup_{F\in \pi_0(\M^\T)}W^-_F\cong N_{red}$$ is a proper subvariety isomorphic to $N_{red}$ the reduction of the nilpotent cone. Thus each component $N_F$ of the nilpotent cone $N$ is labelled by $F\in \pi_0(\M^{\T})$ and its reduction is isomorphic to ${N_F}_{red}\cong\overline{W^-_F}$. We denote by $m_F:=\ell(\calO_{N,N_F})$ the multiplicity of the component $N_F\subset N$. We have the following

\begin{theorem} \label{mainmulti} Let $\calE\in \M^{s\T}$ be a very stable Higgs bundle. Denote by $T^+_\calE<T_\calE\M$ the positive part of the $\T$-module $T_\calE\M$ and by $F_\calE$ the $\T$-fixed point component containing $\calE$. We have  \begin{enumerate} \item the pushforward $h_*(\calO_{W^+_\calE})$ is a locally free sheaf on $\calA$ \item the multiplicity of the corresponding component $N_{F_\calE}$ of the nilpotent cone equals \beq \label{multi} m_{F_\calE}=\rank(h_*(\calO_{W^+_\calE}))=\left. { \chi_\T(\Sym(T^{+*}_\calE))\over \chi_\T(\Sym(\calA^*))}\right|_{t=1}\eeq  \item for generic $a\in \calA$  the intersection $W^+_\calE\cap h^{-1}(a)$ is transversal and has cardinality $m_{F_\calE}$. \end{enumerate}
\end{theorem}
Here for a positive $\T$-module $V$ with weight space decomposition $V=\oplus_{\lambda>0} V_\lambda$ we have denoted the  $\T$-character of its symmetric algebra by 
\bes\chi_\T({\Sym(V^*)})= \sum_{\lambda\leq 0} \dim(\Sym(V^*)_\lambda)t^{-\lambda}= \prod_{\lambda>0}\frac{1}{( 1- t^{\lambda})^{\dim(V_\lambda)}} \in \Z[[t]],\ees where $t$ denotes the $\T$-character of weight $-1$. 
The computation of \eqref{multi} is straightforward and when it is non-integer valued we  deduce the non-existence of very stable Higgs bundles in a given component. This is the case for certain types  of $(1,2)$ components in  rank $3$, see Remark~\ref{counter}. However the integrity of $m_{F_\calE}$ is not a complete obstruction for existence of very stable bundles as examples in Remark~\ref{obstruction} for type $(1,3)$ in rank $4$ show.

It is straightforward to compute \eqref{multi} for a very stable Higgs bundle. Consequently, we  deduce the multiplicity of type $(n)$ and type $(1,\dots,1)$ components of the nilpotent cone. 

\begin{corollary} \label{maincor} The multiplicity of $\calN$, the moduli space of semi-stable rank $n$ degree $d$ bundles, in the nilpotent cone is  $$m_\calN=2^{3g-3}3^{5g-5}\dots n^{(2n-1)(g-1)}.$$ Moreover let $N_{F_\calE}$ be the component of the nilpotent cone for a type $(1,\dots,1)$ Higgs bundle $\calE\in \M^{s\T}$ as in Theorem~\ref{maintype111}. Let $m_i=\deg(b_i)$. Then its multiplicity in the nilpotent cone is $$m_{F_{\calE}}=\prod_{i=1}^{n-1}{n\choose n-i}^{m_i}.$$
\end{corollary}

The main motivation of this paper was to try to understand the mirror of upward flows  \eqref{upwardflow} in the framework of Kapustin and Witten \cite{kapustin-witten}. Fixing $\M:=\M_{n}^{-n(n-1)(g-1)}$ we expect that the mirror of the structure sheaf $\calO_{W^+_\calE}$ of the  Lagrangian upward flow ${W^+_\calE}\subset \M$ will be a hyperholomorphic vector bundle 
on $\M$ if and only if $\calE$ is very stable, in which case we expect the mirror to be of rank   $m_{F_\calE}$.  

More precisely,  for any $\calE\in \M^{s\T}$ we introduce the {\em virtual multiplicity} 
\beq\label{virmul}m_\calE(t):={ \chi_\T(\Sym(T^{+*}_\calE))\over \chi_\T(\Sym(\calA^*))}\eeq which is a priori  a rational function in $t$. In the course of proving Theorem~\ref{mainmulti} we find that for a very stable $\calE$ the pushforward $h_*(\calO_{W^+_\calE})$ is a $\T$-equivariant vector bundle on $\calA$, with  $\T$-equivariant fibre at $0$ satisfying $\chi_\T(h_*(\calO_{W^+_\calE})_0)=m_\calE(t)$. In particular, for a very stable $\calE$ the quantity $m_\calE(t)$ is a polynomial in $t$, which we call the {\em  equivariant multiplicity}. 

Its significance for mirror symmetry is that we expect (see Remark~\ref{mirrorate0}) that  the mirror, denoted $\Lambda_\calE$, of a very stable upward flow $\calO_{W^+_\calE}$ should be a $\T$-equivariant vector bundle on $\M$ whose restriction to the canonical Hitchin section is isomorphic to $h_*(\calO_{W^+_\calE})$ and therefore its fibre at the canonical uniformising Higgs bundle $\calE_0$ should have $\T$-character \beq \label{mirrorexp}\chi_\T(\Lambda_\calE|_{\calE_0})=m_\calE(t).\eeq In particular, its rank should be $$\rank(\Lambda_\calE)=\dim(\Lambda_\calE|_{\calE_0})=m_\calE(1)=m_{F_\calE}$$ the multiplicity of $N_{F_\calE}$ in the nilpotent cone $N$. 

In the case of type $(1,\dots,1)$ fixed points, where we have a complete classification of very stable Higgs bundles in Theorem~\ref{maintype111}, we  find a candidate for the mirror, which is supported by the following results sketched below. For details see Section~\ref{mirror}. 

Let $\delta=(\delta_0,\delta_1,\dots,\delta_{n-1})$ with $\delta_0:=c_{01}+\dots+c_{0m_0}$ a divisor of degree $\ell$, where we assume that either $c_{0j}\in C$ or $-c_{0j}\in C$.  For $i>0$ we have the effective divisor $\delta_i={c_{i1}+\dots+c_{im_{i}}}\in C^{[m_i]}$  of degree $m_i\geq 0$, where $c_{ij}\in C$ are points. This data gives rise to a type $(1,\dots,1)$ Higgs bundle  as in Theorem~\ref{maintype111} by setting $$L_i:=(\delta_0+\delta_1+\dots + \delta_{i})K^{-i}$$ and $$b_i:=s_{\delta_i}\in H^0(C;L_{i-1}^*L_{i}K)\cong H^0(C;\calO_C(\delta_i))$$ the defining section of $\calO_C(\delta_i)$. We  denote this Higgs bundle by  $\calE_\delta$. Let us assume that  $\calE_\delta\in \M^{s\T}$ is very stable, i.e. it is stable and the divisor $\delta_1+\dots + \delta_{n-1}$ is reduced. Then $W^+_{\calE_\delta}\subset \M$ is a closed Lagrangian, and we denote its structure sheaf $$\calL_\delta:=\calO_{W^+_{\calE_\delta}}.$$

We construct  \beq\label{lambdaintro}\Lambda_\delta:= \bigotimes_{i=0}^{n-1}\bigotimes_{j=0}^{m_i} \Lambda^{n-i}(\bE_{c_{ij}}),\eeq  using a certain (twisted by a gerbe and appropriately normalised) universal Higgs bundle $(\bE,\bPhi)$ on $C\times \M^s$ and denote $\bE_{c_{ij}}=\bE|_{\M^s\times{\{c_{ij}\}}}$ when $c_{ij}\in C$ and $\bE_{c_{0j}}=\bE^*|_{\M^s\times{\{-c_{0j}\}}}$, when $-c_{0j}\in C$. When $n$ divides $\sum_{i=1}^{n-1} (n-i)m_i$ \eqref{lambdaintro} is a well-defined vector bundle of rank $m_{\calE_\delta}$.

For simplicity we assume that $$d=n\ell+\sum_{i=1}^{n-1}(n-i)m_i-n(n-1)(g-1)=-n(n-1)(g-1).$$ In this case for a generic $a\in \calA^\#\subset \calA$  we can identify $h^{-1}(a)=J(C_a)$ with the Jacobian of the smooth spectral curve $C_a\subset T^*C$. Define $$\begin{array}{cccc}\iota:&\M&\to& \M \\ &(E,\Phi)&\mapsto& (E^*K^{1-n},\Phi^T)\end{array},$$ which on $h^{-1}(a)$ will induce the inverse map. In particular, we will have $\iota(\calE_\delta)=\calE_{\delta^\iota}$ where $\delta^\iota=(\delta^\iota_0,\delta^\iota_1,\dots,\delta^\iota_{n-1})=(-\delta_{n-1}-\dots-\delta_0, \delta_{n-1},\dots,\delta_1).$

For $\M^\#:=h^{-1}(\calA^\#)$ we  construct a relative Poincar\'e bundle $\tilde{\bP}$ on the fibre product $\M^\#\times_{\calA^\#}\M^\#$, which has fibre $J(C_a)\times J(C_a)$ over $a\in \calA^\#$. 
We then have the following 
\begin{theorem} \label{mainmirror} Let $\calE_\delta\in \M^{s\T}$ be a very stable Higgs bundle of type $(1,\dots,1)$ and degree $-n(n-1)(g-1)$.  The relative Fourier-Mukai transform over $\calA^\#$ satisfies $$S(\calL_\delta|_{\M^\#})=(\pi_2)_*(\pi_1^*(\calL_\delta|_{\M^\#})\otimes \tilde{\bP})= \Lambda_\delta|_{\M^\#}$$
	
	and 	$$S(\Lambda_\delta|_{\M^\#})=\calL_{\delta^\iota}[-n^2(g-1)-1]|_{\M^\#}$$
	
	Moreover,   $\Lambda_\delta$ extends to  a  vector bundle with a hyperholomorphic connection on $\M^s$. 
\end{theorem} 

This theorem gives the mirror of $\calL_\delta$  as a hyperholomorphic vector bundle of rank $m_{F_{\calE_\delta}}$. The following will give evidence of the mirror relationship globally, in particular over the nilpotent cone.  To motivate it recall that mirror symmetry should be an equivalence of categories. Not just objects should match but morphisms between them. If two objects have a $\T$-equivariant structure so will the vector space of morphisms between them. Thus its $\T$-character should agree with the $\T$-character of the vector space of morphisms between the mirror objects.

To formalise this for two $\T$-equivariant coherent sheaves $\calF_1$ and $\calF_2$ on $\M$  we denote the {\em equivariant Euler pairing} as

$$\chi_\T(\M;\calF_1,\calF_2)=\sum_{k,l} \dim(H^k({ R} {H\!om}(\calF_1,\calF_2))_l) (-1)^k t^{-l},$$ which
on a semi-projective variety like $\M$ we expect to be a Laurent series in $\Z((t))$. 

\begin{theorem} \label{maineuler}  Let $\calE_\delta,\calE_{\delta^\prime}\in \M^{s\T}$ be two type $(1,\dots,1)$ very stable Higgs bundles of degree $-n(n-1)(g-1)$. Then there exist $\T$-equivariant structures on all coherent  sheaves of the form $\calL_\delta$ and $\Lambda_\delta$  such that $$\chi_\T\left(\M;\calL_{\delta^\prime},\Lambda_{\delta}\right)=\chi_\T(\det(\C_{1}\otimes\calA^*)) \chi_\T\left(\M;\Lambda_{\delta^\prime},\calL_{\delta^\iota}[-n^2(g-1)-1]
	\right). $$
	
\end{theorem}

Thus up to the $q$-power $\chi_\T(\det(\C_{1}\otimes\calA^*))$, where $\C_1$ denotes the weight $1$ $\T$-character,  we find that the $\T$-character of morphisms between our Lagrangian and hyperholomorphic branes agree with that of the conjectured mirror branes. 

The content of  the paper is as follows. In Section~\ref{bb} we work out from scratch the Bialynicki-Birula theory for a semi-projective variety,  to conclude that the upward flows and the core of a semiprojective variety with a homogeneity $1$ symplectic form are Lagrangian. Here we  introduce very stable upward flows and prove some  properties including the general form of Theorem~\ref{mainclosed}.  In Section~\ref{bbhiggs} we determine which Higgs bundles belong to which upward and downward flows in the moduli space of Higgs bundles. In Section~\ref{verystable} we discuss very stable Higgs bundles and prove part $1$ of Theorem~\ref{mainmulti}. Then we examine in detail the Hecke transformation  of a  type $(1,\dots,1)$ $\T$-fixed Higgs bundle and apply this to classify all very stable Higgs bundles of this type proving Theorem~\ref{maintype111}. Moreover we show that the Hecke transform takes a Lagrangian subvariety to another in general. In Section~\ref{multiplicities} we investigate the multiplicity of a very stable component of the nilpotent cone, by deriving an explicit linear algebra formula for it, and study the implications. Here we prove the rest of Theorem~\ref{mainmulti} and Corollary~\ref{maincor}.  In Section~\ref{mirror} we study the relative Fourier-Mukai transform of our type $(1,\dots,1)$ very stable upward flows proving the first part of  Theorem~\ref{mainmirror}, and prove an agreement of $\T$-equivariant Euler forms for pairs of mirror branes settling Theorem~\ref{maineuler}. In Section~\ref{hyper} using the twistor space construction of the moduli space of Higgs bundles we endow the universal bundle at a point in $C$ with a hyperholomorphic connection and from the expression (\ref{lambdaintro}) prove  the second part of Theorem~\ref{mainmirror}. 

In the final Section~\ref{further} we discuss three issues relevant to further progress. First in Subsection~\ref{simple} we compute the virtual multiplicity \eqref{virmul} for the analogue of  the type $(1,\dots,1)$ fixed points for a simple group $G$ instead of $\GL_n$. This is the case where the $\T$-action at the fixed point set has the same form as the Hitchin section for the group $G$ as in \cite{hitchin-lie}, which is of course a closed upward flow. The $b_i$ which we studied for the general linear group are now replaced by sections $b_1,\dots, b_{\ell}$ of line bundles associated to the simple roots of $G$.
We find that unlike  the $\GL_n$ case considered in the body of this paper, the virtual multiplicity is rarely a polynomial. We prove in Proposition~\ref{cominint} that for {\em cominuscule} Higgs bundles where $b_i$ are nonvanishing sections except at  cominuscule roots it is in fact a polynomial.  In this case we offer Conjecture~\ref{cominmirror} that under precise conditions these are very stable and find a potential mirror in terms of the minuscule representations of the Langlands dual group, satisfying the expectation of \eqref{mirrorexp}. Secondly in Subsection~\ref{multiple} we study the first non-trivial wobbly type $(1,1)$ $\SL_2$ Higgs bundle where $b$ has a single double $0$. We find the closure of its upward flow and formulate a conjecture about the mirror of the universal  $\SO_3$ bundle  in the adjoint representation. Finally in Subsection~\ref{cotfib} we discuss the problem of the mirror of a very stable cotangent fibre to the moduli of semi-stable vector bundles. We relate this to considerations of Donagi--Pantev and to Drinfeld--Laumon in the geometric Langlands correspondence.

\vskip.5cm

{\noindent \bf Acknowledgements.} We would like to thank Brian Collier, Davide Gaiotto, Peter Gothen, Jochen Heinloth, Daniel Huybrechts, Quoc Ho, Joel Kamnitzer, G\'erard Laumon, Luca Migliorini, Alexander Minets, Brent Pym,  Peng Shan, Carlos Simpson, Andr\'as Szenes, Fernando R. Villegas, Richard Wentworth, Edward Witten and K\={o}ta Yoshioka for interesting comments and discussions. Most of all we are grateful for a long list of very helpful comments by the referee. We would also like to thank the organizers of the Summer School on Higgs bundles in Hamburg in September 2018, where the authors and Richard Wentworth were  giving lectures and where the work in this paper started by considering the mirror of the Lagrangian upward flows $W^+_\calE$ investigated in \cite{collier-wentworth}. The second author wishes to thank EPSRC and ICMAT for support.

	\section{Bialynicki-Birula partition and decomposition}
\label{bb}

\subsection{Bialynicki-Birula theory for semi-projective varieties}

Let $M$ be a normal, complex, semi-projective algebraic variety.  Semi-projective means that $M$ is quasi-projective and we have a $\T:=\C^\times$ action on $M$  such that  the fixed point locus $M^{\T}$ is projective and   for every $x\in M$ there is a $p\in M^\T$ such that
$\lim_{\lambda\to 0}  \lambda x=p$. By the latter we mean that there exists a $\T$-equivariant morphism $g:\A^1\to M$ with $g(1)=x$ and $g(0)=p$, where $\T$ acts on $\A^1$ in the standard way. 

Examples of semi-projective varieties include  cotangent bundles of smooth projective varieties, moduli spaces of Higgs bundles and Nakajima quiver varieties.

For every $\alpha\in M^{\T}$ denote by  $$W^+_\alpha:= \{ x\in M | \lim_{\lambda\to 0}  \lambda x = \alpha\}$$ the {\em upward flow} from $\alpha$ and $$W^-_\alpha:= \{ x\in M | \lim_{\lambda\to \infty}  \lambda x = \alpha\}$$ the {\em downward flow} from $\alpha$. We also define for connected components $F\in \pi_0(M^{\T})$ of the fixed point locus $$W^+_F:=\cup_{\alpha\in F} W^+_\alpha$$ the {\em attractor} of $F$  
and $$W^-_F:=\cup_{\alpha\in F} W^-_\alpha.$$ the {\em repeller} of $F$. 

Let $M^s\subset M$ denote the smooth locus of $M$. Let $\alpha\in M^{s\T}$ be a smooth fixed point.  Then $\T$ acts linearly on the tangent space $T_\alpha M$. We can decompose this representation into weight spaces
$T_\alpha M\cong \oplus_{k\in \Z} (T_\alpha M)_k$, where $(T_\alpha M)_k<T_\alpha M$ denotes the isotypical component, where $\lambda\in \T$ acts as multiplication by $\lambda^k$. We will denote by $T^+_\alpha M = \oplus_{k>0}(T_\alpha M)_k$ the positive and by $T^-_\alpha M = \oplus_{k<0}(T_\alpha M)_k$ the negative part of this decomposition. Let  $T^0_\alpha M=(T_\alpha M)_0$ denote the zero weight component then we have the decomposition $$T_\alpha M \cong  T^+_\alpha M \oplus T^0_\alpha M \oplus T^-_\alpha M. $$

\begin{proposition}\label{bbprop} For a smooth fixed point $\alpha\in M^{s\T}$ the upward and downward flows  $W^+_\alpha$ and $W^-_\alpha$ are $\T$-invariant locally closed subvarieties in $M$. Furthermore we have $W^+_\alpha \cong T^+_\alpha M$ and $W^-_\alpha\cong T^-_\alpha M$ as varieties with $\T$-action.  \end{proposition}
\begin{proof} This is proved in \cite{bialynicki-birula} for a smooth complete $M$. The statement above could be reduced to that case by equivariantly compactifying $M$ (see below or \cite[Theorem 3]{sumihiro}) and equivariantly resolving the singularities of the compactification (\cite[Corollary 7.6.3]{villamayor}). 
	
	For completeness, and to prepare the ground for the proof of Proposition~\ref{lagrangian}, we present a proof  inspired by the approach of \cite[Lemma 2.2]{bellamy-etal} see also \cite[\S 1]{drinfeld3}.  
	Let $\alpha\in M^{s\T}$.  As the singular locus $\operatorname{Sing}(M)\subset M$ is $\T$-invariant and closed we have   $W^+_\alpha\subset M^s$. Take $\alpha\in U_0\subset M^s$  a $\T$-invariant open affine neighbourhood (using \cite[Corollary 2]{sumihiro}). In particular, $U_0$ is smooth.  The $\T$-action on $U_0$ induces a $\Z$-grading on $\C[U_0]\cong H^0(U_0,\calO_{U_0})$, defined by setting $f\in \C[U_0]_i$ if and only if \beq\label{transrule}f(\lambda p)=\lambda^if(p)\eeq for all $\lambda \in \T$ and $p\in U_0$. In this case we will say that $f$ is homogeneous of degree $\deg(f)=i$. 
	
	Let $\m_\alpha\triangleleft \C[U_0]$ be the maximal ideal of functions vanishing at $\alpha$. As $\alpha$ is $\T$-fixed $\m_\alpha$ is a homogeneous ideal. We can find homogeneous elements $x_1,\dots,x_{n}\in \m_\alpha\subset \C[U_0]$, where $n=\dim(M)$ such that their image  in $\m_\alpha/\m_\alpha^2\cong T^*_\alpha M$ form a basis. The one-forms $dx_1,\dots, dx_n\subset H^0(U_0,\Omega_{U_0})$ then form a basis at $\alpha$ and we let $U$ be the complement of the divisor of zeroes of the section $dx_1\wedge\dots\wedge dx_n\in H^0(U_0,\Omega^n(U_0))$. It is a $\T$-invariant affine open subset $U\subset U_0$, where $dx_1,\dots, dx_n$ remains a basis at every point. In particular, $U$ is non-singular. 
	As $U$ is invariant under the $\T$-action and the closure of any $\T$-orbit in $\M\setminus U$ is still contained in $\M\setminus U$, as it is closed,  we have   $$W^{\pm}_\alpha \subset U.$$ Therefore we can prove Proposition~\ref{bbprop} only working in $U$.
	
	Let $A:=\C[U]$ and $A=\oplus_{n\in \Z} A_n$ be the grading induced by the $\T$-action. Let $A_+=\oplus_{n\in \Z_{>0}} A_n$ and $A_-=\oplus_{n\in \Z_{<0}} A_n$. Thus in particular $A=A_-\oplus A_0\oplus A_+$.  First we observe that 
	
	\begin{lemma} $U^\T=V(f|f\in A_++A_-)\subset U$ i.e. the subvariety of zeroes of homogenous $f$ of non-zero degree. Thus $\C[U^\T]=A/(A_+,A_-)$. \end{lemma}
	\begin{proof} Clearly if $p\in U^\T$, $f \in A_i$ with $i\in \Z\setminus \{0\}$    and $\lambda\in \T$ such that $\lambda^i\neq 1$ then $$f(p)=f(\lambda p)=\lambda^if(p)=0.$$ On the other hand if $p\in U$ is such that  $f(p) =0$  for all $f\in A_i$ with $i\in \Z\setminus \{0\}$ then the maximal ideal at $p$ satisfies $(A_++A_-)\subset \m_p$  and so $\m_p$ is a homogenous ideal, thus $p\in U^\T$. \end{proof}
	
	Let $F:=U^\T$ and, as above, $W^+_{F}=\coprod_{\beta\in F} W^+_\beta=\{p\in U|\lim_{\lambda\to 0}\lambda p\in F\}$ the attractor of $F$ in $U$. Then we have 
	\begin{lemma} \label{ringofw} $W^+_{F}=V(f|f\in A_-)$ i.e. the subvariety of zeroes of homogenous $f$ of negative degree. Thus $\C[W^+_{F}]=A/(A_-)$. 
	\end{lemma}
	\begin{proof} Again it is clear that if $p\in W^+_{F}$ then $f(p)=0$ for $f\in A_-$. Namely if $f\in A_i$ for $i<0$ then $$f(\lim_{\lambda\to 0} \lambda p)=\lim_{\lambda\to 0} f(\lambda p)=\lim_{\lambda\to 0} \lambda^{i} f(p)$$ implies that $f(p)=0$.
		
		On the other hand assume that for some $p\in U$ we have  $f(p)=0$ for all $f\in A_-$. Then let $\pi_p:A\to A/\m_p\cong \C$ and define \bes \begin{array}{cccc}g:&A&\to& \C[x]\\ &\sum_i a_i&\mapsto& \sum_i \pi_p(a_i)x^i
		\end{array},\ees where we write every element in $A$ in the form of a finite sum of homogeneous elements $\sum_i a_i$ where $a_i\in A_i$. Because $f(a_i)=0$ for $i<0$, the definition makes sense. We see that $g$ is a graded ring homomorphism.  Observe also that for $z\in \A^1=\Spec(\C[x])$ the composition $$(\pi_{z}\circ g)\left(\sum_i a_i\right)=\sum_i \pi_p(a_i)z^i.$$ When $z= 1$ we see that $\ker(\pi_{1}\circ g)=\m_{p}$ and that $\ker(\pi_0\circ g)=(\ker(\pi_p|_{A_0}),A_-,A_+)$ is a maximal ideal containing $A_-$ and $A_+$ thus corresponds to a fixed point of the $\T$-action. Indeed the ring homomorphism $g$ gives a $\T$-equivariant map $\A^1\to U$, such that $g(1)=p$ and $g(0)=\lim_{\lambda\to 0} \lambda p\in F$ by  definition. Thus $p\in W^+_{F}$. The result follows.
	\end{proof}
	
	\begin{remark} We can build a map $\pi:W^+_{F}\to F$ by sending $p\in W^+_{F}$ to $\lim_{\lambda \to 0} \lambda p\in F$. By the construction above we see that the maximal ideal of $\pi(p)$ is $(\ker(\pi_p|_{A_0}),A_-,A_+)\triangleleft A$. Thus we see that $\pi$ is induced by the map \beq\label{pistar}\begin{array}{cccc}\pi^*:&A/(A_+,A_-)&\to& A/(A_-)\\ &a \mbox{ \rm mod } (A_+,A_-)&\mapsto & a_0 \mbox{ mod } (A_-)\end{array},\eeq where $a=\sum_{i\in \Z} a_i$ is the sum of homogeneous elements $a_i\in A_i$. It is straightforward to check that $\pi^*$ is a well-defined graded ring homomorphism and that $(\pi^*)^{-1}(\ker(\pi_p))=(\ker(\pi_p|_{A_0}),A_-,A_+)$.   
	\end{remark}
	
	Next we determine the tangent spaces of $F$, $W^+_{F}$, $W^-{F}$, $W^+_\alpha$ and $W^-_\alpha$. To do this we will denote by $\partial_i\in H^0(U,TU)$ the dual basis to $dx_i$. Recall that  $\lambda\in \T$ acts on the function $x_i\in\C[U]$ via the formula $$\lambda\cdot x_i(p)=x_i( \lambda^{-1}p)=\lambda^{-\deg(x_i)}x_i(p),$$ i.e. $\lambda\cdot x_i=\lambda^{-deg(x_i)}x_i$ which should not be confused with the transformation rule in \eqref{transrule}.  In particular, the induced action on one-forms gives \beq\label{transdxi} \lambda\cdot d x_i =(m_\lambda)_*(d x_i)=\lambda^{-deg(x_i)}dx_i, \eeq where $m_\lambda$ denotes the action isomorphisms:  $$\begin{array}{cccc}m_\lambda: &U&\to& U\\ &p&\mapsto& \lambda p\end{array}$$ for $\lambda\in \T$. One can compare\footnote{There seems to be a related confusion in \cite[Definition 1.1]{bellamy-etal}. Their  definition of elliptic $\T$-action  should either require the existence of limit points in the $t\to 0$ limit as in this paper, or replace the definition of positive weight using $(m_t)_*$ instead of $m_t^*$. In fact, when they use the definition of positive weight in the displayed line after (2.1) on page 2613 of \cite{bellamy-etal}, they use the correct $(m_t)_*$ instead of $m_t^*$ as in   \cite[Definition 1.1]{bellamy-etal}. } \eqref{transdxi} to $$m_\lambda^*(dx_i)=d(m_\lambda^*(x_i))=\lambda^{\deg(x_i)}dx_i .$$  This way we see that $\partial_i$ will have weight $\deg(x_i)$ in the sense that \beq\label{homdi} (m_\lambda)_*(\partial_i)=T m_\lambda(\partial_i)=\lambda^{\deg(x_i)}(\partial_i).\eeq  As a reality check for $\lambda\in \T$  we compute, using \eqref{transrule} $$dx_j(T m_\lambda (\partial_i))=m_\lambda^*(dx_j)(\partial_i)=d ( x_j\circ m_\lambda)(\partial_i)=  \lambda^{\deg x_j} dx_j(\partial_i)=\lambda^{\deg x_j} \delta_{i,j},$$ showing \eqref{homdi}.
	\begin{lemma} For $\beta\in F$ we have $$T_\beta F = (T_\beta U)^\T\cong \operatorname{span}(\partial_i|_\beta|\deg(x_i)=0)\subset T_\beta U.$$ In particular, $F\subset U$ is a non-singular subvariety. 
	\end{lemma}
	\begin{proof} Dually, we will prove that the kernel of the surjective map $g:T^*_\beta U\to T^*_\beta F$ is $$\operatorname{span}(dx_i|_\beta| \deg(x_i)\neq 0)<T^*_\beta U.$$ It is clear that $dx_i|_\beta\in \ker(g)$ when $\deg(x_i)\neq 0$, because $g$ is a map of $\T$-modules when we endow $T^*_\beta F$ with the trivial $\T$-module structure. Let $\eta\in (T^*_\beta U)^\T\cap \ker(g)$. We can represent $\eta=d {y}_0$ with the homogeneous degree $0$ element ${y}_0\in \mathfrak{m}_{\beta,U}\triangleleft A$. We will  denote by $$\bar{y}_0\in \mathfrak{m}_{\beta,F}\triangleleft A/(A+,A_-)\cong \C[F]$$ the image of ${y}_0$ in the quotient $A/(A+,A_-)$.  By assumption $0=g(\eta)=d{\bar{y}}_0$ thus $\bar{y}_0\in \mathfrak{m}_{\beta,F}^2.$ Thus there exists some $\bar{c}_i,\bar{d}_i\in  \mathfrak{m}_{\beta,F}$ such that $\bar{y}_0=\sum_i \bar{c}_i\bar{d}_i$. Denoting by ${c}_i$ and ${d}_i$ some lift of $\bar{c}_i$ and $\bar{d}_i$ in $\mathfrak{m}_{\beta,U}$, we get that ${y_0}-\sum_i {c}_i{d}_i\in (A_+,A_-)$. We can assume that  $c_i$ and $d_i$ are homogeneous and in turn that $\deg(c_i)=\deg(d_i)=0$.   We have  some $a_j\in A_+$, $b_j\in A_-$ such that ${y_0}-\sum_i {c}_i{d}_i=\sum_j a_jb_j$. We can assume that $a_j$ and $b_j$ are homogeneous of non-zero degree and in turn that $\deg(a_j)+\deg(b_j)=0$. 
		Denoting $\pi_\beta(a_j)=\lambda_j \in A/\m_{\beta,U}\cong \C$ and $\pi_\beta(b_j)=\mu_j \in A/\m_{\beta,U}\cong \C$ we get that $a_j-\lambda_j$ and $b_j-\mu_j$ are in $\m_{\beta,U}$. Thus we  conclude that $$y_0-(\sum_j a_j \mu_j+\lambda_j b_j)\in \m_{\beta,U}^2.$$ As $\m_{\beta,U}^2$ is a homogeneous ideal we deduce that $y_0\in \m_{\beta,U}^2$ proving the first claim. 
		
		Finally, we deduce that $\dim(\m_{\beta,U}/\m^2_{\beta,U})=\#\{dx_i|\deg(x_i)=0\}$ is independent of $\beta\in F$, thus indeed $F$ is non-singular.
	\end{proof}
	\begin{lemma} \label{tangentattractor} For $p\in W^+_{F}$ we have $T_p W^+_{F}=\operatorname{span}(\partial_i|_p|\deg(x_i)\geq 0)<T_pU.$ Similarly for $p\in W^-_{F}$ we have $T_p W^-_{F}=\operatorname{span}(\partial_i|_p|\deg(x_i)\leq 0)<T_pU.$
	\end{lemma}
	\begin{proof} Again, we will prove dually that the kernel of the surjective map $h:T^*_p U\to T^*_p W^+_{F}$ is $$\operatorname{span}(dx_i|_p| \deg(x_i)< 0)<T^*_p U.$$ By Lemma~\ref{ringofw} for $\deg(x_i)<0$ the function $x_i$ vanishes on $W^+_{F}$ and thus $dx_i|_p$ is in the kernel of $h$. 
		Assume now that $\sum_{\deg x_i\geq 0} \alpha_i dx_i|_p\in \ker(h)$ for some $\alpha_i\in \C$.  This means that $$\sum_{\deg x_i\geq 0} \alpha_i (x_i-\pi_p(x_i))\in \m_{p,W^+_{F}}^2.$$ It follows that there are $c_j,d_j\in \m_{p,U}$ such that $$\sum_{\deg x_i\geq 0} \alpha_i (x_i-\pi_p(x_i))-\sum_j c_j d_j\in (A_-).$$ So we have homogeneous $a_k\in A_-$ and $r_k\in A$ such that $$\sum_{\deg x_i\geq 0} \alpha_i (x_i-\pi_p(x_i))-\sum_k r_k a_k\in \m_{p,W^+_{F}}^2 .$$ By writing $r_k=r_k-\pi_p(r_k)+\pi_p(r_k)$ and noting that $\pi_p(a_k)=0$ as $p\in W^+_{F}$ and $a_k\in A_-$ we get
		$$\sum_{\deg x_i\geq 0} \alpha_i (x_i-\pi_p(x_i))-\sum_k \beta_k a_k\in \m_{p,W^+_{F}}^2 ,$$ where $\beta_k=\pi_p(r_k)\in \C$. In other words \beq \label{form0} 0=\sum_{\deg x_i\geq 0} \alpha_i dx_i|_p-\sum_k \beta_k da_k|_p \in T^*_p(U).\eeq Our final observation is that for $\deg(x_i)\geq 0$ the function $\partial_i(a_k)$ has degree $-\deg(x_i)+\deg(a_k)<0$. This is because $da_k=  \sum_j \partial_j(a_k)dx_j$ therefore $$ da_k\wedge dx_1\wedge\dots \wedge dx_{i-1}\wedge d x_{i+1}\wedge \dots \wedge dx_n=(-1)^{i-1} \partial_i(a_k)dx_1\wedge \dots \wedge dx_n. $$ Thus $\partial_i(a_k)$ vanishes on $W^+_{F}$. This means that $da_k|_p$ is a linear combination of the $dx_i$ with $\deg(x_i)<0$ and so $\sum_{\deg x_i\geq 0} \alpha_i dx_i|_p$ and $\sum_k \beta_k da_k|_p$ are linearly independent. Thus \eqref{form0} implies that $\sum_{\deg x_i\geq 0} \alpha_i dx_i|_p=0$. The first statement follows. 
		
		The proof of the second statement on the downward flows is  the same as above for the inverse action of $\T$ on $U$.   \end{proof}
	
	\begin{lemma} \label{tangentup} For $\alpha\in F$, and $p\in W^+_\alpha$  we have  $T_p W^+_{\alpha}=\operatorname{span}(\partial_i|_p|\deg(x_i)> 0)<T_pU.$ Similarly,  for $p\in W^-_\alpha$  we have  $T_p W^-_{\alpha}=\operatorname{span}(\partial_i|_p|\deg(x_i)< 0)<T_pU.$
	\end{lemma}
	
	\begin{proof} First we note that the locus of zeroes of $(x_1,\dots,x_n)$ in $U$ is zero dimensional, because $dx_1,\dots,dx_n$ is a basis at every point. Also this locus is invariant by $\T$, they are thus fixed points in $F$. All non-negative degree elements vanish on $F$ therefore the zeroes of the $0$ degree elements $${\bf x}_0:=\{x_i|\deg(x_i)=0\}$$ on $F$ are precisely the zeroes of $(x_1,\dots,x_n)$ on $U$. 
		
		By construction one such common zero is $\alpha$ and so $W^+_\alpha$ is a connected component of the variety of zeroes of $\bf{x}_0$ on $W^+_{F}$. In particular, $W^+_\alpha\subset U$ is closed and thus it is locally closed in $M$. 
		
		Thus we have that  $W^+_\alpha$ is a connected component of $\Spec(A/(A_-,\bf{x}_0))$. For $p\in W^+_\alpha$ we determine the kernel of the surjective map $f:T^*_p U\to T^*_p W^+_\alpha$. All functions in $A_-$ and  $\bf{x}_0$ restrict to $W^+_\alpha$ as zero thus $dx_i|_p$ is in the kernel when $\deg(x_i)\leq 0$. Let now $\sum_{\deg x_i>0} \alpha_i dx_i|_p \in \ker(f)$. This means that $$\sum_{\deg x_i> 0} \alpha_i (x_i-\pi_p(x_i))\in \m_{p,W^+_\alpha }^2,$$ hence there are $c_j,d_j\in \m_{p,U}$ such that $$\sum_{\deg x_i> 0} \alpha_i (x_i-\pi_p(x_i))-\sum_j c_j d_j\in (A_-,\bf{x}_0).$$ In turn, this means that there are homogeneous elements $a_k\in A_-$  and scalars $\beta_j\in \C$ such that \beq \label{final0} 0=\sum_{\deg x_i>0} \alpha_i dx_i|_p-\sum_k  da_k|_p-\sum_{\deg x_j=0} \beta_j dx_j|_p\in  T^*_p(U).\eeq Just as in the proof of Lemma~\ref{tangentattractor} above we can argue that as $a(p)=0$ for all $a\in A_-$, $da_k|_p\in \operatorname{span}(dx_i|_p|\deg x_i<0)$, thus the linear independence of $\{dx_1|_p,\dots,dx_n|_p\}$ and \eqref{final0} implies that $\sum_{\deg x_i>0} \alpha_i dx_i|_p =0$ proving the first claim.
		
		The second claim again follows from the first by inverting the $\T$-action on $U$.
	\end{proof}
	
	To finish the proof of Proposition~\ref{bbprop} we define the map \beq\label{model}f_\alpha:W^+_\alpha\to T^+_\alpha\eeq sending $$p\mapsto \sum_{\deg(x_i)>0} x_i(p)\partial_i$$ where $\partial_i\in T_\alpha M$ is the dual basis of $dx_1,\dots,dx_n$ at $\alpha$. By Lemma~\ref{tangentup} the derivative of $f_\alpha$  is an isomorphism, thus it is \'etale and so open. Furthermore $f_\alpha$ is $\T$-equivariant and  its image is an open $\T$-invariant subvariety in the positive $\T$-module $T^+_\alpha M$, containing the origin, thus $f_\alpha$ is surjective. As $T^+_\alpha M$ is simply connected, it follows that $f_\alpha$ is a trivial covering, and $W^+_\alpha$ being connected implies that $f_\alpha$ is an isomorphism. 
\end{proof}

\begin{remark} \label{afffibr}
	If we put the maps $f_\alpha$ in a family over $F$ we get a $\T$-equivariant isomorphism $f:W^+_{F}\cong  T^+F$ showing that $U^s_{F_\alpha}\subset M$ is a locally trivial affine fibration, another main result of \cite{bialynicki-birula}.  
\end{remark}

\begin{definition}	We call $M=\coprod _{\alpha}W^+_\alpha$ the {\em Bialynicki-Birula partition} and $M=\coprod _{F}W^+_F$ the {\em Bialynicki-Birula decomposition}. Define also \beq\label{core}\calC:=\coprod_{\alpha\in M^\T} W^-_\alpha=\coprod_{F\in \pi_0(M^\T)} W^-_F\eeq to be the {\em core} of $M$. 
	
\end{definition}

Additionally, in all our examples  $(M^s,\omega)$ with $\omega\in \Omega^2(M)$ will be  symplectic  s.t. \beq \label{weightone}\lambda^*(\omega)=\lambda\omega.\eeq Then we have
\begin{proposition} \label{lagrangian} Let $M$ be a normal, semi-projective complex variety with $(M^s,\omega)$ symplectic with $\omega\in \Omega^2(M^s)$. Assume \eqref{weightone} that $\omega$ is homogeneous of weight $1$.  For a smooth point $\alpha\in M^{s\T}$ the subspaces $T^+_\alpha M, T^{\leq 0}_\alpha M:=T^0_\alpha M\oplus T^+_\alpha M <T_\alpha M$ and subvarieties $W^+_\alpha,W^-_{F_\alpha}\subset M^s$ are Lagrangian.
\end{proposition}
\begin{proof} As $\alpha\in M^{s\T}$ the tangent space $T_\alpha M$ is a $\T$-module. Let $X,Y\in T_\alpha X$ be homogeneous  with weight $\nu_1$ and $\nu_2$ respectively. That means that $\lambda\cdot X=\lambda^{\nu_1} X$ and $\lambda\cdot X=\lambda^{\nu_2} Y$. Then from \eqref{weightone} $$\lambda\omega(X,Y)=\lambda^*(\omega)(X,Y)=\omega(\lambda\cdot X,\lambda\cdot Y)=\lambda^{\nu_1+\nu_2}\omega(X,Y).$$ Consequently $\omega(X,Y)=0$ unless $\nu_1+\nu_2=1$. Thus $\omega$ is trivial on $T^+_\alpha$ and $T^{\leq 0}_\alpha$. As $T_\alpha M=T^+_\alpha\oplus T^{\leq 0}_\alpha$ both subspaces are Lagrangian.
	
	Recall the construction of $U\subset M^s$, the $\T$-invariant affine neighbourhood of $\alpha\in U$ from the proof of Proposition~\ref{bbprop} above.  Let $\omega=\sum_{i<j} f_{i,j} dx_i\wedge dx_j$ for some unique $f_{i,j}\in \C[U]$. As $\omega$, $x_i$ and $x_j$ are homogeneous so is $f_{i,j}$ and $\deg{f_{i,j}}=1-\deg(x_i)-\deg(x_j)$. Let now $p\in W^+_\alpha $. By Lemma~\ref{tangentup} the tangent vectors $\partial_k|_p\in T_p U$ with positive weights will be a basis for $T_p W^+_\alpha$. Take two such $\partial_k|_p, \partial_l|_p\in T_pW^+_\alpha$ and compute $$\omega(\partial_k|_p,\partial_l|_p)=\sum_{i<j} f_{i,j} dx_i\wedge dx_j(\partial_k|_p,\partial_l|_p)=f_{k,l}(p).$$
	The degree of $x_l$ and $x_k$ are positive, so the degree of $f_{k,l}$ is negative. In particular, $f_{k,l}$ vanishes on $W^+_\alpha$. Thus $f_{k,l}(p)=0$, and $T_pW^+_\alpha\subset T_p U$ is isotropic, and half-dimensional as $T_\alpha W^+_\alpha=T_\alpha^+<T_\alpha M$ is Lagrangian. Thus $W^+_\alpha$ is indeed Lagrangian in $U\subset M^s$. Similarly we can prove that $W^-_F$ is Lagrangian in $M^s$. The result follows.\end{proof}
\begin{remark} Our approach in the  proof of Proposition~\ref{lagrangian} above was inspired by \cite[\S 2.2]{bellamy-etal}. Similar results were mentioned in \cite[Proposition 7.1]{nakajima}.
\end{remark}
\subsection{Very stable upward flows}
\begin{definition} Let $\alpha\in M^{s\T}$. We say that $\alpha$ is {\em very stable} if the only point in its upward flow which is also in the core $\calC$ of \eqref{core} is $\alpha$: $$W^+_\alpha\cap \calC=\{\alpha\}.$$ 
\end{definition} 

\begin{remark} The origin of the term {\em very stable} is in the work of Drinfeld \cite{drinfeld} and  Laumon \cite{laumon}, which was applied to vector bundles on curves (see Section~\ref{verystable}). However there is an accidental coincidence in  terminology. Namely, Bialynicki-Birula \cite{bialynicki-birula2} using terminology of Smale \cite[II.2]{smale} calls our upward flow $W^+_\alpha$ the {\em stable} subscheme of $\alpha$ while our downward flow $W^-_\alpha$ is called the {\em unstable} subscheme of $\alpha$. Our definition above can be formulated to say that the stable subscheme $W^+_\alpha$ is {\em very stable} if and only if it disjoint from the unstable subschemes $W^-_\beta$ for $\beta\neq \alpha$. 
\end{remark}
\begin{proposition} \label{upwardclose} $\alpha\in M^{s\T}$ is very stable if and only if $W^+_\alpha\subset M$ is closed. 
\end{proposition}
\begin{proof}If $\alpha$ is not very stable then there exists  $\beta\neq\alpha \in W^+_\alpha\cap \calC$  in the upward flow of $
	\alpha$. Then $\lim_{\lambda\to \infty}(\lambda\cdot \beta)\in M^\T$ is fixed by the $\T$-action. Moreover a linearized $\T$-action on a very ample line bundle on $M$ (which exists, because $M$ is normal and quasi-projective and \cite[Corollary 7.2]{dolgachev}) embeds it equivariantly in some projective space $M\subset {\mathbb P}^N$ with linear $\T$-action, where the closure of every non-trivial $\T$-orbit is linear, thus has two distinct $\T$-fixed points. Consequently, $$\lim_{\lambda\to \infty}(\lambda\cdot \beta)\neq\lim_{\lambda\to 0}(\lambda\cdot \beta)=\alpha$$and thus $W^+_\alpha$ is not closed.
	
	Assume now that $\alpha\in M^{s\T}$ is very stable i.e. $$W^+_{\alpha}\cap \calC=\{\alpha\}$$ as sets.  Then we will determine the image of  $W^+_\alpha\setminus \calC$ in the geometric quotient $$Z=(M\setminus \calC)/\T.$$ By construction it is $$\P(W^+_\alpha)=W^+_\alpha\setminus \{\alpha\}/\T\subset Z$$  isomorphic to the  weighted projective space $\P(T_\alpha^+)=(T_\alpha^+\setminus \{0\})/\T$. By the construction of
	\cite[(1.2.3)]{hausel-large} (c.f. also \cite[\S 11]{simpson3} and \cite{hausel-compact} in the Higgs moduli space case) we have the compactification $$\overline{M}=(M\times\C)/\! /\T=\left(M\times\C\setminus \calC\times\{0\}\right)/\T\cong M\sqcup Z.$$ We 
	see that the geometric quotient $$\P(W^+_\alpha\times \C)= \left((W^+_\alpha\times \C)\setminus (\alpha,0)\right)/\T= W^+_\alpha\sqcup \P(W^+_\alpha) \subset M\sqcup Z = \overline{M}$$ is also a weighted projective space, thus projective, thus closed in $\overline{M}$. It follows that all boundary points of the closure $\overline{U}_\alpha$ in $\overline{M}$ lie on the divisor at infinity $Z\subset \overline{M}\setminus M$, thus $W^+_\alpha$ is closed in $M$. 
\end{proof}

 \section{Bialynicki-Birula theory on the moduli space of Higgs bundles}
 \label{bbhiggs}

Fix $C$ a smooth complex projective curve of genus $g>1$.
Recall that a vector bundle $E$ on $C$ is called 
{\em stable} (resp. {\em semi-stable}) 
if for all proper subbundles $F\subset E$ \beq \label{stable}{\deg(F)\over {\rm rank}(F)}
< {\deg(E)\over {\rm rank}(E)}\eeq (resp. ${\deg(F)/ {\rm rank}(F)}
\leq {\deg(E)/ {\rm rank}(E)}$).

Similarly, a Higgs bundle $(E,\Phi)$ where $\Phi\in H^0(\End(E)\otimes K)$, is stable (resp. semi-stable) if \eqref{stable} holds for $\Phi$-invariant proper subbundles $F\subset E$. 

We fix integers $n\in \Z>0$ and $d\in \Z$.	The moduli space of semistable Higgs bundles $\M:=\M_{n}^{d}$ of rank $n$ and degree $d$ was constructed by \cite{hitchin} by gauge-theoretic and by \cite{nitsure} and \cite{simpson1} by algebraic geometric methods.  At its smooth points it carries a hyperk\"ahler metric, in particular an algebraic symplectic structure. It is  a normal \cite[Corollary 11.7]{simpson1b} quasi-projective variety. The open subset $\M^s\subset \M$ of stable points is precisely the smooth locus.  It is a coarse moduli space, but carries a universal Higgs bundle  \'etale locally on $\M^s$ \cite[Theorem 4.5.(4)]{simpson1}. 

For a Higgs bundle $(E,\Phi)$ we can compute the characteristic polynomial of the Higgs field as $\det(t-\Phi)=t^n+a_{1}t^{n-1}+\dots+a_n$, where $a_i\in H^0(C;K^i)$. This leads to the  map $$h:\M\to \oplus_{i=1}^n  H^0(C;K^i),$$ a proper, completely integrable Hamiltonian system \cite{hitchin-stable,simpson,nitsure} in particular, the fibres are Lagrangian at their smooth points. 

The moduli space $\M$ carries a $\C^\times$-action defined by $(E,\Phi)\mapsto (E,\lambda\Phi)$ for $\lambda\in \C^\times$. If we let $\C^\times$ act with weight $i$ on  $H^0(C;K^i)$ then $h$ becomes $\C^\times$-equivariant. As $h$ is proper and all the weights on $\oplus_{i=1}^n  H^0(C;K^i)$ are positive, we  deduce that $\M$ is semi-projective with respect to this $\T$-action. 

Finally, we recall that there is a canonical symplectic structure on $\M^s$ originally arising as part of its hyperk\"ahler structure. The tangent space at $(E,\Phi)\in \M^s$ can be identified (see e.g. \cite{biswas-ramanan})
\beq\label{tangent} T_{(E,\Phi)}\M^s\cong \H^1\left(C;\End(E)\stackrel{\ad(\Phi)}{\to} \End(E)\otimes K\right).\eeq Serre duality will give an isomorphism between this hypercohomology and its dual. This defines a symplectic form on $T_{(E,\Phi)}$. To see that it gives rise to a closed form $\omega$, we recall \cite[Theorem 4.3]{biswas-ramanan} that $\omega=d\theta$ where $$\theta:\H^1\left(C;\End(E)\stackrel{\ad(\Phi)}{\to}\End(E)\otimes K\right)\to \C$$ is given by first mapping to $H^1(C;\End(E))$ and then using  Serre duality to pair this element with $\Phi\in H^0(C;\End(E)\otimes K)$.  With respect to the $\T$-action we clearly have that $\lambda^*(\theta)=\lambda\theta$ and thus $\lambda^*(\omega)=\lambda\omega$ is also of weight $1$.

In this paper we will study the upward and downward flows of the Bialynicki-Birula partition on $\M$.  

\subsection{Fixed points of the $\T$-action $\M^{s\T}$}
First we recall a parametrisation of $\M^\T$. If $(E,\Phi)\in \M^{s\T}$ is a stable Higgs bundle fixed by the $\T$-action then we have for every $\lambda\in\T$ $$(E,\lambda\Phi)\cong (E,\Phi).$$ The isomorphism above gives a vector bundle automorphism $f_\lambda:E\to E$  making the diagram \beq\label{fixed}\begin{array}{ccc} E&\stackrel{\lambda \Phi}{\longrightarrow} & EK\\ \!\!\!\!\! f_\lambda \downarrow &&\,\,\,\,\,\downarrow f_\lambda \\ E &\stackrel{\Phi}{\longrightarrow}& EK \end{array}\eeq
commutative. These define a fibrewise $\T$-action on $E$. Let $E=L_0\oplus \dots\oplus  L_k$ be the weight space decomposition of this $\T$-action, for subbundles $L_i<E$ and $0\leq k \leq n-1$. Let $\T$ act on $L_i$ by weight $w_i\in \Z$, in other words $f_\lambda$ acts on $L_i$ as multiplication by $\lambda^{w_i}$. 
From the diagram \eqref{fixed} we see that $\Phi$ maps the weight $w_i$ space to the $w_i-1$ weight space. Using an overall scaling we can assume that $w_i=-i$. Then $\Phi(L_i)\subset L_{i+1}K$. 
We call the ordered partition $(\rank(L_0),...,\rank(L_k))$  of $n$ the {\em type} of the fixed point $(E,\Phi)\in \M^{s\T}$.  Similarly we can talk about the type of a component of the fixed point set  $F\in \pi_0(\M^{s\T})$ as the type is locally constant on $\M^{s\T}$. The latter can be seen by noting that \'etale locally we have a $\T$-equivariant universal bundle on $\M^s$ and on the connected component $F\in \pi_0(\M^{s\T})$  the weight space for a given weight of the $\T$-action forms a vector bundle. 

\begin{example} For example when $E=L_0$ the type is $(n)$ and $\Phi=0$, thus $E$ is a semistable bundle. We will denote by $\calN$ the moduli space of rank $n$ degree $d$ semistable bundles which is embedded $\calN\subset \M^\T$, the component of $\M^\T$ of type $(n)$.
\end{example}

\begin{example}	\label{ex111} At the other extreme are the fixed points of type $(1,\dots,1)$. Let $(E,\Phi)\in \M^{s\T}$ be a type $(1,...,1)$ fixed point of the $\T$-action. In other words $E=L_0\oplus\dots \oplus L_{n-1}$ is a direct sum of line bundles and the Higgs field $\Phi$ satisfies $\Phi(L_{i-1})\subset L_{i}K$. Such a Higgs bundle is determined by $L_0$ and the choice of non-zero  $$b_i:=\Phi|_{L_{i-1}}\in H^0(C;L_{i-1}^*L_{i}K).$$ Note that the isomorphism class of such a Higgs bundle only depends on the effective divisors $\delta_i:=\div(b_i)\in C^{[m_i]}$, where $m_i=\deg(L_{i-1}^*L_{i}K)$, and $L_0$. To simplify notation we will choose a divisor $\delta_0$ so that $L_0=\calO(\delta_0)$. We will write $\delta=(\delta_0,\delta_1,\dots,\delta_{n-1})$. We also denote by $\calE_{\d}$ the type $(1,\dots,1)$ Higgs bundle for which $$L_i:=(\delta_0+\delta_1+\dots+\delta_{i-1})K^{-i+1}$$ and  $$b_i:=s_{\delta_i}\in H^0(C;\calO(\delta_i))\cong H^0(C;L_{i-1}^*L_{i}K)$$ the defining section of $\calO(\delta_i)$. 
	
	
	Let $\ell:=\deg(L_0)=\deg(\delta_0)$ and  $m:=(m_1,\dots,m_{n-1})$ where $m_i=\deg(b_i)=\deg(\delta_i)$. Then the ambient component $\calE_{\d}\in F_{\mm}\in \pi_0(\M^\T)$ of the fixed point set $\M^\T$ is isomorphic to  \beq\label{type111fix}F_{\mm}\cong J_\ell(C)\times C^{[m_1]}\times \dots \times C^{[m_{n-1}]}.\eeq Here the Jacobian $J_\ell(C)$ means the moduli space of degree $\ell$ divisors on $C$ modulo principal ones, when we represent its elements by degree $\ell$ divisors $\delta_0$, or equivalently $J_\ell(C)$ can mean the moduli space of degree $\ell$ line bundles when we represent its elements by the line bundle $L_0$. 
\end{example}
\begin{remark} \label{remarkn2} For rank $n=2$ we have just two different types. The type $(2)$ fixed point component is $\calN\subset \M$ the moduli space of semistable bundles with zero Higgs field. The type $(1,1)$ components $F_{\ell,m_1}\cong J_\ell(C)\times C^{[m_1]}$ parametrise Higgs bundles of the form $E=L_0\oplus L_1$ and $$\Phi=\left(\begin{array}{cc}0&0\\b&0\end{array} \right),$$ where $L_0$ and $L_1$ are line bundles of degrees $\ell_0:=\ell$ and $\ell_1$ respectively, $b\in H^0(C;L_0^*L_1K)$ and $m_1=\deg(b)=\ell_1-\ell_0+2g-2$. For stability we need $m_1<2g-2$. When $d=\ell_0+\ell_1$ is fixed then $2\ell_0=d-m_1+(2g-2)$ and we will abbreviate  $F_{m_1}:=F_{\ell,m_1}$. We then have $g-1$ type $(1,1)$ fixed point components $F_{1}$, $F_3$, \dots, $F_{2g-3}$ when $d$ is odd and $g-1$ of them $F_0,\dots, F_{2g-4}$ when $d$ is even.  
\end{remark}
\subsection{Upward flows on $\M$}

We  now describe the upward flows in the Bialynicki-Birula partition of $\M$. We know from Proposition~\ref{bbprop} that the upward flow $W^+_\calE$ for $\calE\in \M^\T$ is isomorphic to a vector space modelled on $T_\calE^+$. The following proposition describes which Higgs bundles belong to $W^+_\calE$. 

\begin{proposition}\label{filtrationexists} Let $\calE=(E,\Phi)\in \M$ then \beq\label{limitdown}\lim_{\lambda\to 0} (E,\lambda\Phi)=(E^\prime,\Phi^\prime)\eeq for some   $(E^\prime,\Phi^\prime)\in \M^{s\T}$ if and only if the following hold \begin{enumerate} \item there exists a  filtration \beq\label{filt}0=E_0\subset E_1\subset \cdots E_{k-1}\subset E_k = E\eeq by subbundles such that \item for all $i$ \beq\label{compup}\Phi(E_i)\subset E_{i+1}K \eeq \item  and the induced maps $$gr_0(\Phi):E_i/E_{i+1}\to E_{i+1}/E_{i+2}K$$ satisfy \beq \label{limithiggs}(E_1/E_0\oplus E_2/E_1\oplus \cdots\oplus E_k/E_{k-1},gr_0(\Phi))\cong (gr(E),gr_0(\Phi))\cong(E^\prime,\Phi^\prime).\eeq \end{enumerate}
	Furthermore the filtration \eqref{filt} with these properties is unique. 
\end{proposition}
\begin{proof} The proof is by a Higgs bundle analogue of the Rees construction. The vector bundle version is discussed in \cite{asok,heinloth,halpern}. 
	
	First let us assume that we have a filtration \eqref{filt} which is compatible with the Higgs field i.e. \eqref{compup} holds for all $i$. We denote by $E^\prime:=gr(E)$ and $\Phi^\prime=gr_0(\Phi)$ and assume that $(E^\prime,\Phi^\prime)\in \M^{s\T}$ is a stable Higgs bundle. 
	
	We can define a vector bundle $\tilde{E}$ over $\C\times C=\spec(\C[x])\times C$ given by the $\Z$-graded $\calO_C[x]$ -module \beq \label{zgrading}\tilde{E}:=\bigoplus_{i\in \Z} x^{-i} E_{-i}\eeq on $C$, where $x$ acts via the embedding $x^{-i} E_{-i}\subset x^{-i+1} E_{-i+1}$. Here $E_k=E$ for $k\geq n$ and $E_k=0$ for $k\leq 0$. Note, that we only have non-positive weights in the $\Z$-graded module \eqref{zgrading}.
	
	We also get the vector bundle homomorphism $\tilde{\Phi}:\tilde{E}\to \tilde{E}\otimes K$ by $\tilde{\Phi}:x^i E_i\to x^{i+1} E_{i+1}\otimes K$ given by $\Phi|_{E_i}$. Furthermore the $\Z$-grading \eqref{zgrading} defines a $\T$-action on $\tilde{E}$ covering the standard action of $\T$ on $\C\cong \spec(\bigoplus_{i\in \Z_{\leq 0}} x^{-i})$. Under this $\T$-action $\lambda\in \T$  sends $\tilde{\Phi}$ to $\lambda^{-1}\tilde{\Phi}$. 
	This gives us a family of Higgs bundles over $C$ parametrized by $\C$. Over $\{1\}\times C$ we get the vector bundle $$\tilde{E}/(x-1)\tilde{E}\cong E$$ back on $C$, while $\tilde{\Phi}$  will induce precisely $\Phi$, yielding our original Higgs bundle $$(\tilde{E},\tilde{\Phi})_{\{1\}\times C}\cong (E,\Phi)$$ 
	The $\T$-equivariant structure on $(\tilde{E},\tilde{\Phi})$ shows that over $\{\lambda\}\times C$ for $0\neq\lambda\in\C$ we get \beq\label{etildelambda}(\tilde{E},\tilde{\Phi})_{\{\lambda\}\times C}\cong (E,\lambda\Phi)\eeq Over   $\{0\}\times C\in \C$ the vector bundle $\tilde{E}$ restricts as  $$\tilde{E}/(x)\tilde{E}\cong gr(E)=E^\prime$$ while $\tilde{\Phi}$ restricts as $gr_0(\Phi)=\Phi^\prime$, giving the Higgs bundle \beq\label{etildezero}(\tilde{E},\tilde{\Phi})_{\{0\}\times C}\cong (E^\prime,\Phi^\prime).\eeq
	
	As $\M$ is a coarse moduli space we have a morphism $f:\C\to \M$ such that \beq \label{fdef}f(\lambda)=\left\{\begin{array}{ll} (E,\lambda\Phi)& \mbox{ when } \lambda\neq 0\\ (E^\prime,\Phi^\prime)&\mbox{ when } \lambda=0 \end{array}\right.\eeq which exactly means \eqref{limitdown}.
	
	For the other direction we assume \eqref{limitdown} and note that as $(E^\prime,\Phi^\prime)$ is a stable Higgs bundle $(E,\lambda\Phi)$ is stable too, because stability is an open condition (c.f. \cite[Lemma 3.7]{simpson1} or \cite[Proposition 3.1]{nitsure}). 
	Let us assume the existence of a map like $f$. Then, as the obstruction in $H^2(\M^s,\T)$ vanishes on $f(\C)$, we get a family of stable Higgs bundles $(\tilde{E},\tilde{\Phi})$ over $\C\times C$ with the properties \eqref{etildelambda} and \eqref{etildezero}. As in \S\ref{equiuni} and \cite[\S 4]{hausel-thaddeus} we can construct a $\T$-equivariant structure on $\tilde{E}$ where $\lambda\in \T$ acts on $\tilde{\Phi}$ as $\lambda^{-1}\tilde{\Phi}$. In other words if $v\in \tilde{E}$ then $$\tilde{\Phi}(\lambda(v))=\lambda(\tilde{\Phi}(v)),$$ where the $\T$-action on $\tilde{E}K$ is induced from the $\T$-action on $\tilde{E}$ and the weight  one action on $K$, i.e. for $\omega\in K$ $\lambda(\omega):=\lambda \omega$; compare it with \eqref{univdiag}. Denote by \beq \label{filtdef}E_i=\{v\in E\cong \tilde{E}|_{\{1\}}| \lim_{\lambda\to 0} \lambda^i \lambda(v)  \mbox{ exists} \}.\eeq Note that \beq\label{limphi}\lambda^i\lambda(\Phi(v))=\lambda^i\lambda(\tilde{\Phi}|_{\{1\}}(v))=\lambda^i\tilde{\Phi}(\lambda(v))=\tilde{\Phi}(\lambda^i\lambda(v)).\eeq So if $v\in E_i$ then $\lim_{\lambda\to 0}\lambda^i\lambda(v)$ exists and so does $\lim_{\lambda\to 0}\lambda^i\lambda(\Phi(v))$. Recalling that $\lambda$ acts on $K$ with weight one, we see that $\Phi(v)\in E_{i+1}K$.  As the construction is locally trivial over $C$, we get an increasing filtration $E_i\subset E_{i+1}\subset E$ by subbundles of $E$. Furthermore $\Phi(E_i)\subset E_{i+1}K$ thus $\Phi$ is compatible with the filtration. For $\mu\in \T$ we compute $$\mu(\lim_{\lambda\to 0}\lambda^i\lambda(v)) =\lim_{\lambda\to 0}\lambda^i\mu(\lambda(v))=\mu^{-i}\lim_{\lambda\to 0}(\mu\lambda)^i(\mu\lambda)(v)=\mu^{-i}\lim_{\lambda\to 0}\lambda^i\lambda(v).$$ This way we get a map
	
	 $$f_i:\begin{array}{ccc}E_i&\to &(\tilde{E}|_{\{0\}})_{-i}\\ v &\mapsto& \lim_{\lambda\to 0} \lambda^{i}\lambda(v) \end{array}$$ from $E_i$ to the weight $-i$ $\T$-isotypical component of the vector bundle with $\T$-action $\tilde{E}|_{\{0\}}$ on $C$. The kernel of this map consists of those $v\in E_i$ for which $$ \lim_{\lambda\to 0}\lambda^i\lambda(v)=0,$$ i.e  there exists a section  $s\in \Gamma(\C,\tilde{E}|_{\C\times \{\pi(v)\}}),$ where $\pi:E\to C$ is the projection of the vector bundle,  $s(0)=0$ and $s(\lambda)=\lambda^i\lambda(v)$. Then we can write $s=\lambda s^\prime$ for another $s^\prime\in \Gamma(\C,\tilde{E}|_{\C\times \{\pi(v)\}})$ and so $ \lim_{\lambda\to 0}\lambda^{i-1}\lambda(v)=s^\prime(0)$ exists showing $\ker(f_i)=E_{i-1}$. Thus $f_i$ induces $$f_i:E_i/E_{i-1}\xhookrightarrow{} (\tilde{E}|_{\{0\}})_i\subset \tilde{E}|_{\{0\}}, $$ and if we put them together we get $f:gr(E)\xhookrightarrow{} \tilde{E}|_{\{0\}}$ an injective bundle map between  vector bundles of the same rank, which thus must be an isomorphism. 
	
	Finally, we note that \eqref{limphi} implies that for $v\in E_i$ we have $$f_{i+1}(\Phi(v))=\tilde{\Phi}|_0(f_i(v))$$ thus $$(gr(E),gr_0(\Phi))\cong(\tilde{E}|_{\{0\}},\tilde{\Phi}|_{\{0\}})$$ completing the proof of existence of the required compatible filtration.
	
	In order to prove the uniqueness of the filtration we start with a compatible filtration \eqref{fullfiltration} as in the first part of the proof. Then we construct the $\T$-equivariant family of stable Higgs bundles $\tilde{E}$ over $\C\times C$ as in \eqref{zgrading}. Then we will show that the original filtration agrees with the one constructed in the second part of the proof with the definition \eqref{filtdef}. First we reformulate \eqref{filtdef} by saying that $E^\prime_{-i}\subset E$ is the subsheaf whose sections $s\in H^0(U,E^\prime_{-i})\subset H^0(U,E)$ on open subsets $U\subset C$ satisfy the condition that the section $\lambda^{-i}\lambda(s)\in H^0(\C^*\times U,\tilde{E})$ extends to a section in $H^0(\C\times U,\tilde{E})$. We note that the $\C[x,x^{-1}]$-module $H^0(\C^*\times U,\tilde{E})$ is the localization of the $\C[x]$-module $\bigoplus_{i\in \Z} x^{-i}E_{-i}$  at $x$, thus $$H^0(\C^*\times U,\tilde{E})\cong \bigoplus_{i\in \Z} x^{-i}\left(\bigoplus_{j\in \N} H^0(U;E_{-i+j})\right). $$ We see that the $\T$-invariant section $\lambda(s)\in H^0(\C^*\times U;\tilde{E})^\T$ agrees with $s=x^0s$ in this description  and so the section $\lambda^{-i}\lambda (s)\in H^0(\C^*\times U;\tilde{E}) $ agrees with $x^{-i}s$ in this description also. Finally, we see that $x^{-i}s$ extends to a section over $\C\times U$ if and only it is in the image of the restriction map $H^0(\C\times U, \tilde{E})\to H^0(\C^*\times U, \tilde{E})$. This happens exactly when $s\in H^0(U;E_{-i})$. Thus $E^\prime_{-i}=E_{-i}$. Therefore the filtration in \eqref{fullfiltration}
	satisfying \eqref{limithiggs} is unique.
	
	The proof of the Proposition is complete. 
\end{proof}

\begin{remark} When $n=2$ and $(E,\Phi)\in \M$ then we can find the filtration in Proposition~\ref{filtrationexists}, giving the limiting Higgs bundle $\lim_{\lambda\to 0}(E,\lambda\Phi)\in \M^\T$ as follows. If $E$ is semistable then the limiting Higgs bundle is simply $(E,0)$. When $E$ is not stable then it has a unique (maximal) destabilizing subbundle $L\subset E$. The compatibility condition \eqref{compup} in this case is vacuous, and thus the limiting Higgs bundle is $\calE_{\delta_0,\delta_1}$ where $L=\calO(\delta_0)$, $\delta_1=\div(b)$ with $b:=\Phi|_L:L\to (E/L)\otimes K$. 
\end{remark}

\begin{remark}\label{bbonpe} If we projectivise the total space of the $\T$-equivariant vector bundle $\tilde{E}$ as  ${\mathbb P}(\tilde{E}),$ then the map ${\mathbb P}(\tilde{E})\to \C$ will be proper, and $\T$-equivariant with respect to the induced $\T$-action on ${\mathbb P}(\tilde{E})$. If follows that this  $\T$-action on ${\mathbb P}(\tilde{E})$ is semi-projective. The fixed point set of the action is $$({\mathbb P}(\tilde{E}))^\T=\cup_i {\mathbb P}((\tilde{E}|_{\{0\}})_i).$$ The Bialynicki-Birula decomposition ${\mathbb P}(\tilde{E})=\cup_i W^+_i$ gives   $${\mathbb P}(E_i)\setminus {\mathbb P}(E_{i-1}) = W^+_{-i}\cap {\mathbb P}(\tilde{E})|_{\{1\}},$$ which follows from  the  proof of Proposition~\ref{filtrationexists}. Thus, amusingly, the filtration $E_i\subset E_{i+1}\subset E$ in Proposition~\ref{filtrationexists} describing the Bialynicki-Birula upward flows on $\M$, can be recovered from the Bialynicki-Birula decomposition on ${\mathbb P}(\tilde{E})$.
	
	In fact, one can construct the Bialynicki-Birula partition on the total space of the projectivised equivariant universal bundle (which always exists over the stable locus, and extends over the whole of $\M$ using Simpson's framed moduli Higgs moduli space \cite[Theorem 4.10]{simpson1}, see \eqref{grassmann} for $k=1$) on $\M\times C$ which will contain the information on the filtration in Proposition~\ref{filtrationexists} which determines the limiting  Higgs bundle. 
\end{remark}

\begin{remark}   In a differential geometric language Proposition~\ref{filtrationexists} appeared  in \cite[Proposition 4.2]{collier-wentworth}. Also a closely related partition of the de Rham moduli space of holomorphic connections $\M_{\rm DR}$ appeared in \cite{simpson}. In fact the algorithm in \cite{simpson} can be used verbatim to find the filtration in Proposition~\ref{filtrationexists} producing the limiting Higgs bundle. 
\end{remark}

\begin{remark} \label{hitchinsection} Let us recall the construction of a section of the Hitchin map from \cite{hitchin-lie}.  Let ${{a}}=(a_1,\dots,a_n)\in H^0(C;K)\oplus \dots \oplus H^0(C;K^n)=\calA$ be a point in the Hitchin base. For a line bundle $L$ on $C$ let 
	$$E_L:=L\oplus LK^{-1}\oplus\dots\oplus LK^{1-n}.$$ Take the Higgs field $\Phi_{{a}}:E_L\to E_LK$ given by the companion matrix: 	\beq\label{companion}\Phi_{{a}}:=\left(\begin{array}{ccccc}  0 &0& \dots &0& -a_n\\ 1 & 0 & \dots &0&-a_{n-1}\\ 0 & 1 & \dots &0&-a_{n-2}\\ \vdots & \vdots & \ddots &\vdots &\vdots \\ 0 & 0& \dots &1 & -a_1
	\end{array}\right).\eeq  Denote the Higgs bundle $\calE_{L,{a}}:=(E_L,\Phi_{{a}})$. Define $E_i:=L\oplus LK^{-1}\oplus\dots\oplus LK^{-i+1}$. The filtration $0\subset E_1\dots \subset E_n=E_L$ satisfies $\Phi_{{a}}(E_i)\subset E_{i+1}K$ and the associated graded is $(gr(E_L),gr_0(\Phi_{ a}))=(E_L,\Phi_0)=\calE_{L,0}$. It is straightforward to check this is stable. We call $\calE_{L,0}$ a {\em uniformising Higgs bundle}\footnote{The terminology is motivated by the $n=2$ case when the $\PGL_2$ Higgs bundle associated to our uniformising Higgs bundle corresponds by  non-abelian Hodge theory to the representation $\pi_1(C)\to \PGL_2(\R)\subset \PGL_2(\C)$ giving the uniformising hyperbolic metric on $C$.}.   
	Now Proposition~\ref{filtrationexists} applies and shows that $\calE_{L,a}$ is stable too and $\lim_{\lambda\to 0} (E_L,\lambda \Phi_{ a})=(E_L,\Phi_0)=\calE_{L,0}$. Thus we have a map $s:\calA\to \M$ given by $a\mapsto \calE_a$. By construction  $h(\calE_{L,a})=a$ thus $s$ is a section of the Hitchin map $h:\M\to \calA$. Thus we see that $\calA\cong s(\calA)\subset W^+_{\calE_{L,0}}$ is a subvariety of an affine variety which is isomorphic to a vector space of the same 
	dimension, therefore $s(\calA)=W^+_{\calE_{L,0}}$, and so the upward flow $W^+_{\calE_{L,0}}$ is just a Hitchin section. 
	
	In particular, this implies that if $\calE\in W^+_{\calE_{L,0}}$, i.e. it has a full filtration as in Proposition~\ref{filtrationexists} such that the associated graded is isomorphic to $\calE_{L,0}$, then it has to be isomorphic to a Higgs bundle  of the form $(E_L,\Phi_a)$. For example, the underlying bundle must  split as a direct  sum of line bundles. 
	
	If we choose  $L=\calO_C$ then we have \beq\label{uniform}\calE_0:=\calE_{{\calO_C},0}=E_{\calO_C}\stackrel{\Phi_0}{\to}E_{\calO_C}\otimes K \eeq  the {\em canonical uniformising Higgs bundle}. This will give a {\em canonical section} $W^+_{\calE_0}\subset \M$ of the Hitchin map, corresponding to the structure sheaves of the spectral curves under  the BNR correspondence, see \eqref{canonicalBNR}. 
	

\end{remark}

\subsection{Downward flows on $\M$}

Recall that the core of $\M$ is defined as $\calC:=\coprod_\alpha W^-_\alpha = \coprod_F W^-_F$. From Proposition~\ref{lagrangian} we know that $W^-_F\subset \M^s$ is a Lagrangian subvariety. Thus the core $\calC$ is Lagrangian at its smooth points. 

By recalling the $\T$-equivariant map $h:\M\to\calA$ where $\T$ acts on $\calA$ with positive weights we see that the core of $\calA$ is the origin $0\in \calA$, and so $\calC\subset h^{-1}(0)$. On the other hand $h$ is proper, therefore $h^{-1}(0)$ is projective, thus $h^{-1}(0)\subset \calC$. The zero fibre of the Hitchin map $h^{-1}(0)$ is called the {\em nilpotent cone}, because $h(E,\Phi)=0$ if and only if $\Phi$ is nilpotent. We will think of the nilpotent cone $N:=h^{-1}(0)$ as a subscheme  of $\M$. In particular, we will see later that most of its components are non-reduced. The argument above then implies 

\begin{proposition}\label{nilpcore} The reduced scheme of the nilpotent cone coincides with the core: $N_{red}=\calC$. 
\end{proposition}
\begin{remark}
	The core is a Lagrangian subvariety. Thus Proposition~\ref{nilpcore} shows that the nilpotent cone is Lagrangian, which is a result of Laumon \cite[Theorem 3]{laumon}.
\end{remark}

We finish with the analogue of Proposition~\ref{filtrationexists} for the downward flow.

\begin{proposition}\label{downwardfiltration} Let $\calE=(E,\Phi)\in \M$ then \bes\label{limitup}\lim_{\lambda \to \infty} (E,\lambda\Phi)=(E^\prime,\Phi^\prime)\ees for some   $(E^\prime,\Phi^\prime)\in \M^{s\T}$ if and only if the following hold 
	\begin{enumerate} \item there exists a filtration \beq\label{filtdown}0=E_0\subset E_1\subset \cdots E_{k-1}\subset E_k = E\eeq by subbundles \item such that for all $i$ \bes\label{compdown}\Phi(E_{i+1})\subset E_{i}K\ees \item and the induced maps $gr_\infty(\Phi):E_i/E_{i-1}\to E_{i-1}/E_{i-2}K$ satisfy $$(E_1/E_0\oplus \cdots\oplus E_{k-1}/E_{k-2}\oplus  E_k/E_{k-1},gr_\infty(\Phi))\cong (gr(E),gr_\infty(\Phi))\cong(E^\prime,\Phi^\prime).$$
		\end{enumerate}
	
	Additionally, the fibration \eqref{filtdown} with these properties is unique. 
\end{proposition}

\begin{proof} The proof is an appropriate modification of the proof of Proposition~\ref{filtrationexists}.  
	Just as in Remark~\ref{bbonpe} we can think of the filtration  \eqref{filtdown} on the Higgs bundle induced from the downward flows on the projectivised universal bundle. 
\end{proof}
 
	\section{Very stable Higgs bundles}
\label{verystable}

\subsection{Definition and basic properties} 
Drinfeld and Laumon in \cite[Definition 3.4]{laumon} call a vector bundle $E$ on $C$ {\em very stable} if the only nilpotent Higgs field $\Phi
\in H^0(\End(E)\otimes K)$ is the trivial one.  \cite[Proposition 3.5]{laumon} proves that a rank $n$ very stable bundle is stable, and that very stable bundles form an open dense subset of the moduli space of stable  bundles. 

\begin{definition} Let $\calE\in \M^{s \T}$ be a stable $\T$-invariant Higgs bundle. We call $\calE$ a {\em very stable Higgs bundle} if and only if  the only nilpotent Higgs bundle in $W_\calE^+$ is $\calE$.  On the other hand we call $\calE$ {\em wobbly}, when it is not very stable. 
\end{definition}

\begin{remark} A vector  bundle which is very stable is stable, as shown in \cite[Proposition 3.5]{laumon}. For Higgs bundles we have to impose it as the definition uses upward flows in the semistable moduli space $\M$. We could also define very semistable bundles for strictly semistable bundles in the same way as above. These could also be interesting to study. 
	\end{remark}

In this language \cite[Proposition 3.5]{laumon} implies \begin{theorem} \label{laumon} There exist very stable type $(n)$ Higgs bundles $\calE\in \M^{s\T}$. In fact, they form an open dense subset of $\calN\in \pi_0(\M^{s\T})$. 
\end{theorem}
The following was proved for very stable bundles $E$ (in our context Higgs bundles with zero Higgs field) in  \cite[Theorem 1.1]{pauly-peon-nieto}.

\begin{lemma}\label{transversal} The stable Higgs bundle $\calE\in \M^{s\T}$ is very stable if and only if the upward flow $W^+_\calE\subset\M$ is closed.
\end{lemma}

\begin{proof}This is Proposition~\ref{upwardclose} applied to the semi-projective variety $\M$.  
\end{proof}

\begin{remark} \label{hsvs}
	Recall from Remark~\ref{hitchinsection} that the upward flow    $W^+_{\calE_{L,0}}\subset\M$ of the uniformising Higgs bundle $\calE_{L,0}$  is  a Hitchin section $h:W^+_{\calE_{L,0}} \cong s(\calA)$. Thus $W^+_{\calE_{L,0}}\cap h^{-1}(0)=\{\calE_{L,0}\}$. Therefore  the  Higgs bundle $\calE_{L,0}$ is very stable. In particular, by the lemma above,  $W^+_{\calE_{{L,0}}}\subset\M$ is closed. This also follows from the fact that it is a section of the Hitchin map, which  was already pointed out in \cite[p. 454]{hitchin-lie}. 
\end{remark}

\begin{lemma} \label{locfree} When $\calE\in \M^{\T}$ is very stable then $h:W^+_\calE\to \calA$ is finite, flat, surjective and generically \'etale. In particular, when $\calE\in \M^{\T}$ is very stable then $h_*(\calO_{W^+_\calE})$ is locally free.
\end{lemma}
\begin{proof} As $W^+_\calE\subset \M$ is closed by Lemma~\ref{transversal} the restriction of the Hitchin map $h:W^+_\calE\to \calA$ is proper. Furthermore since $h:W^+_\calE\to \calA$ is a map between  affine varieties (cf. Proposition~\ref{bbprop}), it is quasi-finite. As a proper and quasi-finite map it is  finite (c.f \cite[Theorem 8.11.1]{ega}). Finally, it is a map between smooth equidimensional varieties and so by miracle flatness (c.f. \cite[Corollary 23.1]{matsumura}) $h$ is flat. Thus $h$ is  finite and flat  thus $h_*(\calO_{W^+_\calE})$ is locally free.
	
	Finally, as $h:W^+_\calE\to \calA$ is finite and between equidimensional varieties, it is surjective, and thus generically smooth, and so generically \'etale. 
\end{proof}

\subsection{ Hecke transformations}
\label{heckeverystable}

For a point $c\in C$ recall the exact sequence of sheaves \beq \label{ses}0\to \calO_C(-c)\stackrel{s_c}{\to} \calO_C\to \calO_c\to 0.\eeq  Let $E\to C$ be a vector bundle. By abuse of notation we will also denote by $E$ its locally-free sheaf of sections.   Tensoring \eqref{ses} with $E$ gives \beq\label{tensored} 0\to E(-c){\to} E\stackrel{f}{\to} (E_c)_c \to 0.\eeq Here $E_c$ denotes the fibre of the vector bundle $E$ at $c$ and for any vector space $U$ we denote $U_c:=U\otimes \calO_c.$ In particular, $(E_c)_c=E\otimes \calO_c$. If now $V<E_c$ then $V_c\subset (E_c)_c$ is a subsheaf and if $W:=E_c/V$ then $W_c$ is a quotient sheaf of $(E_c)_c$. Denoting by  $\pi_V:(E_c)_c\to W_c$ the quotient map and  $f_V=f \circ \pi_V$  we make the following 
\begin{definition} Let $E$ be a vector bundle on the curve $C$.  Let $V<E_c$ be a subspace at a point $c\in C$. The {\em Hecke transform} of $E$ at $V$ is defined to be $E_V:=\ker f_V$. \end{definition} Equivalently, the following   short exact sequence of sheaves defines the Hecke transformed vector bundle: $$0\to E_V{\to} E\stackrel{f_V}{\to} W_c \to 0.$$ As a subsheaf of a torsion free sheaf, $E_V$ is torsion free and so a locally free sheaf, i.e. the sheaf of sections of a vector bundle which we also denote by $E_V$ or sometimes $E^\prime=E_V$. 

\begin{remark} \label{remarkV}By construction $E(-c)\subset E_V$ is a subsheaf and fits into a short exact sequence $$0\to E(-c){\to} E^\prime{\to} V_c \to 0.$$ This way we get a surjective map $E^\prime_c\to V$. Denote by $V^\prime < E^\prime_c$ its kernel. 
	Then we see that $E^\prime_{V^\prime}=E(-c)$. 
\end{remark}

\begin{example}\label{heckebundle} When $V=E_c$ then the exact sequence \eqref{tensored} shows that $E_V=E(-c)$. More generally, if a vector bundle on a curve $E\oplus F$ is a direct sum of two subbundles  and we  choose the subspace $E_c<(E\oplus F)_c$ then the Hecke transfrom of $E\oplus F$ at $E_c$ satisfies $$(E\oplus F)_{E_c}\cong E(-c)
\oplus F.$$ This can be seen by adding the trivial sequence $0\rightarrow F\stackrel{\cong}{\rightarrow}F\to 0 \to 0$ to \eqref{tensored}.	\end{example}

Let now $(E,\Phi)$ be a Higgs bundle and let $V<E_c$ be a $\Phi_c$-invariant subspace. Consider the quotient $W:=E_c/V$. Denote by $\overline{\Phi}_c$ the morphism $W_c\to W_c\otimes K$ induced by $\Phi$.   This data induces a unique Higgs bundle $(E^\prime,\Phi^\prime)$ making the following diagram commutative.
\beq \label{HeckeHiggs} \begin{array}{ccccccc} 0 \to & E^\prime& {\to} &E& \stackrel{f_V}{\to} & W_c & \to 0 
	\\ & \Phi^\prime \downarrow &&\Phi \downarrow&& \overline{\Phi}_c\downarrow & \\ 0 \to & E^\prime K&\to &EK&\to& W_c K&\to 0 \end{array}.\eeq 

\begin{definition} Let $(E,\Phi)$ be a Higgs bundle on the curve $C$ and $V< E_c$ a $\Phi_c$-invariant subspace at $c\in C$. The {\em Hecke transform} of $(E,\Phi)$ at $V< E_c$ is the unique Higgs bundle $(E^\prime,\Phi^\prime)$ making the diagram \eqref{HeckeHiggs} commutative. Sometimes we will denote it $\calH_V(E,\Phi)=(E^\prime,V^\prime)$ or by $(E_V,\Phi_V)=(E^\prime,\Phi^\prime)$. 
\end{definition}

\begin{remark} \label{V'Higgs}As  in Remark~\ref{remarkV}, the Hecke transformation will produce a subspace $V^\prime<E^\prime_c$ such that the Hecke transform of $E^\prime$ at $V^\prime$ is $E(-c)$. We note that the Hecke transformed Higgs field $\Phi^{\prime\prime}:E(-c)\to E(-c)K$ as  a section of $\End(E(-c))K=\End(E)K$ agrees with $\Phi\in \End(E)K$ away from $c$, and so agrees everywhere. Thus the Hecke transform of $(E^\prime,\Phi^\prime)$ at $V^\prime$ is the Higgs bundle $(E(-c),\Phi)$. 
\end{remark}

\begin{remark}\label{alonghitchin} We also note that as $\Phi^\prime$ and $\Phi$ agree away from $c$ so  their characteristic polynomial will also agree away from $c$, therefore $h(\Phi^\prime)=h(\Phi)\in \oplus_{i=1}^n H^0(C;K^i)$. Therefore the Hecke transformation is along the fibers of the Hitchin fibration. In the case of a point $a\in {\mathcal A}$ for which  the spectral curve $C_a$ is smooth, the fibre is isomorphic to the Jacobian of $C_a$. At a generic point $c\in C$ the subspace $V<E_c$ is a sum of  eigenspaces of $\Phi_c$. Since $C_a$ is the curve of eigenvalues this is an effective  divisor $D$. In the notation of Section~\ref{intersection} the Hecke transform takes the line bundle $U$ on $C_a$ corresponding to $(E,\Phi)$  to the line bundle $U(-D)$ corresponding to $(E^\prime,\Phi^\prime)$. 
\end{remark}

\begin{example}\label{invarianthecke} Assume that the Higgs bunde $(E,\Phi)$ is such that $E=E_1\oplus \dots \oplus E_k$, a direct sum of subbundles, and $\Phi(E_i)\subset E_{i+1}K$. Then for any $1\leq j \leq k-1$ we have  $V:=(E_{j+1}\oplus \dots \oplus E_k)_c<E_c$ is $\Phi_c$ invariant. Then we can determine the Hecke transform $(E^\prime,\Phi^\prime)$ at $V$ by definition as $E^\prime=E_1^\prime \oplus \dots \oplus E^\prime_k$ such that $E_i^\prime=E_i(-c)$ unless $i>j$ when $E_i^\prime=E_i$ from Example~\ref{heckebundle}.  The modified Higgs field then satisfies $$\Phi^\prime|_{E^\prime_i}=\Phi|_{E_i}\in H^0(C;E_i^{\prime*}E^\prime_{i+1}K)\cong H^0(C;E_i^*E_{i+1}K)$$ for all $i$ except $i=j$ when $$\Phi^\prime|_{E_{j}(-c)}=\Phi|_{E_j}s_c\in H^0(C;E_i^*(c)E_{i+1}K).$$ In particular, $(\Phi^\prime|_{E_{j}(-c)})$ vanishes at $c$. 
Furthermore $V^\prime$ of Remark~\ref{V'Higgs} will be $V^\prime=E^\prime_1\oplus \dots \oplus E^\prime_{j}|_c<E^\prime_c$, so that the $\calH_{V^\prime}(E^\prime,\Phi^\prime)=(E(-c),\Phi)$ as expected.

Doing the Hecke transformation backwards we get that if the Higgs bundle $(E,\Phi)$ has a compatible decomposition $E=E_1\oplus\dots\oplus E_k$ such that $\Phi(E_i)\subset E_{i+1}K$ and $\Phi_c$ is zero on $(E_j)_c$, then $V=(E_1\oplus \dots \oplus E_{j})|_c<E_c$ is $\Phi_c$-invariant.  Then the Hecke transform
$\calH_{V}(E,\Phi)$ has the chain form\footnote{Where arrows represent morphisms twisted by $K$.}
$$E_1\stackrel{\Phi_1}{\longrightarrow}{E_2}\stackrel{\Phi_2}{\longrightarrow}\cdots \stackrel{\Phi_{j-2}}{\longrightarrow}E_{j-1} \stackrel{\Phi_{j-1}/s_c}{\longrightarrow} E_j(-c)\stackrel{\Phi_{j}}{\longrightarrow}\cdots \stackrel{\Phi_{k}}{\longrightarrow}{E_k(-c)},$$ where $\Phi_i:=\Phi|_{E_i}:E_i\to E_{i+1}K$. 

Note that in both cases above the Higgs bundles are $\T$-fixed and so are the Hecke transformed Higgs bundles (when all of them are stable). However if the $\Phi_c$-invariant subspace is not a direct sum of the $E_i|_c$'s then  the Hecke transform will not be $\T$-fixed. Such examples will play a crucial rule in the only if part of the proof of Theorem~\ref{mainverystable}. 

	\end{example}



\subsubsection{The Hecke transform of a full compatible filtration}

Here we investigate what happens to a compatible full filtration on a Higgs bundle,  when we perform a Hecke transformation at a point $c\in C$ where the Higgs field on the fibre of the graded Higgs bundle at $c$ is regular nilpotent. This is equivalent with $b(c)\neq 0$ using the notation \eqref{be} below. 

Let a rank $n$ Higgs bundle $(E,\Phi)$ carry a full filtration \beq\label{fullfiltration}0=E_0\subset E_1\subset\dots\subset E_n=E\eeq by $\rank(E_i)=i$ subbundles, such that \beq\label{compatible}\Phi(E_i)\subset E_{i+1}\otimes K.\eeq Then we have the induced maps \beq\label{bi}b_i=:E_i/E_{i-1}\to E_{i+1}/E_i\otimes K\eeq between line bundles, and the composition \beq \label{be}b=b_{n-1}\circ \dots \circ b_1:E_1\to E/{E_{n-1}}\otimes K^{n-1}.\eeq Suppose that $b$ does not vanish at $c\in C$. This means that for all $i$ the map  $\Phi_c: E_{i}|_c\to E_{i+1}|_c\otimes K|_c$ does not preserve $E_i|_c$.  To perform the Hecke transformation we choose a $\Phi_c$-invariant subspace $V<E_c$ of $\dim(V)=k< n$.  We need the following linear algebra
\begin{lemma}\label{linear} Let $0=U_0<U_1<\dots U_n<U$ be a full flag of subspaces of an $n$-dimensional complex vector space $U$. Furthermore let $A$ be a linear transformation $A:U\to U$ such that $A(U_i)<U_{i+1}$ but $U_i$ is not $A$-invariant. If $V< U$ of $\dim(V)=k< n$ is $A$-invariant, then 
	$\dim (V\cap U_i)= \max(0,k+i-n)$.
\end{lemma}
\begin{proof} By assumption $U_i+A(U_i)= U_{i+1}$. Thus $V+U_i=V+U_{i+1}$ if and only if $V+U_i$ is invariant under $A$. On the other hand $V+U_i=V+U_{i+1}$ in turn implies $V+U_i=V+U_n=U$. 
	Thus we see that the chain of subspaces
	$$V=V+U_0\leq V+U_1\leq \dots \leq V+U_{n-1}\leq V+U_n=U$$ must be strictly increasing until it reaches the whole ambient space $U$. Therefore $\dim(V+U_i)=\min(k+i,n)$, which implies the result.  
\end{proof}

The $\Phi_c$-invariant subspace $V<E_c$ induces a map of sheaves $f_V:E\to W_c$, where $W:=E_c/V$.   This induces the Hecke transformed Higgs bundle $(E^\prime,\Phi^\prime)$:
\bes \begin{array}{ccccccc} 0 \to & E^\prime & \to &E& \stackrel{f_V}{\to} & W_c & \to 0 
	\\ & \Phi^\prime \downarrow &&\Phi \downarrow&& \overline{\Phi}_c\downarrow & \\ 0 \to & E^\prime K&\to &EK&\to& W_c K&\to 0 \end{array}.\ees 
Each subbundle $E_i$ will also induce a Hecke transformed line bundle with respect to the subspace $V_i:=E_i|_c\cap V< E_i|_c$. Denoting $W^i:=E_i|_c/V_i$ we get a new vector bundle $E_i^\prime$ from the kernel of the map $f_{V_i}:E_i\to W^i_c$. In other words we have the top row of the following commutative diagram \beq \label{commutative} \begin{array}{ccccccc} 0\to & E^\prime_i&\to &E_i& \stackrel{f_{V_i}}{\to} &W^i_c&\to 0 \\ &\rotatebox[origin=c]{-90}{$\hookrightarrow$} &&\rotatebox[origin=c]{-90}{$\hookrightarrow$} &&\rotatebox[origin=c]{-90}{$\hookrightarrow$}& \\   0\to & E_{i+1}^\prime&\to &E_{i+1}& \stackrel{f_{V_{i+1}}}{\to} &W^{i+1}_c&\to 0 \\ &\rotatebox[origin=c]{-90}{$\twoheadrightarrow$} &&\rotatebox[origin=c]{-90}{$\twoheadrightarrow$} &&\rotatebox[origin=c]{-90}{$\twoheadrightarrow$}& \\   0\to & E_{i+1}^\prime/E^\prime_i&\to &E_{i+1}/E_i& \stackrel{\overline{f}_{V_{i+1}}}{\to} &(W^{i+1}/W^i)_c&\to 0
\end{array}\eeq Here the embedding  $E_i^\prime\hookrightarrow E_{i+1}^\prime$ is uniquely induced from the rest of the first two top rows of the diagram. The quotient maps connecting the second and third rows of the diagram are the ones making the columns  into short exact sequences. The last row, and in particular the map  $\overline{f}_{V_i}$,  is then  uniquely induced from the rest of the diagram to make it commutative. 

From Lemma~\ref{linear} we have  \beq\label{dimwi}\dim(W^{i+1}/W^i) =\left\{
\begin{array}{cc}1 & \mbox{ for } i< n-k \\ 0 & \mbox{ for } i\geq n-k \end{array}
\right.\eeq
Thus from the bottom row of \eqref{commutative} we obtain \beq \label{'quotients}E_{i+1}^\prime/E^\prime_i\cong\left\{
\begin{array}{cc}(E_{i+1}/E_i)(-c) & \mbox{ for } i< n-k \\ (E_{i+1}/E_i) & \mbox{ for } i\geq n-k \end{array}
\right.\eeq In particular, $E_{i+1}^\prime/E^\prime_i$ is a line bundle, and so $E^\prime_i\subset E^\prime_{i+1}$ is a subbundle. Thus we get a full filtration of $E^\prime$ by  subbundles: \beq \label{'filtration} E^\prime_0=0\subset E^\prime_1\subset\dots\subset E^\prime_n=E
\eeq

By construction it is compatible with $\Phi^\prime$ i.e.  \beq \label{'compatible}\Phi^\prime(E^\prime_i)\subset E^\prime_{i+1}K.\eeq Finally we will determine the maps $b^\prime_i:E^\prime_i/E^\prime_{i-1}\to E^\prime_{i+1}/E^\prime_i\otimes K$ induced from $\Phi^\prime$. We have the commutative diagram

\bes \begin{array}{ccccccc}   0\to & E_{i}^\prime/E^\prime_{i-1}&\to &E_{i}/E_{i-1}& \stackrel{\overline{f}_{V_{i}}}{\to} &(W^{i}/W^{i-1})_c&\to 0 \\ &b_i^\prime \downarrow &&b_i \downarrow &&\overline{b}_i\downarrow & \\   0\to & (E_{i+1}^\prime/E^\prime_i)\otimes K&\to &(E_{i+1}/E_i)\otimes K& \stackrel{\overline{f}_{V_{i+1}}}{\to} &(W^{i+1}/W^i)_c\otimes K&\to 0
\end{array}.\ees Here $\overline{b}_i$ is induced from $\Phi_c$. When $i>n-k$ then from \eqref{dimwi} we see that $\overline{b}_i$ maps between zero-dimensional spaces, which shows that \beq \label{bibip}b_i^\prime=b_i\in H^0(C;(E_{i}^\prime/E^\prime_{i-1})^*(E_{i+1}^\prime/E^\prime_i)K)=H^0(C;(E_{i}/E_{i-1})^*(E_{i+1}/E_i)K).
\eeq
When $i<n-k$ then $\overline{b}_i$ is induced from an isomorphism of $1$-dimensional vector spaces, and we get again the agreement \eqref{bibip}. 

Furthermore, when $i=n-k$ then $W^{i}/W^{i-1}$ is one-dimensional and $W^{i+1}/W^i=0$. This implies that $b_i^\prime=b_i s_c$ where $s_c\in H^0(C;\calO(c))\cong H^0(C;\Hom(\calO(-c), \calO))$ is the defining \eqref{ses} section vanishing at $c$. 

Finally by tracing the kernel of the induced maps on the fibres at $c$ in the diagram \eqref{commutative} and using \eqref{dimwi} we see that $V^\prime=(E^\prime_{n-k})_c<E^\prime_c$. Thus $\calH_{V^\prime}(E^\prime,\Phi^\prime)=(E(-c),\Phi )$.

We proved the following

\begin{proposition} \label{modifiedfilt}Let a Higgs bundle $(E,\Phi)$ carry a full \eqref{fullfiltration} and compatible \eqref{compatible} filtration, such that $b$ \eqref{be} does not vanish at $c$. Let $V<E_c$ be a  $k$-dimensional and $\Phi_c$-invariant subspace. Then the Hecke transformed Higgs bundle $\calH_V(E,\Phi)=(E^\prime,\Phi^\prime)$ carries an induced full filtration \eqref{'filtration} which is compatible with $\Phi^\prime$ \eqref{'compatible}. Furthermore we have \eqref{'quotients} and $b^\prime_i=b_i$ unless $i=n-k$ when $b^\prime_{n-k}=b_{n-k}s_c$ acquires a simple zero at $c$. Finally, the subspace of Remark~\ref{V'Higgs} is  $V^\prime=(E^\prime_{n-k})_c<E^\prime_c$. Thus $\calH_{V^\prime}(E^\prime,\Phi^\prime)=(E(-c),\Phi )$. 
\end{proposition}

\subsubsection{Very stable Higgs bundles via Hecke transformations}

For $\delta=(\delta_0,\delta_1,\dots,\delta_{n-1})\in J_\ell(C)\times C^{[m_1]}\times \dots \times C^{[m_{n-1}]}$ let $\calE_{\d}\in \M^{s\T}$ be a type $(1,...,1)$ fixed point of the $\T$-action constructed in Example~\ref{ex111}.

\begin{theorem} \label{mainverystable} Let $\calE_{\d} \in \M^{s\T}$ be a type $(1,...,1)$ fixed point of the $\T$-action. Then $\calE_{\d}$ is very stable if and only if $b=b_1\circ\dots \circ b_{n-1}\in H^0(C;L_0^*L_{n-1}K)$ has no repeated zero, i.e. if and only if the effective divisor $\div(b)=\delta_1+\dots+\delta_{n-1}$ is reduced. 
\end{theorem}
\begin{proof} First we prove the if part. The proof is by induction on $\deg(b)$. When $b$ has no zeroes, then $\calE_{\d}\cong \calE_{L_0,0}$ are the uniformising Higgs bundles from Remark~\ref{hitchinsection}. Thus $W^+_{\calE_{\d}}$ is a Hitchin section. In particular $\calE_{\d}$ is its intersection with the nilpotent cone $h^{-1}(0)$ (c.f. Remark~\ref{hsvs}). Thus $\calE_{\d}$ is very stable when $\deg(b)=0$. 
	
	Let $\calE_{\d} \in \M^{s\T}$ be such that $b$ has no repeated zero, but at least one zero say $c\in C$ with $b(c)=0$. Assume that it is $b_k(c)=0$. Then $V:=(L_0\oplus \dots \oplus L_{k-1})_c<E_c$ is $\Phi_c$ invariant. We see from Example~\ref{invarianthecke} that $\calH_{V}(\calE_{\d})=\calE_{\d^\prime}$ where $\d^\prime_i=\d_i$ i.e. $b^\prime_i=b_i$ except if  $i=k$ when  $b^\prime_k=b_k/s_c$ i.e. $\d^\prime_k=\d_k-c$. In particular, $b^\prime$ does not vanish at $c$. We will make frequent use of the following
	\begin{lemma}\label{stablemma} Let $\calE_{\d} \in \M^{s\T}$ such that $b_k(c)=0$. Let $V:=(L_0\oplus \dots \oplus L_{k-1})_c<E_c$. Then  $\calE_{\d^\prime}:=\calH_{V}(\calE_{\d})$ has the chain form
$$L_0\stackrel{b_1}{\rightarrow}{L_1}\stackrel{b_2}{\rightarrow}\cdots \stackrel{b_{k-1}}{\rightarrow}L_{k-1} \stackrel{b_k/s_c}{\rightarrow} L_k(-c)\stackrel{b_{k+1}}{\rightarrow}\cdots \stackrel{b_{n-1}}{\rightarrow}{L_{n-1}(-c)}$$ and it is a stable Higgs bundle.
	\end{lemma}
	\begin{proof} As by assumption $\calE_{\d}$ is stable all proper $\Phi$-invariant subbundles have slope less than the slope of $E$. As the  only $\Phi$-invariant subbundles are $L_i\oplus \dots \oplus L_{n-1}\subset E$, we have for every $0\leq i \leq n-1$ the following inequality  $$\frac{\ell_i+\dots+\ell_{n-1}}{n-i}<\frac{\ell_0+\dots+\ell_{n-1}}{n}.$$ To see that $\calE_{\d^\prime}$ is stable we check that \bes \frac{\ell^\prime_i+\dots+\ell^\prime_{n-1}}{n-i}=\frac{\ell_i+\dots+\ell_{n-1}-\min(n-i,n-k)}{n-i}&<&\frac{\ell_0+\dots+\ell_{n-1}}{n}-\frac{\min(n-i,n-k)}{n-i}\\
		< \frac{\ell^\prime_0+\dots+\ell^\prime_{n-1}+n-k}{n}-\frac{\min(n-i,n-k)}{n-i}	
		&<& \frac{\ell^\prime_0+\dots+\ell^\prime_{n-1}}{n} \ees 
		
		because $$\frac{n-k}{n}\leq \frac{\min(n-i,n-k)}{n-i}.$$
	\end{proof}
	
	Assume now that we have a nilpotent Higgs bundle $\calE=(E,\Phi)\in W^+_{\calE_{\d}}$. From Proposition~\ref{filtrationexists} let $E_1\subset\dots\subset E_n=E$ be the $\Phi$-invariant filtration  inducing $\lim_{\lambda\to 0}\lambda\cdot \calE=\calE_{\d}$. Let $V:=(E_{k})_c<E_c$, which by the assumption of $b_k(c)=0$ is $\Phi_c$-invariant. As we saw in the previous subsection $\calH_V(E,\Phi)=(E^\prime,\Phi^\prime)$ is nilpotent in the upward flow of $\calH_{V_{1k}}(\calE_{\d})=\calE_{\d^\prime}$, because $\calE_{\d^\prime}$ is stable by Lemma~\ref{stablemma}. But $\deg(b^\prime)=\deg(b)-1$ and so by induction we can assume that $\calE_{\d^\prime}$ is very stable and hence the nilpotent $(E^\prime,\Phi^\prime)$ in its upward flow should be isomorphic to it: $(E^\prime,\Phi^\prime)\cong \calE_{\d^\prime}$. Now $b^\prime(c)\neq 0$ therefore $\Phi^\prime_c$ is regular nilpotent. Hence there is a unique $\Phi^\prime_c$-invariant $(n-k)$-dimensional subspace of $E^\prime$, namely $V^\prime_{k,n-1}=(L_{k})_c\oplus \dots\oplus (L_{n-1})_c$. The Hecke transform of $(E^\prime,\Phi^\prime)\cong \calE_{\d^\prime}$ at this subspace gives $(E(-c),\Phi)\cong\calE_{\delta-c}$, where $$\delta-c:=(\delta_0-c,\delta_1,\dots,\delta_{n-1}).$$ This shows that $\calE_{\d}$ is also very stable. This proves the ``if" part of the theorem.
	
	For the only if part, when $b$ has repeated zeroes, the strategy will be to produce a {\em Hecke curve} \cite{NR} - a one-parameter family of Higgs bundles produced as Hecke transforms of a fixed Higgs bundle -   in the intersection of the nilpotent cone and the upward flow $W^+_{\calE_{\d}}$ originating at $\calE_{\d}$. 
	
	Take a multiple zero $c$ of $b$. Assume that $b_k(c)=0$. Then $V:=(L_0\oplus \dots \oplus L_{k-1})_c<E_c$ is $\Phi_c$-invariant. The Hecke transformed Higgs bundle from Example~\ref{invarianthecke} $\calH_{V}(\calE_{\d})=\calE_{\d^\prime}=(E^\prime,\Phi^\prime)$ has  $b^\prime(c)=0$ otherwise $\calH_{V^\prime}(\calE_{\d^\prime})$ would have $b(c)=0$ a simple zero from Proposition~\ref{modifiedfilt}. 
	
	Thus $b^\prime(c)=0$ and as before in Lemma~\ref{stablemma} $\calE_{\d^\prime}$ is still stable. Now however the nilpotent $\Phi^\prime_c$ is no longer regular due to $b^\prime(c)=0$. This implies that, besides $V^\prime=(L^\prime_{k}\oplus \dots \oplus L^\prime_{n-1})_c$,  
	there is more than one $\Phi^\prime_c$-invariant $n-k$-dimensional subspace in $E^\prime_c$.  By \cite[Theorem 6]{shayman} the subvariety $S_{n-k}(\Phi^\prime_c)\subset Gr_{n-k}(E^\prime_c)$  of $\Phi^\prime_c$-invariant $(n-k)$-dimensional subspaces in the Grassmannian  of $(n-k)$-planes in $E^\prime_c$ is connected. The connectedness of $S_{n-k}(\Phi^\prime_c)$ also follows from \cite[Theorem 6.2]{horrocks} as one can think of $S_{n-k}(\Phi^\prime_c)$ as the support of the fixed point scheme of the unipotent group action of $\G_a\cong \{\exp(t \Phi^\prime_c)\}$ on the Grassmannian $Gr_{n-k}(E^\prime_c)$.
	
	We will find a curve in $ S_{n-k}(\Phi^\prime_c)$, leading to a Hecke curve in $\M$ by studying a natural $\T$-action on the projective variety $S_{n-k}(\Phi^\prime_c)$. We note that $\calE_{\d^\prime}=(E^\prime, \Phi^\prime)$ is $\T$-invariant, meaning that \beq \label{isom}(E^\prime, \Phi^\prime)\cong (E^\prime, \lambda \Phi^\prime)\eeq for all $\lambda\in \C$. The isomorphism in \eqref{isom} is induced by an automorphism of $(E^\prime,\Phi^\prime)$. Let $\T$ act on $E^\prime=L^\prime_0\oplus \dots \oplus L^\prime_{n-1}$ with weight $-i$ on $L^\prime_i$. This $\T$-action induces the weight $-1$ action of $\T$ on $\Phi^\prime$, showing \eqref{isom}. Restricting this $\T$-action to $E^\prime_c$ we get an action of $\T$ on $Gr_{n-k}(E^\prime_c)$. We note that $\Phi^\prime_c(\lambda\cdot v)= \lambda\cdot \Phi^\prime_c(v)$ and thus if $V\subset E^\prime_c$ is a $\Phi_c$-invariant subspace then $\lambda\cdot V$ is also $\Phi_c$-invariant. Therefore this $\T$-action leaves $S_{n-k}(\Phi^\prime_c)$ invariant. We have $S_{n-k}(\Phi^\prime_c)^\T\subset Gr_{n-k}(E^\prime_c)^\T$.  Fixed points of the $\T$-action on $Gr_{n-k}(E^\prime_c)$ are of the form $(L_{i_1}\oplus \dots \oplus L_{i_{n-k}})_c$ for any $(n-k)$-element subset $\{i_1,\dots,i_{n-k} \}\subset \{0,\dots,n-1\}$. We  have $$V^\prime=(L^\prime_{k},\dots, L^\prime_{n-1})_c\in S_{n-k}(\Phi^\prime_c)^\T\subset Gr_{n-k}(E_c^\prime).$$ By Remark~\ref{V'Higgs} we have that $\calH_{V^\prime}(E^\prime,\Phi^\prime)=(E(-c),\Phi)=\calE_{\delta-c}$. 

 As $S_{n-k}(\Phi^\prime_c)$ is projective, connected and contains more than one point, there is $V^\prime\neq V\in S_{n-k}(\Phi^\prime_c)$ such that $\lim_{\lambda\to 0} \lambda V=V^\prime$. Let $$V_\infty:=\lim_{\lambda\to \infty} \lambda\cdot V\in S_{n-k}(\Phi^\prime_c)^\T.$$ As the $\T$-action on $S_{n-k}(\Phi^\prime_c)$ is linear we see that  $V^\prime\neq V_\infty$. Let $V_\infty=(L_{i_1}\oplus \dots \oplus L_{i_{n-k}})_c$ for some \beq\label{subset}\{k,\dots, n-1\}\neq \{i_1,\dots,i_{n-k} \}\subset \{0,\dots,n-1\}.\eeq We thus have a map  $\P^1\to S_{n-k}(\Phi^\prime_c)$ sending $0$ to $V^\prime$, $\infty$ to $V_\infty$ and $\lambda \in \T$ to $\lambda\cdot V$. We will denote the corresponding subspaces $V_t$ for $t\in \C \cup {\infty}$. 
	
	First we need the following 
	
	\begin{lemma} $\calH_{V_\infty}(\calE_{\d^\prime})$ is a $\T$-fixed stable Higgs bundle. 
	\end{lemma}
	\begin{proof} Let $\calH_{V_\infty}(\calE_{\d^\prime})=(E^{\prime\prime},\Phi^{\prime\prime})$. Then $E^{\prime\prime}\cong L_0^{\prime\prime}\oplus \dots \oplus L_{{n-1}}^{\prime\prime}$ where $L_i^{\prime\prime}\cong L_i^\prime$ when $i\in\{i_1,\dots,i_{n-k}\}$ and $L_i^{\prime\prime}\cong L_i^\prime(-c)$ when $i\notin\{i_1,\dots,i_{n-k}\}$. Thus 
		\beq\label{lpp}  \ell_i^{\prime\prime}:=\deg(L_i^{\prime\prime} ) =\left\{ \begin{array}{cc} \ell_i^\prime & \mbox{ when } i\in\{i_1,\dots,i_{n-k}\} \\ \ell_i^\prime-1 & \mbox{ when } i\notin\{i_1,\dots,i_{n-k}\}  
		\end{array} 
		\right. \eeq
		The only $\Phi^{\prime\prime}$-invariant subbundles of $E^{\prime\prime}$ are of the form $L_i^{\prime\prime}\oplus \dots \oplus L_{n-1}^{\prime\prime}$. Thus to prove stability we check \bes  \frac{\ell^{\prime\prime}_i+\dots+\ell^{\prime\prime}_{n-1}}{n-i}\leq\frac{\ell^\prime_i+\dots+\ell^\prime_{n-1}-\max(0,k-i)}{n-i}&=&\frac{\ell_i+\dots+\ell_{n-1}-(n-i)}{n-i}  \\
		&<&  \frac{\ell_0+\dots+\ell_{n-1}-n}{n}=
		\frac{\ell^{\prime\prime}_0+\dots+\ell^{\prime\prime}_{n-1}}{n}
		\ees
		
	\end{proof}
	
	Let $V\in S_{n-k}(E_c)$ and denote $V_\lambda:=\lambda\cdot V$. Denote by $f_\lambda:E^\prime\to E^\prime$ the automorphism induced by our $\T$-action on $E^\prime$ as in \eqref{fixed}. Finally identify  $(K)_c\cong \C$.   We now have the commutative diagram:

	$$\begin{tikzcd}[column sep={between origins,5em}]
	& E^{\prime}_VK \ar{rr}
	& & E^\prime K \ar{rr} & & (E_c/V)_c \\
	E^{\prime}_V \ar{ur}{\lambda \Phi^{\prime}_V} \ar[crossing over]{rr} \ar{dd}[swap,near end]{f^{\prime\prime}_\lambda} 
	& & E^{\prime} \ar{ur}{\lambda \Phi^\prime}\ar{rr} & & (E_c/V)_c \ar{ur}{\lambda \overline{\Phi}^\prime_c}& \\
	& E^{\prime}_{V_\lambda}K \ar[leftarrow]{uu}[near start]{f^{\prime\prime}_\lambda}  \ar{rr} \ar[leftarrow, swap]{dl}{  \Phi^{\prime}_{V_\lambda} }
	& & E^\prime K \ar[leftarrow, swap]{dl}{ \Phi^\prime} \ar{rr} \ar[leftarrow, near start]{uu}{f^\prime_\lambda}& & (E_c/V_\lambda)_c \ar[leftarrow,near start]{uu}{\overline{f}^\prime_{\lambda c}}\\
	E^\prime_{V_\lambda} \ar{rr} & & E^\prime \ar[leftarrow, near start]{uu}{f^\prime_\lambda} \ar{rr} & & (E_c/V_\lambda)_c \ar[leftarrow, near start]{uu}{\overline{f}^\prime_{\lambda c}} \ar{ur}{ {\overline{\Phi}^\prime_c}} &
	\end{tikzcd}.$$
	
	The map $f_\lambda^{\prime\prime}:E_V^\prime\to E^\prime_{V_\lambda}$ is defined uniquely so as to make the diagram commutative.  The diagram shows that $f_\lambda^{\prime\prime}$ induces $(E^\prime_{V_\lambda},\Phi^\prime_{V_\lambda})\cong (E^\prime_V,\lambda\cdot \Phi^\prime_V)$. 
	
	Therefore   $\calH_{ V_t}(\calE_{\d})$ is stable for all $t$ since stability is an open condition for Higgs bundles (c.f. \cite[Lemma 3.7]{simpson1} or \cite[Proposition 3.1]{nitsure}). Finally, we note that $  \calH_{ V_0}(\calE_{\d})\cong \calE_{\d-c} \not\cong \calH_{ V_\infty}(\calE_{\d})$ because there exists an $i$ such that $\deg(L_i^{\prime\prime})\neq \deg(L_i(-c))$ because of \eqref{subset} and \eqref{lpp}. Thus we found a $\T$-invariant family of nilpotent stable Higgs bundles in $\M^s$ parametrized by $\bP^1$ connecting  $\calE_{\d-c}$ and $\calH_{ V_\infty}(\calE_{\d})$. Thus $W^+_{\calE_{\d-c}}\cap h^{-1}(0)$ is not trivial. Therefore $\calE_{\d-c}$ and so $\calE_{\d}$ is not very stable. The  theorem follows. 
	
\end{proof}

\begin{corollary} \label{type111} There is a very stable Higgs bundle in every type $(1,\dots,1)$ component $F_{\mm}\in \pi_0(\M^{s\T})$ from \eqref{type111fix}. In fact, they form a dense open subset. 
\end{corollary}

\begin{proof} From Theorem~\ref{mainverystable} the very stable locus in $F_{\mm}$ are Higgs bundles $\calE_{\d}$ where $b$ has no repeated zero. Thus they form the complement of some divisor. The result follows. 
\end{proof}

\subsubsection{Example in rank $2$}

Let $L_0$ and $L_1$ be line bundles on $C$ and $E=L_0\oplus L_1$ together with a lower triangular Higgs field: $\Phi=\left(\begin{array}{cc}0&0\\b&0\end{array} \right)$ with $b=b_1\in H^0(C;L_0^*L_1K)$. Assume that $\deg(b)<2g-2$ so that $(E,\Phi)$ is stable. Suppose further that $b=s_c^2 b^{\prime\prime}$ has a double zero at $c\in C$ where $b^{\prime\prime}\in H^0(C;L^*_0L_1K(-2c)).$ Let $V:=(L_0)_c< E_c$. It is $\Phi_c$-invariant as $\Phi_c=0$. We have $$E^\prime=L_0^\prime\oplus L_1^\prime=L_0\oplus L_1(-c)$$ and $V^\prime=V_0= (L_1^\prime)_c$. As $b^\prime=s_c b^{\prime\prime}$ still vanishes at $c$, $\Phi^\prime_c=0$. Thus $S_{1}(\Phi^\prime_c)=\P(E_c)$ the whole $\P^1$. Let $V_\infty=(L_0)_c<E^\prime_c$, which is $\Phi^\prime_c$-invariant. Fix basis vectors $\langle v_1\rangle=(L^\prime_1)_c$ and $\langle v_2 \rangle=(L^\prime_2)_c$ then let $V_t:=\langle tv_1+v_2 \rangle.$ 
We obtain $\lim_{t\to 0} V_t=V^\prime$ and $\lim_{t\to \infty} V_t=V_\infty$. Then we see that $\calH_{V^\prime}(E^\prime,\Phi^\prime)=(E(-c),\Phi)$ and  $\calH_{V_\infty}(E^\prime,\Phi^\prime)=(E^{\prime\prime},\Phi^{\prime\prime})$. Here $$E^{\prime\prime}=L_0^{\prime\prime}\oplus L_1^{\prime\prime}=L_0\oplus L_1(-2c)$$ and $\Phi^{\prime\prime}=\left(\begin{array}{cc}0&0\\b^{\prime\prime}&0\end{array}\right).$ For $t\neq 0,\infty$ the Hecke transforms $E^\prime_{V_t}$ are no longer direct sums but extensions \beq \label{firstext}0\to L_0(-c)\to E^\prime_{V_t}\to L_1(-c)\to 0,\eeq or \beq \label{secondext}0\to L_1(-2c) \to E^\prime_{V_t}\to L_0\to 0.\eeq The first one corresponds to  the modification of $E^\prime_1=L_0^\prime=L_0$ in $E^\prime_{V_t}$. In other words \eqref{firstext} is the modified filtration of Proposition~\ref{modifiedfilt}. This shows that $$\lim_{\lambda\to 0}\lambda\cdot \calH_{V_t}(E^\prime,\Phi^\prime) = (E(-c),\Phi).$$ In particular, as the latter is stable so is $\calH_{V_t}(E^\prime,\Phi^\prime)$.  

The second extension \eqref{secondext} is induced from $L_1(-2c)\subset E^\prime_{V_t}$ the Hecke modification of the subbundle $L_1^\prime\cong L_1(-c)\subset E^\prime$. As $\Phi^\prime$ was trivial on $L_1^\prime$ the modification $\Phi^\prime_{V_t}$ will be trivial on $L_1(-2c)$. Thus $\Phi^\prime_{V_t}$ will be given by projection to $L_0$ in \eqref{secondext} followed by $b^{\prime\prime}:L_0\to L_1(-2c)K$. As in Proposition~\ref{downwardfiltration} this shows that $$\lim_{\lambda\to \infty} \lambda\cdot \calH_{V_t}(E^\prime,\Phi^\prime) =(E^{\prime\prime},\Phi^{\prime\prime}). $$

Thus $\calH_{V_t}(E^\prime,\Phi^\prime)$ is a $
\T$-equivariant family of stable nilpotent Higgs bundles parametrized by $t\in \P^1=\C\cup \infty$, connecting $\calH_{V_0}(E^\prime,\Phi^\prime)=(E(-c),\Phi)$ and $\calH_{V_\infty}(E^\prime,\Phi^\prime)=(E^{\prime\prime},\Phi^{\prime\prime})$.

This shows that $(E(-c),\Phi)$ and thus $(E,\Phi)$ are not very stable. 

\subsubsection{Hecke transforms of Lagrangians}
We have seen above how the different  upward flows of  type $(1,\dots,1)$ very stable Higgs bundles are related by Hecke transforms. We also showed they were Lagrangian in Proposition~\ref{lagrangian}. It is a general fact that the Hecke transform takes Lagrangians to Lagrangians, as we show next.

Simpson constructed \cite[Theorem 4.10]{simpson1} the fine moduli space $\calR^s$ of  stable rank $n$ degree $d$ Higgs bundles, framed at the point $c\in C$. The group $\GL_n$ acts on the framing with quotient  $\M^s$ so that $\calR^s\to \M^s$ is a $\PGL_n$-principal bundle, associated to the projective universal bundle $\P(\E_c)$ restricted to $c$. Let $H_k\subset \GL_n$ be the stabilizer of $\C^k<\C^n$ - a maximal parabolic subgroup of $\GL_n$. The quotient  $\GL_n/H_k$ is isomorphic to the Grassmannian of $k$-planes in $\C^n$ and we can construct the Grassmannian bundle \beq \label{grassmann} Gr_k(\E_c):=\calR^s/H_k\eeq over $\M^s$. 

By identifying $K_c\cong \C$, \'etale locally the universal Higgs field $\bPhi_c$ restricted to $c$ gives an endomorphism $\bPhi_c\in H^0(\M^s;\End(\E_c))$. We will denote the set of $\Phi_c$-invariant $k$-dimensional subspaces of $E_c$ by $$\calH^s:=\{V< \E_{((E,\Phi),c)} | \mbox{ s.t. } 
(E,\Phi)\in \M^s, \dim(V)=k \mbox{ and } \Phi_c(V)\subset V \}\subset Gr_k(\E_c).$$ The subset $\calH^s$ has the structure of a variety, as \'etale-locally in $\M^s$ we can assume that  some $\lambda\in\C$ is never an eigenvalue of $\Phi_c$ in other words $\det(\Phi_c-\lambda)\neq 0$. That means that $\bPhi_c-\lambda$ is invertible and thus acts on $\Gr_k(\E_c)$ and  $\calH^s$ in this \'etale neighbourhood is the fixed point variety of a morphism. By construction \beq\label{proper}\pi:\calH^s\to \M^s\eeq is a proper map.

A point of $\calH^s$ is represented by a pair $((E,\Phi),V) \in \calH^s$  of a Higgs bundle $(E,\Phi)$ and $k$-subspace $V\in Gr_k(E_{c})$ such that $\Phi_c(V)\subset V$.   We also consider ${}^{s}\calH^s\subset \calH^s$ the subset of points $((E,\Phi),V)\in \calH^s$ where $V\in Gr_k(E_{c})$ such that $\calH_{V}(E,\Phi)$ is stable. As stability is an open condition for Higgs bundles ${}^{s}\calH^s\subset \calH^s$ is an open subset. Further we define $\M^\#\subset\M$ the open subset where the spectral curve is smooth and the Higgs field $\Phi_c$ at the point $c\in C$ has distinct eigenvalues. Then we set $\calH^\#:=\pi^{-1}(\M^\#)\subset \calH^{s}$ an open subset.  The Hecke transform does not change the spectral curve, and when it is smooth the Higgs bundle is automatically stable. Thus $\calH^\#\subset {}^{s}\calH^s$. 

By abuse of notation we will still denote $\pi:\calH^\# \to \M^\#$ the restriction of \eqref{proper}. If  $\M=\M_{n,d}$ then we denote $\M^\prime:=\M_{n}^{d-k}$ so that $\calH_V(E,\Phi)\in \M^{\prime s}$ when $((E,\Phi),V)\in {}^{s}\calH^s$.  This defines $\pi^\prime:\calH^\#\to \M^{\prime\#}$.  We have the following
\begin{proposition}\label{heckeinjective} The map
	$(\pi^\prime,\pi):\calH^\#\to \M^{\prime \#}\times  \M^\#$ is injective. 
\end{proposition}
\begin{proof}
	When the Higgs field $\Phi_c$ has $n$ distinct eigenvalues $\lambda_i$, a Hecke transform depends on the choice of  a $k$-element subset  or equivalently a subspace $\C^k< \C^n$ spanned by the corresponding eigenspaces.  If the equation of the spectral curve is given by $\det (x-\Phi)=0$ and $(E,\Phi)$ is the direct image of the line bundle $U$ on $C_a$ then the Hecke transform is the direct image of $U(-D)$  where $D$ is the effective divisor $(\lambda_{i_1},c)+(\lambda_{i_2},c)+\dots +(\lambda_{i_k},c)$.
	
	The map will be injective if for two choices  of subspace the divisor classes $D_1,D_2$ are distinct. This will follow if $\dim H^0(C_a, {\mathcal O}(D))=1$.
	By Riemann-Roch and Serre duality this condition is equivalent to $\dim H^0( C_a, K_{C_a}(-D))=\tilde g-k$ where $\tilde g$ is the genus of $ C_a$, or alternatively if  the restriction $H^0(C_a, K_{ C_a})\rightarrow H^0(D, K_{ C_a})\cong \C^k$ is surjective. 
	
	By adjunction on the total space of $K$, which has trivial canonical bundle, we have $K_{ C_a} \cong \pi^*K^n$ and hence sections of $\pi^*K^n$ of the form $a_0+a_1x+\dots +a_{n-1}x^{n-1}$ where $a_i$ is the pullback of a section of $H^0(C,K^{n-i})$. This space has dimension $n^2(g-1)+1=\tilde g$ and so represents uniquely  every section of $K_{ C_a}$. Restricting to $D$ gives $(v_1,\dots, v_k)$ where 
	$$v_m=a_0(c)+a_1(c)\lambda_{i_m}+\dots +a_{n-1}(c)\lambda_{i_m}^{n-1}.$$
	Since the $\lambda_i$ are distinct, Lagrange interpolation provides a polynomial of degree $(n-1)$ which agrees with any $(v_1,\dots, v_k)$ and because $K^i$ on $C$ has no base points, there exist sections $a_i$ agreeing at $c$ with the coefficients of the polynomial. Hence restriction is surjective.
\end{proof} 

We can define a correspondence on ${\mathcal M}'^s\times {\mathcal M}^s$ by saying that two points $(m,m')$ are equivalent if they are represented by Higgs bundles $(E,\Phi), (E',\Phi')$ where $(E',\Phi')$ is related to $(E,\Phi)$ by a $k$-plane Hecke transformation at $c$. Then, from Proposition~\ref{heckeinjective},  
at a generic point this is locally given by the graph of a map and hence has a dense open set which is a manifold of half the dimension.

Recall that a {\em Lagrangian correspondence} between symplectic manifolds $M_1,M_2$ is a Lagrangian submanifold  $ M_1\times M_2$ with respect to the symplectic form $p_1^*\omega_1-p_2^*\omega_2$, for example the graph of a symplectic transformation. Then points in $M_2$ corresponding to those in a Lagrangian in $M_1$ will, under transversality conditions, form a Lagrangian submanifold. The following result was already discussed in \cite[pp.185-186]{kapustin-witten}. 

\begin{theorem} The Hecke correspondence on ${\mathcal M}'^s\times {\mathcal M}^s$ is Lagrangian at a generic point.
\end{theorem}

\begin{proof}
	We use a differential-geometric approach and assume that $\Phi_c$ has distinct eigenvalues. 
	As defined above, when we have  a rank $n$ vector bundle $E$ on the curve $C$ and a point $c\in C$ the Hecke transform $E'$ is the vector bundle defined by the kernel subsheaf of a homomorphism $h: {\mathcal O}(E)\rightarrow {\mathcal O}_c^{n-k}$ with the kernel at $c$ the vector space $V$.

	Let $v_1,\dots, v_n$ be a local basis of sections of $E$  with $v_1,\dots, v_k$ spanning $V$ at $c$. Using a local coordinate where $c$ is $z=0$,  a local basis for $E'$ is given by $ u_1,\dots,  u_n$ where $v_i=u_i$ for $i>k$  and $v_i=zu_i$ for $1\le i\le k$. 
	We shall use these two local bases to make calculations.
	
	Varying $V\subset E_c$ gives a variation of the holomorphic structure on $E'$ as in \cite{NR}. 
	A first order deformation of $V$  in a parameter $t$  is described by local holomorphic sections of the form 
	$$ v_i+t\dot v_i=v_i+t\sum_{j=k+1}^n a_{ij}v_j$$
	for $1\le i\le k$. The variation of a local basis of the  subsheaf is then 
	$$\dot u_i=\frac{1}{z} \sum_{j=k+1}^n a_{ij}u_j.$$
	The matrix $a=a_{ij}$ in the basis $u_1,\dots, u_n$ for $E'$ gives the    section
	$$\begin{pmatrix}  0 & 0 \\
	a/z& 0\end{pmatrix}$$ over a punctured disc. 
	
	For a fixed holomorphic structure on $E$, this is   a \v Cech
	cocycle for a class in $H^1(C;\End (E'))$ defining the variation of the holomorphic structure on the Hecke transform corresponding to $a(c)\in  \Hom(V,E_c/V)$, a tangent vector to the Grassmannian.    Here however we need to vary $E$ also and it is convenient to use 
	a Dolbeault representative $\dot A\in \Omega^{0,1}(C;\End (E))$. We may make the assumption that the deformation of holomorphic structure on $E$ is trivial in a neighbourhood of $c$ so that we can take $\dot A=0$ in this neighbourhood,  extend $a$ to a $C^{\infty}$ section  supported in a neighbourhood of $c$ and holomorphic in a smaller one and take $\dot A+\bar\partial \dot a/z\in \Omega^{0,1}(C; \End (E'))$ giving the combined deformation of holomorphic structure on $E'$.

	Let $(\dot A,\dot\Phi)$ be a first order deformation of a Higgs bundle preserving the deformation of the subsheaf.   
	A  local section of $\End (E)$ can be written in block matrix form with respect to the subspaces  spanned by $v_1,\dots, v_k$ and $v_{k+1},\dots, v_n$ respectively as
	$$\begin{pmatrix}  P & Q\\
	R & S\end{pmatrix} $$ and  in the basis $u_1,\dots, u_n$ it becomes 
	\begin{equation}
	\begin{pmatrix}  P & zQ\\
	z^{-1}R & S .\end{pmatrix}
	\label{matrix1}
	\end{equation} 
	From this we see explicitly that  a Higgs field $\Phi\in H^0(C;\End (E)\otimes K)$  extends to $\End (E')\otimes K$ if  and only if  the entry $R$ is divisible by $z$, i.e. $\Phi$ preserves $V$ at $c$, which is the definition of $\calH^s$.

	To preserve the subsheaf to first order we need 
	$(\Phi+t\dot\Phi)(v+ta(v))=w+ta(w)$ modulo $t^2$ 
	for $v$ and $w$ linear combinations of  $v_1,\dots,v_k$. It follows that  $\dot\Phi+[\Phi,a]$ preserves the subspace $V$ at $c$.

	Write $$\Phi=\begin{pmatrix}  P & Q\\
	R & S\end{pmatrix}\quad \dot\Phi=\begin{pmatrix}  \dot P & \dot Q\\
	\dot R & \dot S\end{pmatrix}$$
	Since $\dot A=0$ in a neighbourhood the entries are holomorphic.  
	Then
	$$\dot\Phi+[\Phi,a]=\begin{pmatrix}  \dot P- Qa& \dot Q\\
	\dot R +Sa-aP& \dot S-aQ\end{pmatrix}.$$
	This  preserves $V$, so $\dot R +Sa-aP$ must be divisible by $z$.  Since $\Phi_c$ has distinct eigenvalues $\lambda_i$ we can take the basis vectors $v_i$ to be eigenvectors of $\Phi$ in a neighbourhood and then the condition is that 
	$\dot R_{ij}+(\lambda_i-\lambda_j)a_{ij}$ be divisible by $z$ and accordingly  $a_{ij}(0)$ is uniquely determined by $\dot \Phi$.
	
	In the basis for $E'$ we have  
	$$\dot\Phi+[\Phi,a]=\begin{pmatrix}  \dot P- Qa& \dot Q z\\
	(\dot R +Sa-aP)z^{-1}& \dot S-aQ\end{pmatrix}.$$
	which is now holomorphic in a neighbourhood of $c$.

	Let $(\dot A, \dot \Phi)\in \Omega^{0,1}(C;\End (E))\oplus \Omega^{1,0}(C,\End (E))$ be a Dolbeault representative of a tangent vector to ${\mathcal M}^s$. The holomorphic symplectic form $\omega$ is defined by
	$$\int_C\tr(\dot A_1\dot\Phi_2)-\tr(\dot A_2\dot\Phi_1).$$

	With $\dot A'=\dot A+\bar\partial  a/z$ as above we obtain
	$$ \int_C\tr(\dot A'_1\dot\Phi'_2)= \int_C\tr(\dot A_1\dot\Phi_2)+\tr (\bar\partial  a_1Q_2/z)-\tr \,\bar\partial  a_1[\Phi,a_2]/z$$
	But $Q_2$ is holomorphic on the support of $ a_1$ and divisible by $z$ so by Stokes' theorem the second term vanishes. By the Jacobi identity 
	$\tr (a_1[\Phi,a_2])-\tr (a_2[\Phi,a_1])=\tr (\Phi[a_1,a_2])$ but the off-diagonal matrices $a_1,a_2$ commute so  $\tr\, \bar\partial  a_1[\Phi,a_2]/z=\tr \,\bar\partial  a_2[\Phi,a_1]/z$.
	Then
	$$ \int_C\tr(\dot A'_1\dot\Phi'_2)= \int_C\tr(\dot A_1\dot\Phi_2)$$
	and the two symplectic structures coincide. 
\end{proof}

\section{Multiplicities in the nilpotent cone}
\label{multiplicities}

\subsection{Multiplicities of very stable components}
Let $\calE\in \M^{s \T}\subset h^{-1}(0)$ be a very stable Higgs bundle. Denote by $F_\calE\in \pi_{0}(\M^{ \T})$ the component of the $\T$-fixed point set containing $\calE$.   Let $N:=h^{-1}(0)\subset \M$ denote the subscheme of nilpotent Higgs bundles, the so called {\em (global) nilpotent cone}. Then the corresponding subvariety $N_{red}\cong \calC\subset \M$ is isomorphic to the core of $\M$. Denote by $\calC_{F_{\calE}}:= \overline{W^-_{F_\epsilon}}$ the closure of the downward flow from $F_\calE$, which is an irreducible component of $\calC$. Denote by $N_{F_{\calE}}\subset N$ the corresponding irreducible component of the subscheme $N$.  

\begin{definition} \label{defmult} The {\em multiplicity} of the component $N_{F_\calE}$   in the subscheme $N$ is defined to be the length of the local ring $\calO_{N,N_{{F_\calE}}}$ at a generic point  and is  denoted by $m_{F_\calE}:=\ell(\calO_{N,N_{{F_\calE}}})$.   
\end{definition}

 Let $k\in \Z$ be an integer. For a $\T$-module $U=\bigoplus_{\lambda\leq k} U_\lambda$  with finite dimensional weight spaces $\dim(U_\lambda)<\infty$ corresponding to $\lambda\in \Hom(\T,\T)\cong \Z$   we denote its character by \beq\label{trunchar}\chi_\T(U):=\sum_{\lambda\leq k} \dim(U_\lambda)t^{-\lambda} \in \Z((t)).\eeq For a finite dimensional positive $\T$-module $V=\oplus_{\lambda>0} V_\lambda$  the  character of the  symmetric algebra $\Sym(V^*)=\bigoplus_{\lambda\leq 0} \Sym(V^*)_\lambda$ of its dual satisfies
\beq\label{sym} \chi_\T({\Sym(V^*)})&=& \sum_{\lambda\leq 0} \dim(\Sym(V^*)_\lambda)t^{-\lambda}= 
\prod_{\lambda>0}\frac{1}{( 1- t^{\lambda})^{\dim(V_\lambda)}} \in \Z[[t]].\eeq In particular,  the equivariant Euler characteristic (see \S\ref{equivarianteuler} for more details) of the structure sheaf $\calO_V$ on $V$ can be computed as
\beq \label{eqeu} \chi_\T(V;\calO_V)=\sum_{i,k} (-1)^i t^{-k} \dim H^i(\calO_V)_{k} =\sum_k t^{-k} \dim H^0(\calO_V)_{k} = \chi_\T({\Sym(V^*)}),\eeq where we recall that $\T$ acts on $s
\in H^0(\calO_V)$ with non-positive weights by the formula $(\lambda\cdot s)(v)=s(\lambda^{-1}\cdot v)$. 

Finally, we let $T_\calE^+\M<T_\calE\M$ denote the part of the $\T$-module $T_\calE\M$ with positive weights. Then $T_\calE W^+_\calE\cong T_\calE^+\M$ which  we will abbreviate as $T_\calE^+$.
\begin{theorem} \label{multiplicity} When $\calE\in \M^{s \T}\subset N$ is a very stable Higgs bundle then the multiplicity of the component $N_{F_\calE}$   in the nilpotent cone $N$ satisfies $$m_{F_\calE}=\rank(h_*(\calO_{W^+_\calE}))=\left. { \chi_\T(\Sym(T^{+*}_\calE))\over \chi_\T(\Sym(\calA^*))}\right|_{t=1}.$$  Moreover, for a generic $a\in \calA$  the intersection $W^+_\calE\cap h^{-1}(a)$ is transversal and has cardinality $m_{F_\calE}$. 
\end{theorem}
\begin{proof} We shall use basic properties of the 
	K-theory of coherent sheaves with supports as developed in e.g. \cite{baum-etal,gillet-soule, piepmeyer-walker}. For $Z\subset X$ a closed algebraic  subset of a complex algebraic variety we will denote by $K^\circ_Z(X)$ (resp. by $K^Z_\circ(X)$)  the Grothendieck group of bounded complexes of locally free sheaves (resp. coherent sheaves) with homology supported in $Z$.
	
	We shall compute the intersection product \beq \label{rankmult} [\calO_{N}]\cap [\calO_{W^+_\calE}]\in  K_\circ(\{\calE\})\cong \Z\eeq 
	in two different ways. Here $[\calO_{N}]=h^*([\calO_0])\in K^\circ_\calC(\M)$, where  $[\calO_0]\in K^\circ_{\{0\}}(\calA)$
	and $[\calO_{W^+_\calE}]\in K^{W^+_\calE}_\circ(\M)$, where $W^+_\calE\subset \M$ is closed by Lemma~\ref{transversal}.  Thus the intersection product $$[\calO_{N}]\cap[\calO_{W^+_\calE}]=[\calO_N\otimes^L\calO_{W^+_\calE}] \in K_\circ^{\calC\cap W^+_{\calE}}(\M)\cong K^{\{\calE\}}_\circ (\M)\cong K_\circ(\{\calE\})\cong \Z$$ is defined. The last isomorphism is given by the Euler characteristic $\chi:K_\circ(\{ \calE \})\to \Z$.  

	First we note that $W^+_\calE$ intersects $\calC=N_{red}$ transversally at the single point $\calE$, since both $W^+_\calE$ and $\calC$ are smooth at $\calE$ and the tangent spaces  $T^+_\calE$ and $T^{\leq  0}_\calE$ are complementary. This also follows from Bialynicki-Birula's local description of the upward flow (see Remark~\ref{afffibr}). Then \cite[\href{https://stacks.math.columbia.edu/tag/0B1I}{Lemma 42.14.3}]{stackproject} implies that the intersection multiplicity \beq\label{tranint}i(\M,W^+_\calE\cdot \calC,\{\calE\})=i(\M,W^+_\calE\cdot \calC_{F_\calE},\{\calE\})=1.\eeq
	
	We can also find an affine open neighbourhood $\spec(R)$  of $\calE\in \M$ so that $\calO_{W^+_\calE}|_{\spec(R)}$ is represented by the $R$-algebra $A$ and $\calO_{N_{F_\calE}}$ by the $R$-algebra $B$. As ${N_{F_\calE}}$ is irreducible $B$ has a unique minimal prime ideal $I_{min}\lhd B$ which corresponds to the generic point of ${N_{F_\calE}}$ and coincides with the nilradical of $B$. In particular, the reduced ring $B_{red}=B/I_{min}$ and \beq\label{reduced}\spec (B_{red}) \cong \calC_{F_{\calE}}\cap{\spec(R)}.\eeq  By \cite[Theorem 2 of IV.1.4]{bourbaki} 
	we know that $B$ as a $B$-module has a composition series $$0=B_k\subset B_{k-1} \subset \dots\subset B_0=B$$ with factors $B_j/{B_{j+1}}\cong B/P_j$, where $P_j\lhd B$ are prime ideals. Moreover we know from  {\em loc. cit.} that the factor $B/I_{min}$ appears $\ell(B_{I_{min}})$ times in this composition series, which further agrees with $$\ell(B_{I_{min}})=\ell(\calO_{N,N_{{F_\calE}}})=m_{F_\calE}.$$ 
	
	Now $\calO_{W^+_\calE}\otimes^L \calO_{N_{F_\calE}}$ has support $\{\calE\}$. The Euler characteristic of its stalk at $\{\calE\}$ is given by Serre's Tor formula
	$$\chi(A_{I_\calE},B_{I_\calE})=\sum (-1)^i \ell(\Tor_i(A_{I_\calE},B_{I_\calE})),$$
	where $I_\calE\lhd R$ is the maximal ideal corresponding to the point $\calE\in \spec(R)\subset \M$. Computing this further
	$$\sum_i (-1)^i \ell(\Tor_i(A_{I_\calE},B_{I_\calE}))=\sum_i \sum_j (-1)^i \ell(\Tor_i(A_{I_\calE},(B/P_j)_{I_\calE}))=\sum_j \chi(A_{I_\calE},(B/P_j)_{I_\calE}).$$
	We know that $P_j\supset I_{min}$. When $P_j\neq I_{min}$ then $\dim(B/P_j)<\dim(B/I_{min})$ and so Serre's \cite[Theorem C.1.1.a]{serre} shows that $\chi(A_{I_\calE},(B/P_j)_{I_\calE})=0$ in this case. Thus  the Euler characteristic of the stalk of $\calO_{W^+_\calE}\otimes^L \calO_{N_{F_\calE}}$ at $\{\calE\}$ equals $$\chi(A_{I_\calE},B_{I_\calE})=m_{F_\calE} \chi(A_{I_\calE},(B/I_{min})_{I_\calE})=m_{F_\calE},$$
	as $$\chi(A_{I_\calE},(B/I_{min})_{I_\calE})=\chi(A_{I_\calE},(B_{red})_{I_\calE})=i(\M,W^+_\calE\cdot \calC_{F_\calE},\{\calE\})=1$$ from \eqref{tranint},\eqref{reduced} and \cite[Theorem C.1.1.b]{serre}. This  implies  that \beq \label{intermult}\chi([\calO_{W^+_\calE}]\cap [\calO_{N}])=m_{F_\calE}.\eeq

	On the other hand, by the projection formula in $K$-theory \cite[(1)]{piepmeyer-walker} \beq\label{projection}h_*([\calO_{N}] \cap [\calO_{W^+_\calE}])=[\calO_0] \cap h_*([\calO_{W^+_\calE}]) \in K^{\{0\}}_\circ(\calA).\eeq
	By Lemma~\ref{locfree} $h_*([\calO_{W^+_\calE}])$ is locally free and we  compute $$\chi([\calO_{N}]\cap [\calO_{W^+_\calE}])=\chi(h_*([\calO_{N}]\cap [\calO_{W^+_\calE}]))=\chi([\calO_0]\cap h_*([\calO_{W^+_\calE}]))=\chi([h_*([\calO_{W^+_\calE}])_0])=\rank(h_*(\calO_{W^+_\calE})).$$
	Thus the multiplicity $m_{F_\calE}$ in \eqref{intermult} agrees with $\rank(h_*(\calO_{W^+_\calE}))$ the rank of the locally free sheaf $h_*(\calO_{W^+_\calE})$. 
	
	To determine this rank, we notice that  $W^+_\calE$ is $\T$-invariant. 
	We  compute the equivariant Euler characteristic of $\calO_{W^+_\calE}$  as \beq \label{one}\chi_\T(\calO_{W^+_\calE})=\chi_\T(\calO_{T_\calE^+})= \chi_\T(\Sym(T_\calE^{+*}))\in\Z((t))\eeq since the $\T$-action on $W^+_\calE$ is modelled on the $\T$-module $T_\calE^+$ i.e. $W^+_\calE\cong T_\calE^+$ are $\T$-equivariantly isomorphic (c.f. \eqref{model}). 
	
	On the other hand because $h:W^+_{\calE}\to \calA$ is finite the $\T$-equivariant sheaf $h_*(\calO_{W^+_\calE})$ is locally free. By the equivariant  Serre conjecture proved for $\T$ in \cite{equserre} the $\T$-equivariant locally free sheaf $h_*(\calO_{W^+_\calE})$ on the affine space $\calA$ is trivial. In other words \beq\label{hitchinpush}h_*(\calO_{W^+_\calE})\cong h_*(\calO_{W^+_\calE})_0\times \calA,\eeq where the $\T$-action on the $0$ fibre  $h_*(\calO_{W^+_\calE})_0$ is inherited from the $\T$-equivariant structure on $h_*(\calO_{W^+_\calE})$ and the $\T$ action on $\calA$ is the standard one.  Thus we  compute \beq \label{other}\chi_\T(\calO_{W^+_\calE})=\chi_\T(Rh_*(\calO_{W^+_\calE}))=\chi_\T(h_*(\calO_{W^+_\calE}))=\chi_\T(h_*(\calO_{W^+_\calE})_0) \chi_\T(\Sym(\calA^*)) \in \Z((t)).\eeq Here  $\chi_\T(h_*(\calO_{W^+_\calE})_0)$ of \eqref{trunchar} is thought of as the $\T$-equivariant Euler characteristic over the point $0\in \calA$.
	Comparing \eqref{one} and \eqref{other} we see that the equivariant Euler characteristic of the $\T$-module $h_*(\calO_{W^+_\calE})_0$ can be expressed as \beq \label{t-character}\chi_\T(h_*(\calO_{W^+_\calE})_0)={ \chi_\T(\Sym(T^{+*}_\calE))\over \chi_\T(\Sym(\calA^*))}\in \Z((t)).\eeq In particular, the dimension of $h_*(\calO_{W^+_\calE})_0$ is $$\rank(h_*(\calO_{W^+_\calE}))=\dim(h_*(\calO_{W^+_\calE})_0)=\left. { \chi_\T(\Sym(T^{+*}_\calE))\over \chi_\T(\Sym(\calA^*))}\right|_{t=1}.$$
	This proves the first statement. For the second statement note that 
	when $\calE\in \M^{s\T}$ is very stable then $W^+_\calE\subset \M$ is closed and for   $a\in\calA$ \beq \label{generic}h_*([\calO_{W^+_\calE}]\cap [h^{-1}(a)])=[h_*(\calO_{W^+_\calE})\cap \calO_a]=[h_*(\calO_{W^+_\calE})_a]=m_{F_\calE}\in K(\{a\})\cong \Z.\eeq When $a\in \calA$ generic then by Lemma~\ref{locfree} the intersection of $W^+_\calE$ and $h^{-1}(a)$ is transversal and thus has cardinality $m_{F_\calE}$ from \eqref{generic}.

\end{proof}

\begin{definition}\label{equmult} For $\calE\in \M^{s
		\T}$ define the rational function $$m_{\calE}(t):={{ \chi_\T(\Sym(T^{+*}_\calE))\over \chi_\T(\Sym(\calA^*))}}\in \Z((t)).$$ 
	We call it the {\em virtual equivariant multiplicity} or {\em virtual multiplicity} for short. 
\end{definition}

\begin{corollary}\label{eqmult}
	When $\calE\in \M^{s
		\T}$ is very stable $m_{\calE}(t)$ is a palindromic and monic polynomial with non-negative integer coefficients, such that $m_{\calE}(1)=m_{F_\calE}$ is the multiplicity of the component $N_{F_\calE}\subset N$ in the nilpotent cone.
\end{corollary}
\begin{proof}
	When $\calE$ is very stable  by \eqref{t-character} $m_{\calE}(t)$ is the Laurent polynomial $\chi_\T(h_*(\calO_{W^+_\calE})_0)\in \Z[t,t^{-1}]$, the character of the finite dimensional $\T$-module $h_*(\calO_{W^+_\calE})_0$. In particular, $m_{\calE}(t)$ has non-negative coefficients.  
	
	When all the weights of $\T$ on a finite dimensional $\T$-module $V$ are positive, then ${1/ \chi_{\T}(\Sym(V^*))}$ is a monic Laurent series i.e. $$\left({1\over \chi_{\T}(\Sym(V^*))}\right)_{t=0}=1$$ from \eqref{sym}. It follows that $m_{\calE}(t)\in \Z[t]$ is a polynomial and $m_{\calE}(0)=1$. 
	
	As  $$\prod_{\lambda>0}{( 1- t^{\lambda})^{\dim(V_\lambda)}}=(-1)^{\dim V}t^{\sum \lambda\dim(V_\lambda)} \prod_{\lambda>0}{( 1- t^{-\lambda})^{\dim(V_\lambda)}}$$ we get  \beq\label{palindromic}m_{\calE}(t)=m_{\calE}(t^{-1}) t^{\sum \lambda\left(\dim(\calA_\lambda)-\dim((T^+_\calE)_\lambda)\right)}=m_{\calE}(t^{-1}) t^{\deg(m_{\calE}(t))}
	\eeq where $$\deg(m_{\calE}(t))=\sum \lambda\left(\dim(\calA_\lambda)-\dim((T^+_\calE)_\lambda)\right)$$ is the degree of the polynomial $m_{\calE}(t)$. Thus $m_{\calE}(t)$ is palindromic.  As $m_{\calE}(0)=1$  it is monic too. The result follows.   
\end{proof}

\begin{remark} In Section~\ref{further} we shall see that, in the cases considered there, $m_{\calE}(t^2)$ is a product of Poincar\'e polynomials of compact homogeneous spaces and hence monic and palindromic by Poincar\'e duality. 
\end{remark} 

\begin{remark} For a very stable $\calE\in \M^{s\T}$ we have $m_{\calE}(t)=\chi_\T(h_*(\calO_{W^+_{\calE}})_0)$ and so $\T$ acts on the vector space $h_*(\calO_{W^+_{\calE}})_0$ with non-positive weights and the multiplicity of the trivial character of $\T$ is $1$ in $h_*(\calO_{W^+_{\calE}})_0$.
	
\end{remark}

\begin{definition} When $\calE\in M^{s\T}$ is very stable we call $m_{\calE}(t)$ the {\em equivariant multiplicity} of $N_{F_\calE}\subset N$. 
\end{definition}
\begin{remark} Note that by the $\T$-equivariant analogue of the projection formula \eqref{projection} we see that the equivariant multiplicity at a very stable Higgs bundle satisfies $$m_{\calE}(t)=[\calO_N]^\T\cap [\calO_{W^+_\calE}]^\T\in K_\circ^\T(\{\calE\})\cong \Z[t,t^{-1}]. $$ Hence the name. 
\end{remark}
\begin{remark} \label{multN}We know from \cite[Proposition 3.5]{laumon} that   very stable rank $n$ vector bundles $E\in \calN_n$ exist.  Then $\calE=(E,0)$ is a very stable Higgs bundle, and $T^+_\calE\cong T^*_E\calN$ and $\T$ acts with weight one. Thus  \beq \label{cotchar}\chi_\T(\Sym(T^{+*}_\calE))={1\over (1-t)^{n^2(g-1)+1}}\eeq  while $$\chi_\T(\Sym(\calA^*))={1\over (1-t)^g(1-t^2)^{3g-3}(1-t^3)^{5g-5}\cdots (1-t^n)^{(2n-1)(g-1)}}.$$ Thus the equivariant multiplicity in this case
	is \beq \label{eqmultN}m_{\calE}(t)={{ \chi_\T(\Sym(T^{+*}_\calE))\over \chi_\T(\Sym(\calA^*))}}=[2]_t^{3g-3}[3]_t^{5g-5}\cdots [{n}]_t^{(2n-1)(g-1)}, \eeq  where $$[n]_t:=\frac{1-t^n}{1-t}=1+t+\dots+t^{n-1}$$ are the quantum integers. So by Theorem~\ref{multiplicity} the multiplicity of the component of the nilpotent cone $N$ whose reduced subvariety is isomorphic with $\calN$, the moduli space of rank $n$ stable bundles, is    $$m_{\calN}=2^{3g-3}3^{5g-5}\cdots n^{(2n-1)(g-1)}.$$ This number was also computed in \cite{beauville-etal} as the degree of the dominant map $h^{-1}(a)\to \calN$.   We will see in Subsection~\ref{cotfib} below that this is not a coincidence.  
	
	Finally, we note that for any other $\calF\in \M^{s\T}$ $${\chi_\T(\Sym(T^{+*}_\calE)) \over \chi_\T(\Sym(T^{+*}_\calF)) }\in \Z[t]$$ is a polynomial from the definition \eqref{sym}, thus we can deduce the following. 
\end{remark}

\begin{corollary}  For $\calF\in \M^{s\T}$  very stable, the equivariant multiplicity divides \eqref{eqmultN}:
	$$m_\calF(t)|[2]_t^{3g-3}[3]_t^{5g-5}\cdots [{n}]_t^{(2n-1)(g-1)}.$$  In particular, it is a product of cyclotomic polynomials. 
\end{corollary}

\begin{remark}\label{rank2} Let $n=2$. By Corollary~\ref{type111} and Theorem~\ref{laumon} we know that there exists a very stable Higgs bundle $\calE$ in every component $\calN,F_i\in \pi_0(\M_2)$ of the fixed point set of the $\T$-action on $\M$ (see Remark~\ref{remarkn2}).  We see that $T^+_\calE=V_2\oplus V_1$ has only weights $1$ and $2$ and $V_1\cong T^*F_i$, thus $$\chi_\T(\Sym(T^{+*}_\calE))={1\over (1-t)^{\dim F_i}(1-t^2)^{\dim \M/2-\dim F_i}}={1\over(1-t)^{i+g}(1-t^2)^{3g-i-3}}.$$ On the other hand $$\chi_\T(\Sym(\calA^*))={1\over(1-t)^g(1-t^2)^{3g-3}}.$$ Consequently $$m_{\calE}(t)={{ \chi_\T(\Sym(T^{+*}_\calE))\over \chi_\T(\Sym(\calA^*))}}=(1+t)^{i}=[2]_t^i.$$  Thus from Theorem~\ref{multiplicity} we get that $$m_{F_i}=2^i.$$ This result was proved in \cite[Proposition 6]{hitchin1} for $\SL_2$ Higgs bundles ($i$ even) and for twisted $\SL_2$ Higgs bundles ($i$ odd) in \cite{hausel-thaddeus}.	\end{remark}

\begin{remark} \label{counter}
	One can compute  the rational function ${{ \chi_\T(\Sym(T^{+*}_\calE))/ \chi_\T(\Sym(\calA^*))}}$ for every $\calE\in \M^{s\T}$, which by 
	Corollary~\ref{eqmult} equals the polynomial $m_{\calE}(t)$, when $\calE$ is very stable. This quantity depends only on the ambient $F_\calE\in \pi_0(\M^\T)$ and when it is not a polynomial, we  deduce that there exists no very stable Higgs bundle in $F_\calE$. 
	
	As an example recall \cite{gothen} that a type $(1,2)$ fixed point $\calE$ of the $\T$-action has underlying rank $3$ vector bundle of the form $E=L\oplus V$, where $L$ is a rank $1$ and $V$ is a rank $2$ vector bundle on $C$ and
	the Higgs field only non-trivial in $H^0(C;L^{*}VK)$. If $\ell=\deg(L)$ and $v=\deg(V)$ then $d=\deg(E)=\ell+v$. It is shown in \cite[Proposition 2.5]{gothen} that there is a stable Higgs bundle of this form if $ d/3<\ell<d/3+g-1$ or equivalently $0<2\ell-v<3g-3$. 
	
	We can easily compute the weight spaces of the $\T$-module 
	$T^+_\calE$. It has two weights $1$ and $2$, and we can compute the $2$-weight space isomorphic to $H^0(C;V^*LK)$ whose dimension is $2\ell-v+2g-2$ from Hirzebruch-Riemann-Roch and the fact that $H^1(C;V^*LK)\cong H^0(C;VL^*)^*\cong 0$ from \cite[proof of Proposition 4.2]{gothen}. Because of the homogeneity $1$ symplectic form,  the weight $2$ space is isomorphic with the dual of the weight $-1$ space, thus $2\ell-v+2g-2$ also agrees with half of the Morse index of the ambient fixed point component, matching \cite[Proposition 4.2.i]{gothen}. It follows that  $$\chi_\T(\Sym(T_\calE^{+*}))={1\over(1-t)^{9g-8-(2g-2+2\ell-v)}(1-t^2)^{2\ell-v+2g-2}}.$$ This implies \beq \label{nonpol}m_\calE(t)={{ \chi_\T(\Sym(T^{+*}_\calE))\over \chi_\T(\Sym(\calA^*))}}=(1+t)^{g-1-2\ell+v}(1+t+t^2)^{5g-5}.\eeq
	This is a polynomial in $t$  if and only if  $2\ell-v\leq g-1$. This shows that the type $(1,2)$ components of $
	\M_3^\T$ where $g-1<2\ell-v<3g-3$ do not contain very stable Higgs bundles.
	
	In fact, we can see this fact directly as follows. Our fixed point is $L\oplus V$ with Higgs field $\phi:L\rightarrow VK$. Assume for simplicity that the rank $3$ bundle has degree $0$. The upward flow consists of extensions $L\rightarrow E\rightarrow V$ where $\Phi$ acting on $L$ and projecting to $V$ is $\phi$.
	
	Let $U\subset V$ be the line bundle generated by $\phi$, or the image of $\phi:LK^*\rightarrow V$. Then 
	$$\deg U\ge \deg L-\deg K=\ell-(2g-2).$$
	Its annihilator is $(V/U)^*\subset V^*$ which has degree 
	$\deg U-\deg V=\deg U+\ell$ so the degree of $L(V/U)^*K$ is 
	$$\deg U+2\ell +(2g-2)\ge 3\ell.$$
	
	If $\ell>(g-1)/3$ then $L(V/U)^*K$ always has a non-trivial section $\psi$ which we can think of as lying in $\Hom(V, LK)$.
	Then $(\phi,\psi):L\oplus V\rightarrow (V\oplus L)\otimes K$ is a Higgs field on $L\oplus V$ and $\langle \phi,\psi\rangle = 0$ means it is nilpotent.

\end{remark}

\begin{remark} \label{obstruction}  For rank $2$ we know that every component of $\M^\T$ contains a very stable Higgs bundle, thus every component of the nilpotent cone is very stable (see Remark~\ref{rank2}). When a bundle is not very stable, the current term for it is {\em wobbly}. So for  rank $3$ above we found type $(1,2)$ wobbly components of the nilpotent cone by computing $m_{F}(t)$  in \eqref{nonpol} and finding that it is not a polynomial. Interestingly, precisely in the case when $m_{F}(t)$ was not a polynomial we found wobbly directions in $W^+_\calE$.    In fact  for rank $3$ we conjecture that all other type $(1,2)$ components of the nilpotent cone are very stable. As the type $(2,1)$ fixed points behave similarly to the type $(1,2)$ ones by duality, for rank $3$ we can formulate the conjecture that a component  $F\in \pi_0(\M^{s\T})$ contains a very stable Higgs bundle if and only if $m_{\calE}(t)$ (which is independent of $\calE\in F$) is a polynomial for a $\calE\in F$. 
	
	Adapting the argument for the type $(1,2)$ wobbly fixed point components above to the rank  $4$ case  we find that all stable type $(1,3)$ $\T$-fixed Higgs bundle are  wobbly. However,  the corresponding $m_\calE(t)$ is a polynomial, using the description in \cite[Example 6.6]{garcia-prada-heinloth-schmitt}. 
	Thus the obstruction of integrality of $m_\calE(t)$  for very stable Higgs bundles discussed in Remark  \ref{counter} is not sufficient. In particular, already for rank $4$ we do not have a conjectured list of very stable components of the nilpotent cone. 
	
\end{remark}

\begin{remark}\label{remark111}
	We know from Theorem~\ref{mainverystable} that all type $(1,\dots,1)$ components of the nilpotent cone are very stable and we can explicitly compute  ${{ \chi_\T(\Sym(T^{+*}_\calE))/ \chi_\T(\Sym(\calA^*))}}$ for a fixed point  $\calE\in \M^{s\T}$ of type $(1,\dots,1)$. In this case $E=L_0\oplus \dots \oplus L_{n-1}$ and  $\Phi(L_{i-1})\subset L_{i}K$ is determined by $b_i:=\Phi|_{L_{i-1}}:L_{i-1}\to L_{i}K$ non-zero. 
	We denote $\ell_i=\deg(L_i)$ and $m_i=2g-2-\ell_i+\ell_{i+1}$. As $b_i$ is non-zero $m_i\geq 0$. 
	
	By \eqref{tangent} the tangent space $T_\calE$  to ${\mathcal M}^s$ at this point is the hypercohomology group
	\begin{multline*}T_\calE\cong \H^1\left(\End(E)\stackrel{\ad(\Phi)}{\rightarrow}\End(E)\otimes K\right)\cong \\ \cong \bigoplus_{k=-n+1}^{k=n} \H^1\left(\bigoplus_{i-j=k}\Hom(L_i,L_{j})\stackrel{\ad(\Phi)}{\rightarrow}\bigoplus_{i-j=k-1}\Hom(L_i,L_{j})\otimes K\right)\end{multline*} indexed according to the weights of the $\T$ action on $T_\calE$. Thus we see that \bes T^+_\calE= \bigoplus_{k=1}^{k=n} (T_\calE)_k = \bigoplus_{k=1}^{k=n} \H^1\left(\bigoplus_{i-j=k}\Hom(L_i,L_{j})\stackrel{ad(\Phi)}{\rightarrow}\bigoplus_{i-j=k-1}\Hom(L_i,L_{j})\otimes K\right).\ees Because of the stability of the Higgs bundle $\calE$ \cite[Theorem 4.3]{hausel-vanishing} we get that the zeroth and second hypercohomology of all the complexes vanish, thus \bes\dim T^k_\calE&=&-\sum_{i-j=k} \chi(\Hom(L_i,L_j)) +\sum_{i-j=k-1} \chi(\Hom(L_i,L_j)\otimes K)\\ &=& -\sum_{i-j=k} (-\ell_i+\ell_j+1-g) + \sum_{i-j=k-1} (-\ell_i+\ell_j + g-1) \\ &=& (2n-2k+1)(g-1)-\ell_k+\ell_{n-k+1} \\ &=& (2k-1)(g-1) + m_{k,n-k+1},\ees where $$m_{a,b}=\left\{ \begin{array}{cc}\sum_{a\leq j< b}^b m_j & {\mbox{ when } a<b}\\  0 & {\mbox{ when } a=b} \\  -\sum_{b\leq j<a}^b m_j & {\mbox{ when } a> b} \end{array} \right.$$
	Then a straightforward computation gives \beq\label{explicit}m_\calE(t)={{ \chi_\T(\Sym(T^{+*}_\calE))\over \chi_\T(\Sym(\calA^*))}}={\prod_{j=\lceil{n \over 2}\rceil}^{{n-1}} (1-t^{j+1})^{m_{n-j,j+1}}\over \prod_{j=1}^{\lfloor {n\over 2} \rfloor} (1-t^j)^{m_{j,n-j+1} }}=\prod_{i=1}^{n-1} \left[\begin{array}{c} n \\ i\end{array}\right]^{m_i}_t.\eeq 
	Here  $$\left[\begin{array}{c} n \\ k \end{array}\right]_t=\prod_{j=1}^k{1-t^{n-j+1} \over 1-t^j }$$ is the quantum binomial coefficient, in particular a polynomial. Incidently, this is the Poincar\'e polynomial of the Grassmanian $\Gr_k(\C^n)$. This coincidence will be further discussed in Subsection~\ref{simple}. 
	
	We thus see that the rational function in \eqref{explicit} is always a polynomial. The polynomiality also follows from the existence of very stable Higgs bundles in each type $(1,1,\dots,1)$  component of $\M^\T$, as proved in Corollary~\ref{type111}. 
\end{remark}

We obtain immediately  
\begin{corollary} \label{mult111} For a fixed point component $F_\calE\in \pi_0(\M^T)$  of type $(1,1,...,1)$ as in Remark~\ref{remark111} the multiplicity of $N_{F_\calE}$ in the nilpotent cone $N$ is $$m_{F_\calE}=m_\calE(1)=\prod_{i=1}^{n-1} 
	{n \choose i}^{m_i}.$$
\end{corollary}

Another consequence is the following
\begin{corollary} For $\calE\in \M^{s\T}_n$ very stable the following are equivalent \begin{enumerate}\item $m_{{F_\calE}}=1$ 		\item $m_{{F_\calE}}(t)=1$  \item $\calE$ is a uniformising Higgs bundle of Remark~\ref{hitchinsection} 
		
		\item $W^+_\calE$ is a section of the Hitchin map $h$ 
	\end{enumerate}
\end{corollary}
\begin{proof} Corollary~\ref{eqmult} implies that $m_{\calE}(t)\in \Z_{\geq0}[t]$ is palindromic with non-negative coefficients  showing $1.\Rightarrow 2.$.  By Definition~\ref{equmult}  $m_{\calE}(t)=1$ implies that there are $n$ distinct weights in the $\T$-module $T^+_\calE$, to match those in $\calA$.  This implies in turn that $\calE$ is of type $(1,\dots,1)$. For a type $(1,\dots,1)$ $\calE$ we will have $m_{\calE}(t)=1$ if all $m_i=0$ from \eqref{explicit}. This only happens for a  uniformising Higgs bundle. This shows $2.\Rightarrow 3.$.  From Remark~\ref{hitchinsection} for a uniformising Higgs bundle   $W^+_\calE$ is a section of the Hitchin map. This shows $3.\Rightarrow 4.$. Finally the second part of Theorem~\ref{multiplicity} implies $4.\Rightarrow1.$.
\end{proof}

\subsection{Intersection of upward flows with generic fibres of $h$}
\label{intersection}

Let ${{a}}=(a_1,\dots,a_n)\in H^0(C;K)\oplus \dots \oplus H^0(C;K^n)=\calA$ be coefficients that define a smooth spectral curve $C_a\subset T^*C$.   Assume that $\div(b)$ for $\calE$ of type $(1,\dots,1)$ avoids the ramification divisor of $\pi_a:C_a\to C$. In other words over any zero $c\in C$ of $b$ we have $|\pi_a^{-1}(c)|$ has $n$ distinct preimages of $c$ in $C_a$.   

We recall that the spectral curve is a subscheme of $T^*C$ defined as $$C_a=\det(\pi^*\Phi-x)^{-1}(0),$$ where $\pi:T^*C\to C$ is the projection and $x\in H^0(T^*C;\pi^*(K))$ is the tautological section.  For $(E,\Phi)\in h^{-1}(a)$ denote by \beq \label{definitionU}U:={\rm coker}(\pi_a^*(E\otimes K^{-1})\stackrel{\pi_a^*\Phi-x}{\to} \pi_a^*(E)).\eeq 

Then by \cite[\S 5.1]{hitchin-stable} and \cite[\S 3.7]{beauville-etal}  we have the following proposition, which we prove for completeness as the proof is not readily available in the literature.  

\begin{proposition}  
	
	\noindent 1. The direct image of the trivial bundle   \beq \label{canonicalBNR} (\pi_a)_*(\calO_{C_a}\stackrel{x}{\to}\pi_a^*(K))\cong \calE_{\calO_C, a}=(E_{\calO_C},\Phi_a)\in W^+_{\calE_0} \eeq   defines  the canonical section  $W^+_{\calE_0}$ of the Hitchin map (see \eqref{uniform}).
	
	\noindent 2. If a Higgs bundle $(E,\Phi)$ is the direct image of a line bundle $U$:
	\beq \label{utoephi}(\pi_a)_*(U\stackrel{x}{\to} U\otimes \pi_a^*(K))=(E,\Phi).\eeq then  on $C_a$ we have the exact sequence \beq\label{exactu}0\to  U\otimes \pi_a^*(K^{-n}) \to \pi_a^*(E\otimes K^{-1})\stackrel{\pi_a^*\Phi-x}{\to} \pi_a^*(E) \to U\to 0. \eeq 
\end{proposition}

\begin{proof}  To prove \eqref{canonicalBNR} we note that the pushforward $(\pi_a)_*(\calO_{C_a})$ is the $\calO_C$-algebra ${\rm Sym}(K^{-1})/{\mathcal I}$ where the ideal sheaf ${\mathcal I}$ is generated by the image of $u:K^{-n}\to {\rm Sym}(K^{-1})$ the sum  of $a_i:K^{-n}\to K^{-n+i}$. Thus we see that as a sheaf $(\pi_a)_*(\calO_C)\cong \calE_{\calO_C}$ and the algebra structure on it induces the Higgs field $\Phi_a$ \eqref{companion}  the companion matrix of $a$.

	For the proof of the second statement we also denote 
	by $$U^\prime:={\rm ker}(\pi_a^*(E\otimes K^{-1})\stackrel{\pi_a^*\Phi-x}{\to} \pi_a^*(E))$$ the kernel. Then the pushforward of the middle square of
	the diagram $$\begin{array}{ccccccccc} 0\to  &U^\prime& \to &\pi_a^*(E\otimes K^{-1})&\stackrel{\pi_a^*\Phi-x}{\to} &\pi_a^*(E) &\to &U& \to 0 \\  & x \downarrow    &&   x \downarrow  &&  x \downarrow     &&   x \downarrow   & \\ 0\to  &U^\prime\otimes \pi_a^*(K)& \to &\pi_a^*(E)&\stackrel{\pi_a^*\Phi-x}{\to} &\pi_a^*(E\otimes K) &\to &U\otimes \pi_a^*(K) & \to 0 \end{array}$$
	is the middle square of the following diagram 
	
	$$\begin{array}{ccccccccc} 0\to  &EK^{-n}& \stackrel{\Psi}{\to }&E_{\calO_C}\otimes E K^{-1}&\stackrel{Id_{E_{\calO_C}} \otimes \Phi-\Phi_a\otimes Id_{EK^{-1}} }{\longrightarrow} &E_{\calO_C}\otimes E &\stackrel{\Xi}{\to }&E& \to 0 \\  & _\Phi \downarrow    &&   _{\Phi_a\otimes Id_E} \downarrow  &&   _{\Phi_a\otimes Id_E} \downarrow     &&   _\Phi \downarrow   & \\ 0\to  &EK^{-n+1}&  \stackrel{\Psi}{\to} &E_{\calO_C}\otimes E&\stackrel{Id_{E_{\calO_C}} \otimes \Phi-\Phi_a\otimes Id_{EK^{-1}} }{\longrightarrow} &E_{\calO_C}\otimes EK &\stackrel{\Xi}{\to} &E K & \to 0 \end{array}.$$ The rest of the diagram is given by the maps
	$$\Psi=\left(\begin{array}{c} \Psi_1\\ \Psi_2 \\ \vdots \\ \Psi_n \end{array}\right):EK^{-n}\to EK^{-1}\oplus EK^{-2}\oplus \dots \oplus EK^{-n}$$  with  $$\Psi_i=\Phi^{n-i}+a_1 \Phi^{n-i-1}+\dots + a_{n-i} Id_{EK^{-n}}:EK^{-n}\to EK^{-i}$$ and $\Xi=(\Xi_1,\Xi_2,\dots,\Xi_n): E\oplus EK^{-1}\oplus \dots \oplus EK^{-n+1}\to E$ is the map given by $\Xi_i=\Phi^{i-1}:EK^{i-1}\to E$. It is straightforward to check that the second  diagram commutes due to the Cayley-Hamilton identity $0=\Phi^n+a_1\Phi^{n-1}+\dots + a_{n} Id_E$ .  Therefore the second diagram is the pushforward of the first. This implies  \eqref{utoephi} and that $U^\prime\cong U\otimes \pi_a^*(K^{-n})$.  The result follows. 
\end{proof}

Now let $\calE$ be, as in Remark~\ref{remark111}, a very stable Higgs bundle of type $(1,1,\dots,1)$. We know from Proposition~\ref{filtrationexists} that $(E,\Phi)\in \W^+_{\calE}$ if and only if there exists a filtration with a full flag of subbundles as in \eqref{fullfiltration} which is compatible with the Higgs field \eqref{compatible}. Recall the maps $b_i$ and $b$ from \eqref{bi} and \eqref{be} respectively. 

Assume that $(E,\Phi)\in W^+_\calE\cap h^{-1}(a).$  Denote by $f_i:\pi_a^*(E_i)\to U$ the composition of the  embedding $\pi_a^*(E_i)\subset \pi_a^*(E)$ and projection $\pi_a^*(E)\to U$ arising from the definition \eqref{definitionU} of $U$. Furthermore, denote by $U_i=\im(f_i)$ the image sheaf.  As a non-trivial  subsheaf of the invertible sheaf $U$ it is itself an invertible sheaf. We have thus constructed a filtration $$U_1\subset U_2\subset \dots\subset U_{n-1}\subset U_{n}=U$$ by invertible subsheaves. Denote  the induced map $\tilde{f}_i: U_i\to U$ and by $\Delta_i:=\div(\tilde{f}_i)$ the divisor of zeroes of $\tilde{f}_i\in H^0(C;U_i^*U)$. Note that the divisor $\Delta_1$ determines $$U=\pi_a^*(E_1)(\Delta_1),$$ and so in turn by \eqref{utoephi} it also determines our Higgs bundle $(E,\Phi)\in W^+_\calE\cap h^{-1}(a)$. We get the chain of divisors on $C$ $$\Delta_1\geq \Delta_2\geq \dots \geq \Delta_{n-1}\geq \Delta_n=0.$$ 

\begin{proposition} \label{vsintersect} Let $\calE\in \M^{s\T}$ be a type $(1,...,1)$ very stable Higgs bundle. Let $a\in \calA$ be such that the spectral curve $C_a\subset K$ is smooth and $\div(b)$ avoids the ramification divisor of $\pi_a$. \begin{enumerate} \item  If $(E,\Phi)\in W^+_\calE\cap h^{-1}(a)$ then the effective divisor $\Delta_i-\Delta_{i+1}\geq 0$ is reduced and supported on  $(n-i)$ preimages in $C_a$ of $\pi_a$ over each zero of $b_i$ in $C$. More precisely if $c\in C$ is a zero of $b_i$ then $\Phi_c$ leaves $(E_i)_c$ invariant and  $\tilde{c}\in \pi_a^{-1}({c})$ is in the support of $\Delta_i-\Delta_{i+1}$ if and only if $\tilde{c}\in K_c$ is not an eigenvalue for $\Phi_c|_{(E_i)_c}:(E_i)_c\to(E_i)_c K_c $. 
		
		\item 	On the other hand for any choice of an $(n-i)$ element subset of $\pi_a^{-1}(c)$ for all zeroes $c$ of all the $b_i$, there is a unique $(E,\Phi)\in W_\calE^+\cap h^{-1}(a)$ such that the corresponding invertible sheaf on the spectral curve $C_a$ has the form $U=\pi_a^*(E_1)(\Delta_1)$ for $\Delta_1$ effective and reduced and supported at our chosen points in $C_a$. 
		
		\item	Finally, any two distinct choices of such $(n-i)$ element subsets of $\pi_a^{-1}(c)$ over all zeroes of $b_i$ for all $i$ gives non-isomorphic Higgs bundles. In particular, $$|W^+_\calE\cap h^{-1}(a)|= \prod_{i=1}^{n-1}{n\choose i}	^{m_i},$$ where $\deg(b_i)=m_i$. \end{enumerate}	\end{proposition}

\begin{proof}
	First we prove that $\Delta_i-\Delta_{i+1}\geq 0$ is supported at precisely $(n-i)$ preimages in $C_a$ over each zero of $b_i$ in $C$ with no multiplicity, i.e. $\Delta_i-\Delta_{i+1}$ is reduced.  
	
	Note that $\Phi(E_i)\subset E_{i+1}K$ thus we have the map $$p_i:=\pi_a^*(\Phi)-x:\pi_a^*(E_iK^*)\to \pi_a^*(E_{i+1}).$$ Let $\overline{\im(p_i)}\leq \pi_a^*(E_{i+1})$ denote the subbundle which is the saturation of the subsheaf $\im(p_i)\subset \pi_a^*(E_{i+1})$. Let $V_{i+1}:=\pi_a^*(E_{i+1})/\overline{\im(p_i)}$ be the quotient bundle and $$g_{i+1}:\pi_a^*(E_{i+1})\to V_{i+1}$$ the quotient map.  Another way to think about $V_{i+1}$ is that it is the quotient of the sheaf ${\rm coker}(p_i)$ by its torsion: $$V_{i+1}= \coker(p_i)/T(\coker(p_i)).$$ As $p_i$ is generically an embedding $V_{i+1}$ is a line bundle. We have a map $$h_i:V_{i}\to V_{i+1}.$$ This can be seen by  noting that the embedding $\pi_a^*(E_i)\subset \pi_a^*(E_{i+1})$ induces a map $\coker(p_i)\to \coker{(p_{i+1})}$ which induces a map on the quotients by torsion. By chasing around definitions we see that \beq \begin{array}{ccc} \pi_a^*(E_{i+1}) &\stackrel{g_{i+1}}{\to} &V_{i+1} \label{diagram} \\  \hookuparrow && \ \ \ \uparrow_{\scriptsize h_{i}} \\\pi_a^*(E_{i}) &\stackrel{g_{i}}{\to}  &V_{i}  \end{array}\eeq is commutative.
	We define $V_1:=\pi_a^*(E_1)$ and $h_1:V_1\to V_2$ as $g_{2}|_{\pi_a^*(E_1)}$ so that the diagram \eqref{diagram} is still commutative for $i=1$. Finally, we have a map $h_{n}:V_{n}\to U$ induced by the map $$\coker(p_{n-1})\to \coker(\pi_a^*(\Phi)-x)=U.$$  This induces the commutative diagram
	
	\beq \label{tricomm}
	\begin{tikzcd}
	& V_n \arrow{dr}{h_n} \\
	\pi_a^*(E) \arrow{ur}{g_n} \arrow{rr}{f_n} && U
	\end{tikzcd}
	\eeq
	
	If we denote by $$h_{\geq i}:=h_{n}\circ \dots \circ h_i : V_i\to U,$$ then we see from the commutativity of
	the diagrams \eqref{diagram} and \eqref{tricomm} that $f_i=h_{\geq i}\circ g_i$. Because $g_i$ is surjective we can deduce 
	 ${\mathrm{im}} (h_{\geq i})={\mathrm{im}}(f_i)=U_i.$ In particular, $h_{\geq i}:V_i\to U_i$ is surjective and thus an isomorphism. 
	 Thus $\Delta_i-\Delta_{i+1}=\div(h_i)$. 
	
	The map $g_{i+1}|_{\pi_a^*(E_{i})}:\pi_a^*(E_{i})\to V_{i+1}$ vanishes at $\tilde{c}=(c,\lambda)\in C_a$ provided that $$\pi_a^*(E_{i})_{\tilde{c}}\subset p_i(\pi_a^*(E_iK^*)_{\tilde{c}}).$$ This happens when
	$$(E_{i})_c\subset (\Phi-\lambda Id_{E})((E_iK^*)_c).$$ As $(E_{i})_c$ and $(E_iK^*)_c$ are of the same dimension this 
	is equivalent to \beq \label{equal}(E_{i})_c= (\Phi-\lambda Id_{E})((E_iK^*)_c).\eeq  This implies $\Phi((E_i)_c)\subset (E_iK)_c$, i.e. that $(E_i)_c$ is $\Phi$-invariant, which implies that $(b_i)_c=0$. As by assumption $c$ avoids the ramification divisor of $\pi_a$ we see that $\Phi_c$ has distinct eigenvalues, i.e. it is regular semisimple.  Finally, when $\lambda$ is not an eigenvalue of $\Phi|_{(E_i)_c}$ then $(\Phi-\lambda Id_{E})$ is injective so gives an isomorphism between   $(E_iK^*)_c$ and $(E_{i})_c$ and thus \eqref{equal} holds.
	
	Thus we see that $g_{i+1}|_{\pi_a^*(E_i)}$ vanishes at $n-i$ preimages of each zero of $b_i$, the ones corresponding to the eigenvectors of $\Phi_c$ not in $(E_i)_c$. From the commutativity of \eqref{diagram} and surjectivity of $g_i$ we see that the same follows for $h_i$. 
	
	We  compute \bes \deg(\Delta_1)=\sum_{i=1}^{n-1}\deg(\Delta_i-\Delta_{i+1}) = \sum_{i=1}^{n-1}\deg(h_i)  \geq \sum_{i=1}^{n-1} m_i (n-i)=\deg(\Delta_1).\ees Indeed, by  assumption $\div(b_i)$ is reduced, $\deg(\div(b_i))=m_i$, so $\deg(h_i)\geq m_i(n-i)$  and we have \bes \sum_{i=1}^{n-1}m_i(n-i)=\sum_{i=1}^{n-1}(-\ell_i+\ell_{i+1}+2g-2)(n-i)&=&-n\ell_1 +\sum_{i=1}^{n-1} \ell_i + n(n-1)(g-1)\\&=&-n\deg(E_1)+\deg(E)+\deg(K^{n\choose 2})\\&=&-\deg\pi_a^*(E_1)+\deg(U)=\deg(\Delta_1).\ees Here  we used $\deg(E)=\deg(U)-\deg(K^{{n\choose 2}})$ for which see \cite[\S 4]{beauville-etal}. 
	It follows that $\Delta_i-\Delta_{i-1}\geq 0$ is reduced. Over a zero of $b_i$ the morphism $h_i:V_i\to V_{i+1}$ is zero at precisely those points which correspond to eigenvectors of $\Phi_c$  not in $(E_i)_c$. This proves the first part of the statement. 
	
	For the second part of Proposition~\ref{vsintersect}, we will construct a Higgs bundle such that $\Delta_1\subset C_a$ is the union of the prescribed choices of an $(n-i)$ tuple in $\pi_a^{-1}(c)$ over every zero $c$ of every $b_i$. The construction is by induction on $\deg(b_i).$ When $\deg(b_i)=0$ then $U=\pi_a^*(E_1)$ is unique and given by the Higgs bundle $\calE_{E_1,a}$ defined   in Remark~\ref{hitchinsection}. 
	
	Recall the notation $\calE_{\delta}$ for a type $(1,\dots,1)$ Higgs bundle from Example~\ref{ex111}. Suppose  that we have constructed a Higgs bundle in $(E,\Phi)\in W^+_{\calE_{\d}}\cap h^{-1}(a)$ for any choice of an $(n-i)$ tuple in $\pi_a^{-1}(c)$ over every zero $c$ of every $b_i$. Let $c\in C$ be a point which avoids the ramification divisor of $\pi_a$ and $b(c)\neq 0$. Denote by $b^\prime_j:=b_j$ and $\delta_j^\prime:=\div(b_j)$ except when $j=i$,  when we let $b^\prime_i:=b_i s_c$ or equivalently $\d^\prime_i:=\d_i+c$. Assume that $\calE_{\d^\prime}$ is still stable.  Then $\Phi_c$ is regular semisimple and we can take an $i$-dimensional $\Phi_c$-invariant subspace $V_i\subset E_c$ corresponding to any $i$ element subset of $\pi_a^{-1}(c)$. Then by Proposition~\ref{modifiedfilt} the Hecke transform  $\calH_{V_i}(E,\Phi)=(E^\prime,\Phi^\prime)$ will be contained in $W^+_{\calE_{\d^\prime}}\cap h^{-1}(a)$, and $b_i(c)=0$ so that $\Phi^\prime_c$ leaves $(E_i)_c$ invariant. By induction the second part of the proposition follows.   
	
	For the last part we note that there are ${n\choose n-i}^{m_i}={n\choose i}^{m_i}$ choices for $\Delta_i-\Delta_{i-1}$. 	This gives us $\prod_{i=1}^{n-1}{n \choose i}^{m_i}$ possibilities for $\Delta_1$.  In theory, two distinct divisors could represent the same line bundle, but in our case, this cannot happen. Namely, a Higgs bundle $(E,\Phi) \in W^+_{\calE_{\d}}\cap h^{-1}(a)$ will have a unique filtration yielding the limiting Higgs bundle $\calE_{\d}$ from Proposition~\ref{filtrationexists}. Then in turn the  filtration  defines a unique subset $\Delta_1\subset C_a$ as above in the proof of part one of the proposition. The result follows.

	
\end{proof}

\begin{remark}When $\calE\in \M^{s\T}$ is very stable then $W^+_\calE\subset \M$ is closed and we have \eqref{generic}. 
	When $\calE$ is of type $(1,\dots,1)$ from Corollary~\ref{mult111} and Proposition~\ref{vsintersect} we get that $|W^+_\calE\cap h^{-1}(a)|=m_\calE$, provided that $C_a$ is smooth and $\div(b)$ avoids the ramification divisor of $\pi_a$. Consequently, in this case $W^+_\calE$ and $h^{-1}(a)$ intersect transversally. 
	
	Similarly as above in Proposition~\ref{vsintersect} one can prove that $|W^+_\calE\cap h^{-1}(a)|<m_\calE$ when $b$ has repeated $0$'s. The second part of Theorem~\ref{multiplicity}  implies that such type $(1,\dots,1)$ Higgs bundles must be  wobbly. This gives an alternative proof of the ``only if " part of Theorem~\ref{mainverystable}.
\end{remark}

\section{Mirror of  upward flows}
\label{mirror}

In this section we consider the mirror of $\calO_{W^+_\calE}$ the structure sheaf of a very stable type $(1,...,1)$ upward flow. We recall that we work with $\GL_n$-Higgs bundles. Because $\GL_n$ is its own Langlands dual we expect the mirror of $\calO_{W^+_\calE}$ to be a coherent sheaf on $\M$, or a complex of sheaves which -- generically at least -- should be a fibrewise Fourier-Mukai transform on the self-dual Jacobian of a smooth spectral curve. We start by recalling these notions. 
\subsection{Jacobians, duality, Fourier-Mukai transform}
The Jacobian $J:=J_0(C)$ of a smooth complex projective curve $C$ is the moduli space of degree $0$ line bundles on $C$. It is an Abelian variety with multiplication of line bundles. 
Its dual Abelian variety is defined as $\Pic^0(J)$ the moduli space of translation invariant line bundles on $J$. One can identify $\Pic^0(J)$ with $J$ in two different ways yielding the same isomorphism between the two. 

First, we fix a point $c_0\in C$ and define the {\em Abel-Jacobi map} as $\alpha_{c_0}:C\to J$ by $$\alpha_{c_0}(c):=\calO(c-c_0).$$ Then it turns out \cite[Lemma 11.3.1]{birkenhake-lange} that the  restriction map \beq\label{ajiso} \alpha_{c_0}^*:\Pic^0(J)\to J\eeq is an isomorphism. 



Another isomorphism between $\Pic^0(J)$ and $J$ is given by $\calL$ the line bundle of the Theta divisor $\Theta$.  One defines \beq \label{philiso} \phi_\calL: J\to \Pic^0(J)\eeq by $L\mapsto t_L^*(\calL) \otimes \calL^*$, where $t_L:J\to J$ is translation by $L$. The map \eqref{philiso} is an isomorphism as $\calL$ defines  a principal polarization on $J$. It follows from \cite[Proposition 11.3.5]{birkenhake-lange} that $-\phi_\calL=(\alpha^*_{c_0})^{-1}$. Thus $-\phi_\calL$ and $\alpha^*_{c_0}$ defines the same isomorphism between $J$ and $\Pic^0(J)$. In particular, $\alpha^*_{c_0}$ is independent of the choice of $c_0\in C$.  We will identify $J$ with $\Pic^0(J)$ using this isomorphism below. 

We have the {\em Poincar\'e bundle} $\bP$ on $J\times \Pic^0(J)$ unique with the properties that \beq \label{poiniden} \bP|_{J\times \{L\}}\cong L\eeq for every $L\in \Pic^0(J)$ and $$\bP|_{\{\calO_C\}\times \Pic^0(J)}\cong \calO_{\Pic^0(J)}.$$
For $L\in J$ we will denote the line bundle $$\bP_{L}:= \bP|_{\{L\}\times \Pic^0(J)}$$ on $\Pic^0(J)\cong J$, where we identified  $\Pic^0(J)$ and $J$ with $\alpha_{c_0}^*$. This translation invariant line bundle on $\Pic^0(J)$ is characterised by the property that \beq\label{char}\alpha^*_{c_0}(\bP_L)\cong L.\eeq
In particular, for ${L_1},L_2\in J$ we have \beq \label{tensorp}\bP_{{L_1}L_2}\cong  \bP_{{L_1}}\otimes \bP_{L_2},\eeq and $$\bP_{\calO_C}=\calO_{J}.$$
Using the Poincar\'e bundle $\bP$ on $J\times \Pic^0(J)\cong J\times J$  we can define the Fourier-Mukai transform \cite{mukai}. For a bounded complex $\calF^\bullet \in D_{coh}^b(J)$ of coherent sheaves on $J$ we denote by $$S(\calF^\bullet):=(\pi_2)_*(\pi_1^*(\calF^\bullet)\otimes^L \bP)\in D^b_{coh}(J),$$ where $\pi_1$ and $\pi_2$ are the projections of $J\times J$ to the two factors. Then $S$ is called
the {\em Fourier-Mukai transform }of $\calF^\bullet$. For example we have $$S(\calO_{L})=\bP_L.$$ 

We recall three properties of the Fourier-Mukai transform. First by \cite[Theorem 3.13.1]{mukai} for any $\calF^\bullet \in D^b_{coh}(J)$ \beq \label{skewinv} S(S(\calF^\bullet))=(-1)^*(\calF^\bullet )[-g], \eeq where $-1:J\to J$ denotes the inverse map on the abelian variety $J$ and $g=\dim(J)$ is the genus of $C$.   The  second \cite[Theorem 3.13.4]{mukai} \beq \label{skewdual}  S(\calF^{\bullet\vee}) \cong (-1)^*S(\calF^\bullet)^\vee [g].\eeq Here $(\calF^\bullet)^\vee:=R{\mathcal Hom}(\calF^\bullet,\calO_J)$ is the derived dual complex. Finally, by proper base change and \eqref{poiniden} we see that the  restriction of the Fourier-Mukai transform to the identity satisfies \beq \label{fmiden} S(\calF^\bullet)|_{\{\calO_C\}}\cong R\pi_*(\calF^\bullet), \eeq
where $\pi:J\to \{\calO_C\}$ is the projection to the point. 

We can pull back \beq \label{univlambda}\L:=(\alpha_{c_0}\times \alpha_{c_0}^{*-1})^*(\bP)\eeq the Poincar\'e bundle on $J\times \Pic_0(J)$ to $C\times J$ via the map $$\alpha_{c_0}\times \alpha_{c_0}^{*-1}:C\times J\to J\times \Pic_0(J)$$ to get a universal line bundle on $C\times J$ satisfying \beq \label{univl}\L|_{C\times \{L\}}\cong L\eeq for all $L\in J$ and normalized
by an isomorphism \beq\label{normalize} \bbbeta:\L|_{\{c_0\}\times J}\stackrel{\cong}{\to}\calO_{J}.\eeq

\begin{remark} \label{framedjacobian} As a universal line bundle on $C\times J$  the line bundle $\L$ is not unique as we can tensor it with the pull back of any line bundle on $J$, and the universal property still holds. However
	$(\L,\bbbeta)$ is the unique universal object for the functor assigning to a scheme $T$  families of framed line bundles $(L,\beta)$, where $L$ is a line bundle on $C\times T$ and $\beta:L|_{\{c_0\}\times T}\cong \calO_T$. 

If $c^\prime_0\in C$ is another point, then $\L\otimes \pi_2^*\left(\bP_{\calO(c_0-c^\prime_0)}\right)|_{\{c^\prime_0\}\times J}\cong \calO_J$.  This is because  $\L|_{\{c^\prime_0\}\times J}\cong (\alpha_0^{*-1})^*\bP_{\calO(c^\prime_0-c_0)}$ from \eqref{char}. Therefore the universal bundle normalised at $c^\prime_0$ is given by \beq \label{changenorm} \L^\prime=\L\otimes \pi_2^*\left(\bP_{\calO(c_0-c^\prime_0)}\right). \eeq 
\end{remark}


Finally, we denote by $\pi:\mathbf{C}\to C$  a finite morphism of smooth projective curves, by  $\tilde{J}$  the Jacobian of $\mathbf{C}$ and  $\tilde{\bP}$ denotes the Poincar\'e bundle on $\tilde{J}\times \tilde{J}$. Then, in this setting,  we have the norm map ${\mathrm{Nm}}_\pi:\tilde{J}\to J$ by \beq \label{norm}L\mapsto \det(\pi_*(M))\otimes \det(\pi_*(\calO_{\mathbf{C}}))^{-1},\eeq which is a group homomorphism. 
\begin{lemma} \label{pullnorm} For $M\in J$ we have $\pi^*(M)\in \tilde{J}$ and $$\tilde{\bP}_{\pi^*(M)}\cong {\mathrm{Nm}}_{\pi}^*(\bP_M).$$ 
\end{lemma}
\begin{proof}  Fix $\tilde{c}_0\in \mathbf{C}$ a lift such that $c_0=\pi(\tilde{c}_0)$. Then by \cite[p. 331]{birkenhake-lange} the diagram $$\begin{array}{lcl} \Pic_0(\tilde{J}) &\stackrel{\alpha_{\tilde{c}_0}^*}{\rightarrow} & \tilde{J}\\  \!\!\!\!\!\!{{\mathrm{Nm}}}_{\pi}^*\uparrow && \uparrow  \pi^*\\ \Pic_0(J) &\stackrel{\alpha^*_{c_0}}{\rightarrow} & {J} \end{array}$$ commutes. Thus $$\alpha_{\tilde{c}_0}^*({\mathrm{Nm}}_{\pi}^*(\bP_M))=\pi^*(\alpha_c^*(\bP_M))=\pi^*(M).$$ By \eqref{char} this implies the statement. 
\end{proof}

\subsection{Fourier-Mukai transform of $\calO_{W_\calE^+}$}

In this section we will fix $d=-n(n-1)(g-1)$ the degree of our Higgs bundles. This will ensure that the line bundle $U$ defined in \eqref{definitionU} on a spectral curve corresponding to our Higgs bundle will have degree $0$. In particular, the canonical uniformising Higgs bundle $\calE_0\in \M$ and the canonical Hitchin section $W^+_{\calE_0}\subset \M$ will be contained in our Higgs moduli space $\M$.   

Consider now a very stable type $(1,\dots,1)$ Higgs bundle 
$\calE_\delta\in \M^{s\T}$. That means  $\delta=(\delta_0,\delta_1,\dots,\delta_{n-1}) \in J_\ell(C)\times C^{[m_1]}\times \dots \times C^{[m_{n-1}]}\subset \M$ such that $\calE_{\d}$ is very stable. 

\subsubsection{Construction of $\Lambda_\delta$ from universal bundles}
\label{lambdadelta}
The proposed mirror of $\calO_{W^+_{\calE_\delta}}$ will be a vector bundle $\Lambda_\delta$ on $\M^s$ constructed from universal Higgs bundles whose existence was proved in \cite{simpson}. We will have to overcome two difficulties in the construction. One is that the individual universal Higgs bundles will only be sheaves in a twisted sense, and they will become untwisted vector bundles only if an appropriate number of them is tensored together. Thus we will first construct universal bundles in an \'etale covering of $\M^s$ and only when an appropriate number will be tensored together the obstruction to glue will vanish. The second difficulty will be in that the universal Higgs bundles are only unique up to tensoring with line bundles pulled back from $\M^s$. Thus in order to construct well defined universal bundles we will have to normalise them appropriately.

\'Etale locally $U_i\to \M^s$ we can define  \cite[Theorem 4.7.(4)]{simpson1} universal Higgs bundles $(\bE_{U_i},\bPhi_{U_i})$ on $U_i\times C$.  More precisely,  choosing a point $c_0\in C$, we use the framed moduli space $\calR^s$ \cite[Theorem 4.10]{simpson1}. Then  $C\times \calR^s$ carries  a unique universal framed Higgs bundle $(\bE_{\calR^s},\bPhi_{\calR^s},\bbbeta_{\calR^s})$, where the map $\bE_{\calR^s}$ is a rank $n$ vector bundle on $C\times \calR^s$, the universal Higgs field $\bPhi_{\calR^s}\in H^0(C\times\calR^s; \End(\bE_{\calR^s}\otimes K))$ and the universal framing $\bbbeta_{\calR^s}:\bE_{\calR^s}|_{\{c_0\}\times \calR^s}\stackrel{\cong}{\to} \calO_{\calR^s}^n$. We can take slices $U_i\subset \calR^s$ of the $\PGL_n$-principal bundle $\calR^s\to \M^s$ by  Luna's slice theorem \cite[Theorem III.1]{luna} such that $U_i\to \M^s$ is \'etale and $\coprod U_i \to \M^s$ covers $\M^s$. Then $(\bE_{U_i},\bPhi_{U_i})$ is just the Higgs bundle part of the family of framed Higgs bundles $(\bE_{U_i},\bPhi_{U_i},\bbbeta_{U_i})$ which is obtained as the pull back of the universal framed Higgs bundle $(\bE_{\calR^s},\bPhi_{\calR^s},\bbbeta_{\calR^s})$ from $C\times \calR^s$.  

The local universal framings $\bbbeta_{U_i}$  can be  used to normalise our local universal Higgs bundles. For this denote the {\em normalized determinant} $$\Nm(E):=\det(E)\otimes K^{n \choose 2}$$ for a vector bundle (or a family of vector bundles) on $C$. Fixing an isomorphism $K_{c_0}\cong \C$, a framing $\beta:E_{c_0}\to 
\C^n$ gives rise to a framing $\Nm(\beta):\Nm(E)|_{c_0}\cong \C$ of the line bundle $\Nm(E)$.  Then $(\Nm(\E_{U_i}),\Nm(\bbbeta_{U_i}))$ is a $U_i$-family of framed degree $0$ line bundles on $C$, thus unique from Remark~\ref{framedjacobian}. This gives our local normalisation  \beq \label{determinant}\phi_i:\Nm(\bE_{U_i})\stackrel{\cong}{\to}\Nm^*(\L),\eeq where $\L$ is the line bundle of \eqref{univlambda} and $\Nm:\M^s\to J$ is  given by \beq \label{normdet} \Nm(E,\Phi):=\Nm(E).\eeq

We will now attempt to glue our normalised local Higgs bundles together. On the overlap $U_i\times_{\M^s} U_j$  of two of our \'etale covers \cite[(4.2)]{hausel-thaddeus} implies that we get  a  line bundle $L_{ij}\in {\rm Pic}(U_i\times_{\M^s} U_j)$  which satisfies  $$(\bE_{U_i},\bPhi_{U_i})\cong (L_{ij}\bE_{U_j},\bPhi_{U_j}).$$ Passing to a refinement of the \'etale covering if necessary we can assume that all $L_{ij}$ are trivial. It is obvious how to do this in the analytic topology, the \'etale version which we need is  more complicated.   Namely, we can trivialise $L_{ij}$ on $U_i\times_{\M^s} U_j$ even on a Zariski covering, and then we can use \cite[Theorem 4.1]{artin} ( see also \cite[Lemma III.2.19]{milne}) to refine the original coverings so that the connecting line bundles will trivialise on the overlaps. 

This way we can assume that $$(\bE_{U_i},\bPhi_{U_i})\cong (\bE_{U_j},\bPhi_{U_j})$$ on $U_i\times_{\M^s} U_j$ for all ${ij}$. Taking such an isomorphism $$\phi_{ij}: \bE_{U_i} \stackrel{\cong}{\to} \bE_{U_j}$$ and recalling the fixed isomorphisms \eqref{determinant} we consider $$\phi_j\circ \Nm(\phi_{ij}) \circ \phi_i^{-1}:\Nm^*(\L)\to \Nm^*(\L).$$ By passing to an \'etale cover if needed we can assume that this function has an $n$th root (for this we use again \cite[Theorem 4.1]{artin}), and thus we can rescale $\phi_{ij}$ so that we can assume that $$\phi_j\circ \Nm(\phi_{ij}) \circ \phi_i^{-1}={\mathrm{Id}}_{\Nm^*(\L)}.$$ The isomorphisms $\phi_{ij}$ with this property are unique only up to multiplying with an $n$th root of unity, but we can make sure that $\phi_{ji}=\phi_{ij}^{-1}$.  
Thus $\phi_{jk}\circ\phi_{ij}$ agrees with $\phi_{ik}$ up to an $n$th root of unity, which we denote by $\theta_{ijk}\in \bbmu_n\subset \C^\times$ so that $$\phi_{jk}\circ\phi_{ij}=\theta_{ijk}\phi_{ik}.$$
We compute $$\theta_{ijk} \theta_{kli}{\mathrm{Id}}_{\bE_{U_k}}= \theta_{ijk}\phi_{ik}\circ \theta_{kli}\phi_{ki}=\phi_{jk}\circ\phi_{ij}\circ \phi_{li}\circ\phi_{kl}\in H^0(U_i\times_{\M^s}U_j\times_{\M^s}U_k\times_{\M^s}U_l;\End(\bE_{U_k}))$$ and similarly  
$$\theta_{jkl} \theta_{lij}Id_{\bE_{U_l}}= \phi_{kl}\circ\phi_{jk}\circ \phi_{ij}\circ\phi_{li}\in H^0(U_i\times_{\M^s}U_j\times_{\M^s}U_k\times_{\M^s}U_l;\End(\bE_{U_l})).$$ 
These two scalar endomorphisms are conjugate by $\phi_{kl}$ and therefore the scalars agree: $$\theta_{ijk} \theta_{kli}=\theta_{jkl} \theta_{lij}.$$ It follows that the $n$th roots of unity $\theta_{ijk}$ form a {\v{C}}ech $2$-cocycle $\theta$, thus define a cohomology class $[\theta]\in H^2(\M^s;\muhat_n)$, that is a $\muhat_n$-gerbe. By definition \cite[\S 1.2]{caldararu} the data of the vector bundles $\bE_{U_i}$ and  $\phi_{ij}$ constitute a $\theta$-twisted vector bundle $\bE$ on $\M^s\times C$, when $\theta$ is considered a  \v{C}ech $2$-cocycle for the sheaf $\calO_{\M^s}^*$ via the embedding $\bbmu_n\subset \C^\times$. We can multiply $\phi_{ij}$ by $\alpha_{ij}\in \muhat_n$ for a $1$-cocycle representing $[\alpha] \in H^1(\M^s; \bbmu_n)$, i.e. an order $n$ line bundle, then all of the above will hold. Thus the twisted bundle $\bE$ constructed above  is well-defined up to tensoring with an order $n$ line bundle on $\M^s$  but normalized such that \beq\label{normalise}\Nm(\bE)\cong \Nm^*(\L).\eeq We also note that the $2$-cocycle $\theta$ does not change if we rescale with the $1$- cocycle $\alpha$ as above. 

For a rank $n$ vector bundle $E$ and line bundle $M$ we have  $\Lambda^{i}(ME)=M^i\Lambda^i(E)$. It follows that $\Lambda^i(\bE)$ is naturally a $\theta^i$-twisted sheaf. For  $(\ell,m_1,\dots,m_{n-1})\in \Z\times \Z_{\geq 0}$  and  $\d=(\delta_0,\delta_1,\dots,\delta_{n-1})\in J_\ell(C)\times C^{[m_1]}\times\dots\times C^{[m_{n-1}]}$  with $\delta_0=c_{01}+\dots+c_{0m_0}$ where $ c_{0j}\in C$ or $-c_{0j}\in C$ is a point in $C$ and  $\delta_i=c_{i1}+\dots + c_{im_i}\in C^{[m_i]} $ for points $c_{ij}\in C$  we consider~\footnote{It is easy to check \eqref{normdeterminant} that $\otimes_{j=1}^{m_0}\Lambda^n(\E_{c_{0j}})\cong \Nm^*\left(\bP_{\calO(\delta_0-\ell c_0)}\right)$ depends only on the line bundle $\calO(\delta_0)$ and not on the choice of divisor $\delta_0=c_{01}+\dots+c_{0m_0}$.}   \beq\label{bigprod}\Lambda_\delta:=\bigotimes_{i=0}^{n-1}\bigotimes_{j=1}^{m_i} \Lambda^{n-i}(\bE_{c_{ij}}), \eeq  where $\bE_{c_{ij}}=\bE|_{\M^s\times{\{c_{ij}\}}}$ when $c_{ij}\in C$ and $\bE_{c_{0j}}=\bE^*|_{\M^s\times{\{-c_{0j}\}}}$, when $-c_{0j}\in C$. Then for a line bundle $M$ on $\M^s$ we have \beq \label{twistedM}\bigotimes_{i=0}^{n-1}\bigotimes_{j=1}^{m_i} \Lambda^{n-i}(M\bE_{c_{ij}})=M^{D}\bigotimes_{i=0}^{n-1}\bigotimes_{j=1}^{m_i}  \Lambda^{n-i}(\bE_{c_{ij}}),\eeq where \beq \label{D}D=n\ell+\sum_{i=1}^n(n-i)m_i.\eeq Thus $\Lambda_\delta$ in \eqref{bigprod} is a $\theta^{D}$-twisted sheaf on $\M^s$. In particular, when $n|D$ the $2$-cocycle  $\theta^D=1$ is trivial, thus in this case \eqref{bigprod} descends as a vector bundle to $C\times\M^s$. Additionally, when $n|D$ the ambiguity of tensoring $\E$ with an order $n$ line bundle disappears, and so \eqref{bigprod} is a well-defined vector bundle on $\M^s$, which is normalised such that  \beq\label{bignorm}\det(\Lambda_\delta)=\det\left(\bigotimes_{i=0}^{n-1}\bigotimes_{j=1}^{m_i} \Lambda^{n-i}(\bE_{c_{ij}}))\right) \cong \Nm^*\left(\bigotimes_{i=0}^{n-1}\bigotimes_{j=1}^{m_i} \L_{c_{ij}}^{\bino{n-1}{n-i-1}}\right).\eeq 

To study the dependence of the vector bundle $\Lambda_\delta$ on the base point $c_0\in C$ let  $c^\prime_0\in C$ be another point. The twisted vector bundle $\bE^\prime:=\bE \otimes \Nm^*(\pi_2^*\bP_{L})$, where $L^n\cong \calO(c_0-c_0^\prime)$, will be a twisted universal bundle normalised at $c^\prime_0$ in the sense that $\Nm(\bE^\prime)=\Nm^*(\L^\prime),$ where $\L^\prime$ is the universal line bundle \eqref{changenorm} normalised at $c^\prime_0$. When $D=0$, which we are assuming in this section, we see from \eqref{twistedM} that $\Lambda_\delta$ does not depend on the choice of the base point $c_0\in C$. 

\subsubsection{Fourier-Mukai transform over the good locus}

We shall compute the relative Fourier-Mukai transform of $\calO_{W_{\calE_\delta}^+}$ over a certain open subset of $\calA$. Denote by $\mathbf{C}\subset T^*C\times \calA$ the universal spectral curve. 
We denote by $\calA^{\#}\subset \calA$ the open subset of characteristic equations $a\in\calA^{\#}$ for which the following  properties hold \begin{enumerate} \item $C_a$ is smooth \item $\div(b)$ avoids the ramification divisor of the spectral cover $\pi_a:C_a\to C$  
	.\end{enumerate} We let $$\mathbf{C}^{\#}\subset T^*C\times \calA^{\#}$$ denote the family of spectral curves over $\calA^{\#}$. By the BNR correspondence we can identify $\M^{\#}:=h^{-1}(\calA^{\#})$ with the relative Jacobian of the family $\mathbf{C}^{\#}\to \calA^{\#}$: $$\M^{\#}\cong J({\mathbf{C}}^{\#}/{\calA}^{\#}).$$

We shall denote the relative Poincar\'e bundle on
$$J({\mathbf{C}}^{\#}/{\calA}^{\#})\times_{\calA^{\#}} J({\mathbf{C}}^{\#}/{\calA}^{\#})$$
by $\tilde{\bP}$. 
In this section
we will determine the relative Fourier-Mukai transform $$S(\calO_{W_{\calE_\delta}^+}|_{\M^\#})=(\pi_2)_*(\pi_1^*(\calO_{W_{\calE_\delta}^+}|_{\M^\#})\otimes^L \tilde{\bP})$$ over $\calA^{\#}$.

Recall that $\delta=(\delta_0,\delta_1,\dots,\delta_{n-1}) \in J_\ell(C)\times C^{[m_1]}\times \dots \times C^{[m_{n-1}]}\subset \M$ such that $\calE_{\d}$ is very stable. In other words $\delta_1+\dots + \delta_{n-1}$ is reduced, and we denote $\delta_i=c_{i1}+\dots+c_{im_i}$, for points $c_{ij}\in C$ or $-c_{0j}\in C$.  We also introduce the notation $$\calL_{\d}:=\calO_{W^+_{\calE_{\d}}}$$ for the structure sheaf of the Lagrangian upward flow of $\calE_{\d}$. We recall that in this section we set $d=-n(n-1)(g-1)$ for our Higgs bundles and clearly $d=D-n(n-1)(g-1)$ of \eqref{D} thus $D=0$ and 
$$\Lambda_{\d}:= \bigotimes_{i=0}^{n-1}\bigotimes_{j=1}^{m_i} \Lambda^{n-i}(\bE_{c_{ij}})$$ constructed in Section~\ref{lambdadelta} together with the normalisation \eqref{bignorm} is  a well-defined vector bundle on $\M^s$. With this notation we have

\begin{theorem} \label{fouriermukai}   The relative Fourier-Mukai transform over 
	$\calA^\#$  satisfies 
	\beq \label{mirrorl} S(\calL_{\d}|_{\M^\#})=\Lambda_{\d}|_{\M^\#}.\eeq
\end{theorem}


\begin{proof} First we  check the theorem fibrewise along $h$.    We let $a\in \calA^\#$ so that   $C_a$ is smooth, $\div(b)$ avoids the ramification divisor of the spectral cover $\pi_a:C_a\to C$ and  where $W_\calE^+\cap h^{-1}(a)$ is transversal.
	
	As $\calE_\delta$ is a very stable  type $(1,\dots,1)$ Higgs bundle, we have from Theorem~\ref{mainverystable}  that $b=b_{n-1}\circ \dots \circ b_1$ has no repeated $0$. We will denote by $c_{i1},\dots,c_{im_i}$ the $m_i$ zeroes of $b_i$. For $1\leq i\leq n-1$ and $1\leq j \leq m_i$  denote by $\tilde{c}^{1}_{ij},\dots,\tilde{c}^{n}_{ij}\subset C_a$ the $n$ points in the preimage $\pi_a^{-1}(c_{ij})$ of the zero $c_{ij}$ of $b_i$.
	From Proposition~\ref{vsintersect} we see that  $L\in W_\calE^+\cap h^{-1}(a)\cong W_\calE^+ \cap J(C_a)$ is a degree $0$ line bundle on $C_a$ of the form $L\cong \pi^*(E_1)(\Delta_1)$ where $\Delta_1=\sum_{i=1}^{n-1} \Delta_i-\Delta_{i+1}$. We know all possibilities for the divisor $\Delta_i-\Delta_{i+1}$ on $C_a$; namely it is effective and reduced and it has the form $$\Delta_i-\Delta_{i+1}=\sum_{j=1}^{m_i} \tilde{c}^{k^1_{ij}}_{ij}+\dots+\tilde{c}^{k^{n-i}_{ij}}_{ij}$$ where  $1\leq k^1_{ij}< \dots < k^{n-i}_{ij}\leq n$. Thus we see that \beq \label{sumsky}\calO_{W_\calE^+\cap h^{-1}(a)}=\bigoplus_{\{1\leq k^1_{ij}< \dots < k^{n-i}_{ij}\leq n\}^{1\leq j \leq m_i}_{1\leq i\leq n}} \calO_{\pi^*(E_1)\left(\sum_{i=1}^{n-1}\sum_{j=1}^{m_i} \tilde{c}^{k^1_{ij}}_{ij}+\dots+\tilde{c}^{k^{n-i}_{ij}}_{ij}\right)}\eeq is a direct sum of skyscraper sheaves. 
	Fixing a lift $\tilde{c}_0\in C_a$ of $c_0\in C$ we can compute the Fourier-Mukai transform of \eqref{sumsky} as follows:
	\bes S(\calO_{W_\calE^+\cap h^{-1}(a)})&=& \bigoplus_{\{1\leq k^1_{ij}< \dots < k^{n-i}_{ij}\leq n\}^{1\leq j \leq m_i}_{1\leq i\leq n}} S\left(\calO_{\pi^*(E_1)\left(\sum_{i=1}^{n-1}\sum_{j=1}^{m_i} \tilde{c}^{k^1_{ij}}_{ij}+\dots+\tilde{c}^{k^{n-i}_{ij}}_{ij}\right)}\right)\\
	&=& \bigoplus_{\{1\leq k^1_{ij}< \dots < k^{n-i}_{ij}\leq n\}^{1\leq j \leq m_i}_{1\leq i\leq n}} S\left(\calO_{\pi^*(E_1)\left((\sum_{i=1}^n(n-i)m_i)\tilde{c}_0+\sum_{i=1}^{n-1}\sum_{j=1}^{m_i} \tilde{c}^{k^1_{ij}}_{ij}-\tilde{c}_0+\dots+\tilde{c}^{k^{n-i}_{ij}}_{ij}-\tilde{c}_0\right)}\right)\\ &=&\bigoplus_{\{1\leq k^1_{ij}< \dots < k^{n-i}_{ij}\leq n\}^{1\leq j \leq m_i}_{1\leq i\leq n}} \tilde{\P}_{\pi^*(E_1)(-n\ell \tilde{c}_0)}\otimes  \bigotimes_{i=1}^{n-1}\bigotimes_{j=1}^{m_i} \tilde{\P}_{( \tilde{c}^{k^1_{ij}}_{ij}-\tilde{c}_0)}\otimes \dots \otimes \tilde{\P}_{(\tilde{c}^{k^{n-i}_{ij}}_{ij}-\tilde{c}_0)}\\
	&=&  \tilde{\P}_{\pi^*(E_1)(-n\ell \tilde{c}_0)}\otimes \bigotimes_{i=1}^{n-1}\bigotimes_{j=1}^{m_i} \Lambda^{n-i} \left(\tilde{\P}_{(\tilde{c}^{1}_{ij}-\tilde{c}_0)}\oplus \dots \oplus \tilde{\P}_{(\tilde{c}^{n}_{ij}-\tilde{c}_0)} \right) 
	\ees
	
	We see that $$\tilde{\P}_{(\tilde{c}^{k}_{ij}-\tilde{c}_0)}=\tilde{\L}|_{J(C_a)\times \{\tilde{c}^{k}_{ij}\}},$$ where $\tilde{\L}$ is a universal Poincar\'e bundle on $J(C_a)\times C_a$ normalized at $\tilde{c}_0$ as defined in \eqref{univlambda}. Furthermore, \beq\label{universal}\tilde{\P}_{(\tilde{c}^{1}_{ij}-\tilde{c}_0)}\oplus \dots \oplus \tilde{\P}_{(\tilde{c}^{n}_{ij}-\tilde{c}_0)}=\pi_*(\tilde{\L})|_{h^{-1}(a)\times\{c_{ij}\}},\eeq and $(\pi_*(\tilde{\L}),\bPhi_{h^{-1}(a)})$ with $\bPhi_{h^{-1}(a)}:=\pi_*(x:\tilde{\L}\to \tilde{\L}\otimes \pi^*(K))$ is a universal Higgs bundle on $h^{-1}(a)\times C$. 
	
	\begin{lemma}\label{detpull} For  $L\in J(C_a)$ let ${\mathrm{Nm}}:J(C_a)\to J(C)$ given by ${\mathrm{Nm}}(L)=\det(\pi_*(L))\otimes K^{n\choose 2}$ the usual norm map. Then on $J(C_a)\times C$  we have $$\Nm((\pi_a)_*(\tilde{\L}))={\mathrm{Nm}}^*(\L),$$ where $\tilde{\L}$ is a universal line bundle on $J(C_a)\times \tilde{C_a}$ normalized at $\tilde{c}_0\in \tilde{C_a}$ and $\L$ is a universal line bundle on $J(C)\times C$ normalized at $c_0=\pi_a(\tilde{c}_0)$. 
	\end{lemma}
	\begin{proof} For  $L\in J(C_a)$ we have $${\mathrm{Nm}}^*(\L)|_{\{L\}\times C}={\mathrm{Nm}}(L)=\Nm(\pi_*(\tilde{\L}))|_{\{L\}\times C}
		.$$
		
		On the other hand for $c\in C$ we have $${\mathrm{Nm}}^*(\L)|_{J(C_a)\times\{c\} }={\mathrm{Nm}}^*(\L|_{J(C)\times\{c\}})={\mathrm{Nm}}^*(\P_{(c-c_0)}).$$
		While for the other side $$\Nm(\pi_*(\tilde{\L}))|_{J(C_a)\times\{c\} }=\det(\pi_*(\tilde{\L})|_{J(C_a)\times\{c\} })=\det\left(\tilde{\P}_{(\tilde{c}_1-\tilde{c}_0)}+\dots+\tilde{\P}_{(\tilde{c}_n-\tilde{c}_0)}\right)=\tilde{\P}_{\pi^*(c-c_0)}$$ where $\pi^{-1}(c)=\{\tilde{c}_1,\dots,\tilde{c}_n\}$.
		
		Now Lemma~\ref{pullnorm} implies that  $$\det(\pi_*(\tilde{\L}))|_{J(C_a)\times\{c\} }= {\rm det}^*(\L)|_{J(C_a)\times\{c\} }.$$
		
		The result follows. 
	\end{proof}
	The norm map ${\mathrm{Nm}}:J(C_a)\to J(C)$ in Lemma~\ref{detpull} agrees with the one restricted from
	the normalised determinant map in \eqref{normdet} to $h^{-1}(a)\cong J(C_a)$. Thus $\Nm(\pi_*(\tilde{\L}))=\Nm^*({\L})$ and therefore  $\pi_*(\tilde{\L})\cong \bE$ on $h^{-1}(a)$. Furthermore $$ \tilde{\P}_{\pi^*(E_1)(-n\ell \tilde{c}_0)}\cong \Nm^*(\P_{E_1(-\ell c_0)})$$ from Lemma~\ref{pullnorm}. By the normalisation \eqref{normalise}, the definition \eqref{univlambda} and $E_1\cong\calO(\delta_0)=\calO(c_{01}+\dots+c_{0m_0})$   we get \beq\label{normdeterminant}\Nm^*(\P_{E_1(-\ell c_0)})\cong \bigotimes_{j=1}^{m_0} \Lambda^{n}(\bE_{c_{0j}}).\eeq Thus \eqref{universal} and Lemma~\ref{detpull} verify Theorem~\ref{universal} on $h^{-1}(a)$.
	
	By noting that   \'etale-locally on $\calA^\#$  we have sections of the universal spectral curve corresponding to the preimages of zeroes of $b$ and $c_0$, we see that the same argument can be repeated relative to \'etale neighbourhoods $U_i\to \calA^\#$, proving Theorem~\ref{fouriermukai} \'etale-locally, which as argued above glue together to prove Theorem~\ref{fouriermukai}.
\end{proof}

\begin{remark} As a sanity check we can see what this gives for $\delta=0$. Then $\calL_0=\calO_{W^+_0}$ is the structure sheaf of the canonical Hitchin section. We expect its mirror to be the structure sheaf of the whole space $\M^s$. Indeed, choosing the representative $0=c+-c$ for the trivial divisor on $C$ we get $\Lambda_0=\Lambda^n(\bE_{c})\otimes \Lambda^n(\bE_{-c})=\Lambda^n(\bE_{c})\otimes \Lambda^n(\bE^*_{c})\cong \calO_{\M^s}$. 
	
	More generally, we can determine $\Lambda_\delta$ in the case when $\delta_1=\dots=\delta_{n-1}=0$ and $\ell=0$, i.e. when $\calE_\delta$ is a uniformising Higgs bundle and $\calL_\delta=W^+_\delta$ is the structure sheaf of a Hitchin section. We see from \eqref{normdeterminant} that $\Lambda_\delta=\Nm^*(\bP_{\calO(\delta_0)})$ is a line bundle on $\M^s$ pulled back from the Jacobian $J$. 
\end{remark}

\begin{remark} \label{remarkdual} Vector bundle duality defines a natural  involution \beq \label{iota} \begin{array}{cccc} \iota:&\M&\to &\M\\ &(E,\Phi)&\mapsto &(E^*K^{1-n},\Phi^T)\end{array}.\eeq where $\Phi^T$ is the induced action  of $\Phi$ on $E^*$. This is well-defined as $\deg(E^*K^{1-n})=-\deg(E)-n(n-1)2(g-1)=-n(n-1)(g-1)$. From the exact sequence \eqref{exactu} we see that $(\pi_a)_*(U^*) =\iota(E,\Phi)$ and therefore $\iota$ restricts to the abelian variety $h^{-1}(a)\cong J(C_a)$ as the inverse map $-1$. If $\calE\in \M$ we will denote $\calE^\iota:=\iota(\calE)$, then $\iota(W^+_\calE)=W^+_{\calE^\iota}$.  
	
	For example for a type $(1,\dots,1)$ very stable Higgs bundle $\calE_{\d}$ we can determine $\calE^\iota_{\d}=\calE_{ \d^\iota}$, where  $\d^\iota=(\delta_0^\iota,\delta_1^\iota,\dots,\delta_{n-1}^\iota)=(-\delta_0-\delta_1-\dots-\delta_{n-1},\delta_{n-1},\dots,\delta_1)$.  We thus see that the skew involutive property \eqref{skewinv} of the Fourier-Mukai transform  yields
	\beq \label{mirrormirror} S(\Lambda_{\d}|_{\M^\#})=\calL_{\d^\iota}[-n^2(g-1)-1]|_{\M^\#}.\eeq 
\end{remark}


\subsection{An agreement for $\T$-equivariant Euler forms} 

As a motivation recall that mirror symmetry should be an equivalence of categories. Not only objects should match but morphisms between them. If two objects have $\T$-equivariant structure so will the vector space of morphisms between them. Thus its $\T$-character should agree with the $\T$-character of the vector space of morphisms between the mirror objects. In this subsection we find the agreement of these $\T$-characters (aka Euler pairing) for our mirror pairs of branes. Incidently, it is easier to pair a Lagrangian brane with a hyperholomorphic one because the computation reduces to a computation on the Lagrangian support which in our case is just a vector space.

\subsubsection{$\T$-equivariant Euler characteristics and Euler forms}
\label{equivarianteuler}
Recall that for a $\T$-equivariant coherent sheaf $\calF$ on a semiprojective variety $M$ the {\em equivariant Euler characteristic} is defined by
$$\chi_\T(M;\calF):=\sum_{ij} (-1)^i \dim(H^i(M;\calF)^{j}) t^{-j},$$ 
when the weight spaces $H^i(M;\calF)^{j}$, where $\T$ acts with weight $j$, are finite dimensional. The finiteness condition holds for smooth semiprojective varieties and  then $\chi_\T(M;\calF)\in \Z((t)),$ i.e.  $H^i(M;\calF)^j=0$ for $i$ sufficiently large. More generally for a $\T$-equivariant bounded complex of coherent sheaves $\calF^\bullet$ we define $\chi_\T(M;\calF^\bullet)=(-1)^k\sum_k \chi_\T(M;\calF^k)$. 

Furthermore, given two $\T$-equivariant coherent sheaves $\calF_1$ and $\calF_2$ on $M$ we define the {\em equivariant Euler pairing} as

\begin{multline*}\chi_\T(M;\calF_1,\calF_2)=\sum_{k,l} \dim(H^k({ R} {\mathcal Hom}(\calF_1,\calF_2))^l) (-1)^k t^{-l}=\\ =\sum_{k,l} \dim\left(\Hom_{{\mathbf D}_{coh}(M)}(\calF_1,\calF_2[k])^l\right) (-1)^k t^{-l}=\sum_{k,l} \dim(\Ext^k(M;\calF_1,\calF_2)^l) (-1)^k t^{-l} .\end{multline*}

As $\chi_\T(M;\calF_1,\calF_2)=\chi_\T(M;\calF_1^\vee\otimes^L\calF_2),$ where $\calF_1^\vee$ is the derived dual and $\otimes^L$ is the derived tensor product, we expect that  $\chi_\T(M;\calF_1,\calF_2)\in \Z((t))$ on a semiprojective variety $M$.

We note that the upward flow $W^+_\calE$ for $\calE\in \M^{s\T}$ is invariant under the $\T$-action and we shall endow $\calO_{W^+_\calE}$  with the trivial $\T$-equivariant structure. 
\subsubsection{$\T$-equivariant universal bundles}
\label{equiuni}
On the other hand we would like to endow the  sheaf \eqref{bigprod}, when $n|D$,  on $\M^s\times C$ with a $\T$-equivariant structure. For this we need

\begin{lemma} \label{teqcover} We can choose the \'etale cover $\coprod_i U_i\to \M^s$ so that the $U_i$ are $\T$-equivariant and  the local universal Higgs bundle  $(E_{U_i},\bPhi_{U_i})$ on $U_i\times C$ carries a $\T$-equivariant structure with $\T$ acting on $\bPhi_{U_i}$ with weight $-1$. Furthermore the isomorphism \eqref{determinant} is $\T$-equivariant with some $\T$-equivariant structure on $\L$.
\end{lemma}

\begin{proof}
	First we note that over $\calR\times C$ there exists, from \cite[Theorem 4.10]{simpson1}, the universal  framed Higgs bundle $(\bE_\calR,\bPhi_\calR,{\bbbeta}_\calR)$. We have the $\T$-action on $\calR$ multiplying the Higgs field and we  denote  the morphism $m_\lambda:\calR\to \calR$ given by scaling the Higgs field with $\lambda\in \T$. Then $$(m_\lambda^*(\bE_\calR),\lambda^{-1}m_\lambda^*(\bPhi_\calR),m_\lambda^*(\bbbeta_\calR))$$ is also a universal framed Higgs bundle and as $\calR$ is a fine moduli space with a unique universal object we get that $\bE_\calR\cong m_\lambda^*(\bE_\calR)$ canonically, inducing $f_\lambda:\bE_\calR\to \bE_\calR$ covering $m_\lambda$ and making the diagram \beq\label{univdiag}\begin{array}{ccc} \bE_\calR & \stackrel{\bPhi_\calR}{\longrightarrow} &\bE_\calR\otimes K \\ f_\lambda \downarrow && \downarrow f_\lambda  \\ \bE_\calR & \stackrel{\lambda^{-1} \bPhi_\calR}{\longrightarrow} &\bE_\calR\otimes K \end{array}\eeq commute. This will endow $\bE_\calR$ with a canonical $\T$-equivariant structure which acts on $\bPhi_\calR$ with weight $-1$.
	
	Now we take a  point $\calE\in \M^{s}$ such that the orbit $\T\calE\subset \M^s$ is closed. This means that  either  $\calE\in \M^{s\T}$ or that  neither $\lim_{\lambda\to 0}\lambda\calE$ nor (if it exists) $\lim_{\lambda\to \infty}\lambda\calE$ is stable. We take  a lift $\tilde{\calE}\in \calR^s$. Recall that $\PGL_n$ acts freely on $\calR^s$ by changing the framing and $\T$ acts by multiplying the Higgs field. The $\PGL_n$ quotient $p:\calR^s\to \M^s$ is $\T$-equivariant. It follows  that for $\G:=\T\times \PGL_n$ the orbit $\G\tilde{\calE}\subset \calR^s$ is closed. We let $\G_{\tilde{\calE}}\subset \G$ denote the stabilizer which projects injectively into $\T$ as $\PGL_n$ acts freely.  We apply Luna's slice theorem \cite[Theorem III.1]{luna} for the action of $\G$ on $\calR^s$ around the point $\tilde{\calE}$. This way we get a locally closed $\G_{\tilde{\calE}}$-invariant affine slice $\tilde{\calE} \in V\subset \calR^s$, such that the map  $\G\times_{\G_{\tilde{\calE}}} V \to \calR^s$ is \'etale and whose image $U\subset \calR^s$ is  a saturated affine open $\G$-invariant subset, the morphism \beq \label{etale}V/\G_{\tilde{\calE}} \to U/\G\eeq is also \'etale and the natural morphism induces \beq \label{iso}\G\times_{\G_{\tilde{\calE}}} V \cong U\times_{U/\G} V/\G_{\tilde{\calE}}\eeq a $\G$-isomorphism.  We let $V^\prime:=\T\times_{\G_{\tilde{\calE}}}V=(\G\times_{\G_{\tilde{\calE}}}V)/\PGL_n$ then $V^\prime/\T=V/\G_{\tilde{\calE}}$ and dividing \eqref{iso} by $\PGL_n$ gives $$V^\prime=\T\times_{\G_{\tilde{\calE}}} V\cong U/\PGL_n\times_{U/\G} V/\G_{\tilde{\calE}}.$$ By base change of the \'etale morphism \eqref{etale}  by $U/{\PGL_n}\to U/\G=(U/\PGL_n)/\T$  \beq\label{etalelocal}V^\prime \to U/\PGL_n\subset \M^s\eeq is also \'etale and $\T$-equivariant. Finally, the $\T$-equivariant universal Higgs bundle $(\bE_\calR,\bPhi_{\calR})$ restricts to the $\T$-equivariant $V^\prime\to  \calR$ giving  $\T$-equivariant universal bundles $(\bE_{V^\prime},\bPhi_{V^\prime})$ \'etale locally \eqref{etalelocal} around $\calE\in U/\PGL_n$. 
	
	Finally, we note that a $\T$-invariant open covering of the closed $\T$-orbits in $\M^s$ covers the whole $\M^s$ and so the result follows.  
	
\end{proof}

Thus we have local $\T$-equivariant universal Higgs bundles on $U_k\to \M^s$. On overlaps $C\times U_k\times_{\M^s} U_l$ they differ by a $\T$-equivariant line bundle by \cite[(4.2)]{hausel-thaddeus}. Thus the $\T$-equivariant projective bundles $\bP\E_{U_i}$ and $\bP\E_{U_j}$ are canonically isomorphic over the overlap. Therefore they descend as a $\T$-equivariant projective bundle $\bP\E$ over $C\times \M^s$. 

\begin{lemma} In the notation of \eqref{bigprod} when $n|D=n\ell+\sum_{i=1}^{n-1}(n-i)m_i$ there is a $\T$-equivariant structure on the vector bundle $\Lambda_\d=\bigotimes_{i=0}^{n-1}\bigotimes_{j=1}^{m_i} \Lambda^{n-i}(\bE_{c_{ij}})$ normalised such that 
	\beq\label{equinormalise} \det\left(\bigotimes_{i=0}^{n-1}\bigotimes_{j=1}^{m_i} \Lambda^{n-i}(\bE_{c_{ij}}))\right)\cong \Nm^*\left(\bigotimes_{i=0}^{n-1}\bigotimes_{j=1}^{m_i} \L_{c_{ij}}^{\bino{n-1}{n-i-1}}\right)\eeq as equivariant line bundles, where the line bundle $\L_{c_{ij}}^{\bino{n-1}{n-i-1}}$ on $J$ is endowed with a weight ${n-i\choose 2}-{n\choose 2}{\bino{n-1}{n-i-1}}$ $\T$-action.
\end{lemma}

\begin{proof}
	Similarly to $\P\bE$ above,  we can argue that $\bigotimes_{i=0}^{n-1}\bigotimes_{j=1}^{m_i} \Lambda^{n-i}(\bE_{U_k}|_{c_{ij}\times U_k})$ on $U_k\to \M^s$  carries a $\T$-equivariant structure, and they differ by a $\T$- equivariant line bundle on the overlaps. Thus the total space of the projective bundle $\P\left(\bigotimes_{i=0}^{n-1}\bigotimes_{j=1}^{m_i} \Lambda^{n-i}(\bE_{c_{ij}})\right)$ descends to $\M^s$ with a $\T$-equivariant structure.  When $n|D=\sum_{i=0}^{n-1}(n-i)m_i$ then we saw before Theorem~\ref{fouriermukai} that even the vector bundle $\bigotimes_{i=0}^{n-1}\bigotimes_{j=1}^{m_i} \Lambda^{n-i}(\bE_{c_{ij}})$ descends to $\M^s$. Thus the projective bundle $\P\left(\bigotimes_{i=0}^{n-1}\bigotimes_{j=1}^{m_i} \Lambda^{n-i}(\bE_{c_{ij}})\right)$ has a sheaf $\calO(1)$ on it, such that \beq \label{pushdown} \pi_*(\calO(1))^*\cong \bigotimes_{i=0}^{n-1}\bigotimes_{j=1}^{m_i} \Lambda^{n-i}(\bE_{c_{ij}}).\eeq By \cite[Theorem 4.2.2]{brion} the $\T$-action on  $\P\left(\bigotimes_{i=0}^{n-1}\bigotimes_{j=1}^{m_i} \Lambda^{n-i}(\bE_{c_{ij}})\right)$ lifts to $\calO(1)$ because $\Pic(\T)$ is trivial and thus the last map in \cite[(12)]{brion} is an isomorphism. Thus we have a $\T$-equivariant structure on $\bigotimes_{i=0}^{n-1}\bigotimes_{j=1}^{m_i} \Lambda^{n-i}(\bE_{c_{ij}})$ from \eqref{pushdown}. 
	
	We can modify this $\T$-equivariant structure by tensoring with a $\T$ -equivariant  structure on $\calO_{\M^s}$. We can choose it  such that at the canonical uniformising  Higgs bundle $\calE_0=\calE_{\calO_C,0}$  we get that
	$\bigotimes_{i=0}^{n-1}\bigotimes_{j=1}^{m_i} \Lambda^{n-i}(\bE_{c_{ij}}))|_{\calE_0}$ has only non-positive weights and non-trivial weight $0$ space. We can endow $E_{\calO_C}=\calO\oplus K^{-1}\oplus \dots \oplus K^{-n+1}$ with a $\T$-action which acts by weight $-i$ on the $i$th summand $K^{-i}$. This will act on $\Phi_{0}$ with weight $-1$. Therefore any $\T$-equivariant family of stable Higgs bundles containing $\calE_0$ will restrict to $\calE_0$ as this $\T$-equivariant structure, up to a rescaling of the $\T$-action by an overall $1$-dimensional $\T$-module. Therefore  our normalisation of the $\T$-action on  $\bigotimes_{i=0}^{n-1}\bigotimes_{j=1}^{m_i} \Lambda^{n-i}(\bE_{c_{ij}})$  will satisfy \beq\label{eqmultuniv}\chi_\T(\bigotimes_{i=0}^{n-1}\bigotimes_{j=1}^{m_i} \Lambda^{n-i}(\bE_{c_{ij}}))|_{\calE_0})=\prod_{i=0}^{n-1} \left[\begin{array}{c} n \\ n-i\end{array}\right]^{m_i}_t\eeq which incidently also agrees with the equivariant multiplicity $m_{\calE_{\d}}(t)$ of \eqref{explicit}. This will be related to  mirror symmetry in \eqref{mirroriden}. 
	
	A straightforward computation checks that this normalisation of the $\T$-action will induce the one on the determinant as claimed in the Lemma. 
\end{proof}

  The following is an indication  that the mirror relationships \eqref{mirror} and \eqref{mirrormirror}  over $\calA^\#$ hold over the whole $\calA$. 

\begin{theorem} \label{euleriso} For $\ddelta=(\delta_0,\delta_1,\dots,\delta_{n-1})\in J_{{\ell}}(C)\times C^{[m_1]}\times \dots \times C^{[m_{n-1}]}$ and $\ddelta^\prime=({\delta_0}^\prime,\delta^\prime_1,\dots,\delta^\prime_{n-1})\in J_{{\ell^\prime}}(C)\times C^{[m^\prime_1]}\times \dots \times C^{[m^\prime_{n-1}]}$  let $\calE_{\ddelta}, \calE_{\ddelta^\prime}\in \M^{s\T}$ be two very stable type $(1,\dots,1)$ Higgs bundles. Then \bes \chi_\T\left(\M;\calL_{\ddelta^\prime},\Lambda_{\ddelta}\right)&=&\chi_\T(\det(\C_{1}\otimes\calA^*)) \chi_\T\left(\M;\Lambda_{\ddelta^\prime},\calL_{\ddelta^\iota}[-n^2(g-1)-1]
	\right)
	\\ &= &(-1)^{n^2(g-1)+1} t^{(4n+1)(n-1)n(g-1)/6}  \chi_\T\left(\M;\Lambda_{\ddelta^\prime},\calL_{\ddelta^\iota}\right).\ees
	
\end{theorem}   	\begin{proof} We compute
	
	\begin{multline*} \chi_\T\left(\M;\calL_{\ddelta^\prime},\Lambda_{\ddelta}\right)= \chi_\T\left(\M;\calL_{\ddelta^\prime}^\vee \otimes ^L \Lambda_{\ddelta}\right)=\\ = \chi_\T\left(W^+_{\d^\prime};\det(N_{W^+_{\d^\prime}})[-n^2(g-1)-1] \otimes \Lambda_{\ddelta}|_{W^+_{\d^\prime}}\right),\end{multline*} where in the last equation we used \cite[Corollary 3.40]{huybrechts} and the simplified notation $W^+_{\d^\prime}:=W^+_{\calE_{\d^\prime}}$.
	
	As $\det(N_{W^+_{\d^\prime}}) \otimes \Lambda_{\ddelta}|_{W^+_{\d^\prime}}$ is a $\T$-equivariant vector bundle on the  affine space $W^+_{\d^\prime}$ its $\T$-equivariant index can be computed as $$\chi_\T\left(W^+_{\d^\prime};\det(N_{W^+_{\d^\prime}}) \otimes \Lambda_{\ddelta}|_{W^+_{\d^\prime}}\right)=\chi_\T\left((\det(N_{W^+_{\d^\prime}}) \otimes \Lambda_{\ddelta})|_{{\calE_{\d^\prime}}}\right) \chi_\T(W^+_{\d^\prime};\calO_{W^+_{\d^\prime}}).$$
	
	The same argument as above for \eqref{eqmultuniv} gives  $$\chi_\T(\Lambda_{\d}|_{\calE_{\d^\prime}})=m_{\calE_{\d}}(t).$$ From Definition~\ref{equmult} we have $$\chi_\T(W^+_{\d^\prime};\calO_{W^+_{\d^\prime}})=m_{\calE_{\d^\prime}}(t)\chi_\T(\Sym(\calA^*)).$$ Finally $N_{W^+_{\d^\prime}}|_{\calE_{\d^\prime}}\cong T^{\leq  0} _{\calE_{\d^\prime}}$ 
	consequently $$\chi_\T(\det(N_{W^+_{\d^\prime}})|_{{\calE_{\d^\prime}}}))= \chi_\T(\det(T^{\leq  0}_{\calE_{\d^\prime}}))$$  as $\dim \calA = \dim T^+_{\calE_{\d^\prime}}=\dim \M/2$. Thus for the left hand side we have 
	\beq \label{lhs}	\chi_\T\left(\M;\calL_{\ddelta^\prime},\Lambda_{\ddelta}\right)=(-1)^{\dim \cal A}m_{\calE_{\d}}(t)m_{\calE_{\d^\prime}}(t)\chi_\T(\Sym(\calA^*))\chi_\T(\det(T^{\leq  0}_{\calE_{\d^\prime}})).\eeq
	
	For the other side we compute similarly \beq \chi_\T\left(\M;\Lambda_{\ddelta^\prime},\calL_{\ddelta^\iota}\right)\nonumber&=&\chi_\T(\M;\Lambda^\vee_{\ddelta^\prime}\otimes^L\calL_{\ddelta^\iota})\\ \nonumber&=&\chi_\T(W^+_{\d^\iota};\Lambda^\vee_{\ddelta^\prime}|_{W^+_{\d^\iota}} )\\ \nonumber&=&\chi_\T(\Lambda^\vee_{\ddelta^\prime}|_{{\calE_{\ddelta^\iota}}} )\chi_\T(W^+_{\d^\iota};\calO_{W^+_{\d^\iota}} )\\ \nonumber&=& m_{\calE_{\d^\prime}}(t^{-1}) m_{\calE_{\d^\iota}}(t)\chi_\T(\Sym(\calA^*))\\ \nonumber&=& \frac{\chi_\T(\det(\calA))}{\chi_\T(\det(T^+_{\calE_{\d^\prime}}))}m_{\calE_{\d^\prime}}(t) m_{\calE_{\d^\iota}}(t)\chi_\T(\Sym(\calA^*)) \\ &=& \frac{\chi_\T(\det(T^{\leq 0}_{\calE_{\d^\prime}}))}{\chi_\T(\det(\C_{1}\otimes\calA^*))}m_{\calE_{\d^\prime}}(t) m_{\calE_{\d^\iota}}(t)\chi_\T(\Sym(\calA^*)), \label{rhs}
	\eeq
	where in the fourth equation we used \eqref{palindromic}. For the last equation we can argue as follows. As $\iota$ of \eqref{iota} is a $\T$-equivariant isomorphism, it follows that $T^+_{\calE_{\d^\iota}}\cong T^+_{\calE_{\d}}$ and thus $m_{\calE_{\d^\iota}}(t)=m_{\calE_{\d}}(t)$. Using the homogeneity $1$ symplectic form on $T_{\calE_{\d^\prime}} $ we see that $\chi_\T(\C_{-1}\otimes\det(T^+_{\calE_{\d^\prime}}))\chi_\T(\det(T^{\leq  0}_{\calE_{\d^\prime}}))=1$. This implies the last equation. Comparing \eqref{lhs} and \eqref{rhs} implies the first equation in the Theorem.
	
	For the second equation we note that $\dim \calA=\sum_{i=1}^n\dim H^0(C;K^i)= n^2(g-1)+1$ and that a straightforward computation  gives $$\chi_\T(\C_1\otimes \calA^*)=t^{\sum_i (i-1)\dim \calA_i}=t^{\sum_{i=1}^n (i-1)\dim H^0(C;K^i)}=t^{\sum_{i=1}^n (i-1)(2i-1)(g-1)}=t^{(4n+1)(n-1)n(g-1)/6} .$$ The result follows.

\end{proof}

\begin{remark} We proved in Theorem~\ref{fouriermukai} that the relative Fourier-Mukai transform $S(\calL_{\d})=\Lambda_{\d}$ holds over $\calA^\#$ and in Remark~\ref{remarkdual} that  $S(\Lambda_{\d})=\calL_{\d^\iota}[-n^2(g-1)-1]$ also holds over $\calA^\#$. The above theorem  would be a consequence  of a $\T$-equivariant extension of  the mirror symmetry conjecture in the "classical limit"  as in \cite[Conjecture 2.5]{donagi-pantev}. 
\end{remark}

\begin{remark} Similar agreement of $\T$-equivariant Euler pairings between certain conjectured mirror pairs of Lagrangian and hyperholomorphic branes of \cite{hitchin2} were observed 
	in \cite[\S 7]{hausel-mellit-pei}, when $n=2$. When $n=2$ we have $\chi_\T(\C_1\otimes \calA^*)=t^{3g-3}$ in Theorem~\ref{euleriso} which we cannot get rid of by just renormalising the $\T$-action on our branes when $3g-3$ is odd. The same phenomenon was observed in \cite[Remark 7.7]{hausel-mellit-pei}. 
	
	In fact, one can also pair certain $n=2$ analogues of the mirror pairs of this paper with the ones in \cite{hausel-mellit-pei} and find the expected agreement. 
	
\end{remark}

\begin{remark} Just as in \cite[Corollary 4.11]{hausel-mellit-pei} we get a clearly symmetric expression in $\d$ and $\d^\prime$ if we compute, equivalently as in  \eqref{lhs},
	\beq	\label{tensorsym}\chi_\T(\M;\calL_{\d^\prime}\otimes \Lambda_{\d})=m_{\calE_{\d^\prime}}(t)m_{\calE_{\d}}(t)\chi_\T(\Sym(\calA^*)).\eeq When deriving the agreement of the corresponding Euler forms \cite[Corollary 7.1]{hausel-mellit-pei} we basically argued that it follows from the fact that both of the branes $\calL_i$ and $\Lambda_j$ were self-dual. In our present situation ${\calL_{\d^\prime}}$ is self-dual (up to a shift) for the trivial reason that it is supported on an affine space. On the other hand $\Lambda_{\d}$ is typically not self-dual, rather it satisfies $\Lambda^\vee_{\d}=\iota^*(\Lambda_{\d})=\Lambda_{\d^\iota}$ using the normalised inverse map $\iota$ of Remark~\ref{remarkdual}.  However in \cite{hausel-mellit-pei} both types of our branes $\calL_i$ and $\Lambda_j$ were secretly invariant under $\iota$ as well. In general, what \eqref{skewinv} tells us is that self-dual branes should be mirror to branes invariant under $\vee \circ \iota$. The former is satisfied by $\calL_{\d^\prime}$ and the latter by $\Lambda_{\d}$. 
\end{remark}

\begin{remark}  \label{mirrorate0}
	It is instructive to take $\delta^\prime=0=(0,0,\dots,0)$ in \eqref{tensorsym}. Then $\calE_{0}=\calE_{\calO_C,0}$ is the canonical uniformising Higgs bundle of  \eqref{uniform}. It has $m_{\calE_0}(t)=1$ and $\calL_{0}\cong \calO_{W^+_0}$ is the structure sheaf of the canonical Hitchin section, while $\Lambda_0\cong \calO_{\M^s}$.   We get $$\chi_\T(\M;\calO_{W_{\delta}^+}\otimes \Lambda_{0})= m_{\calE_{\d}}(t)\chi_\T(\Sym(\calA^*)).$$   On the other hand $$\chi_\T(\M;\calO_{W_0^+}\otimes \Lambda_{\d})=\chi_\T(\calO_{W_0^+};\Lambda_{\d}|_{W_0^+})=\chi_\T(\Lambda_{\d}|_{\calE_0})\chi_\T(\Sym(\calA^*)).$$ Thus $\chi_\T(\M;\calO_{W_{\delta}^+}\otimes \Lambda_{0})=\chi_\T(\M;\calO_{W_0^+}\otimes \Lambda_{\d})$ (or the equivalent Theorem~\ref{euleriso}) implies that \beq \label{mirroriden} \chi_\T(\Lambda_{\d}|_{\calE_0})=m_{\calE_{\d}}(t).\eeq This was observed already in \eqref{eqmultuniv}, but here we derived it as a consequence of mirror symmetry. Indeed such a relationship is expected from the mirror in general. Because  of the restriction property \eqref{fmiden} of the Fourier-Mukai transform to the identity we expect from the fibrewise Fourier-Mukai transform that the mirror of a bounded complex of coherent sheaves $\calF^\bullet$ on $\M$ should restrict to the canonical Hitchin section as $h^*(h_*(\calF^\bullet) )$. The latter was computed for the upward flow of a very stable Higgs bundle in \eqref{hitchinpush} and \eqref{t-character}. Thus \eqref{mirroriden} verifies this expectation from the mirror of $\calL_{\d}$. 
\end{remark}

\begin{remark} We can also generalise these ideas further for very stable upward flows of type not $(1,\dots,1)$. For  $\calE\in \M^{s\T}$  a very stable Higgs bundle let $\calL_\calE:=\calO_{W^+_\calE}$ denote the structure sheaf of its upward flow and  $\Lambda_\calE$   denote its mirror, which we expect to be a $\T$-equivariant vector bundle on $\M^s$ of rank $m_\calE$. If $\calF\in \M^{s\T}$ is another very stable Higgs bundle then denote by $$m_{\calE,\calF}(t):=\chi_\T(\Lambda_\calE|_\calF)\in \Z[t]$$ the $\T$-character  of the fibre $\Lambda_\calE|_\calF$, which we expect to contain only non-positive weights. In particular,  by \eqref{mirroriden} we have $m_{\calE,\calE_0}(t)= m_\calE(t)$ the equivariant multiplicity. As above from mirror symmetry we expect $$\chi_\T(\M;\calL_\calE\otimes \Lambda_\calF)=\chi_\T(\M;\calL_\calF\otimes \Lambda_\calE)$$ which in turn leads to the identity
	\beq\label{mefidentity} m_\calE(t) m_{\calF,\calE}(t)=m_\calF(t) m_{\calE,\calF}(t).\eeq When both $\calE$ and $\calF$ are of type $(1,\dots,1)$ then $m_{\calE,\calF}(t)=m_\calE(t)$ is just the equivariant multiplicity as $\Lambda_\calE|_{\calF}\cong \Lambda_\calE|_{\calE_0}$  and so \eqref{mefidentity} holds. If only $\calF$ is of type $(1,\dots,1)$ we know what $\Lambda_\calF$ should be and how it restricts to any $\calE$ giving a formula for $m_{\calF,\calE}(t)$ only depending on the type of $\calE$. From \eqref{mefidentity} we can deduce $$m_{\calE,\calF}(t)=\frac{m_\calE(t) m_{\calF,\calE}(t)}{m_\calF(t)}.$$ As it is again expected to be a polynomial in $t$ we get more obstructions on $\calE$ being very stable. 
	
	To illustrate the above take the example of $n=2$ and $d=-2(g-1)$ and let $\calE$ be a very stable Higgs bundle of type $(2)$, i.e. a very stable bundle $E$ together with the zero Higgs field. For some $0\leq i<g$ let $\calF\in F_{l,2i}$ be a very stable type $(1,1)$ Higgs bundle. Then $\Lambda_\calF$ is a tensor product of universal bundles restricted to points of $C$.  By carefully following the normalisation of the $\T$-action on the tensor product of these restricted  universal bundles given by Lemma~\ref{equinormalise} we find that 
	$$m_{\calF,\calE}(t)=\chi_\T(\Lambda_\calF|_\calE)=2^{2i}t^i .$$ By \eqref{mefidentity} we get that
	\beq \label{cotmir} m_{\calE,\calF}(t)=\frac{m_\calE(t) m_{\calF,\calE}(t)}{m_\calF(t)}=\frac{(1+t)^{3g-3} 2^{2i}t^i}{(1+t)^{2i}}=2^{2i}t^i(1+t)^{3g-3-2i}.\eeq Thus even though we do not know (see \ref{cotfib} for more on this) what vector bundle $\Lambda_\calE$ the mirror of the cotangent fibre at $E$  should be we can infer its $\T$-character at type $(1,1)$ fixed points  from \eqref{cotmir}.

\end{remark}

\section{Hyperk\"ahler aspects}
\label{hyper}

\subsection{Hyperholomorphic connections}

Thus far we have considered $\M$ as a variety with a holomorphic symplectic structure in which the upward flows are complex Lagrangian subvarieties. At this point we bring into play the hyperk\"ahler metric on $\M$. This is a metric which is K\"ahler with respect to complex structures $I,J,K$ which obey the quaternionic relations $I^2=J^2=K^2=IJK=-1$. It exists as a consequence of the theorem \cite{hitchin},\cite{simpson1},\cite{simpson1b} that a stable Higgs bundle admits a Hermitian metric, unique up to gauge equivalence, satisfying the equation $F_A+[\Phi,\Phi^*]=0$. 

Within this context, mirror symmetry predicts that a BAA-brane should correspond to a BBB-brane on the mirror of $\M$, which is to be understood as the statement that to a flat vector bundle on a complex Lagrangian on $\M$ there should exist a hyperholomorphic connection on a vector bundle over a hyperk\"ahler submanifold of the mirror.  A hyperholomorphic connection is a unitary connection whose curvature is of type $(1,1)$ with respect to each complex structure in the hyperk\"ahler family  of complex structures $x_1I+x_2J+x_3K$ with   $(x_1,x_2,x_3)\in S^2\subset \R^3$.  At this point in time, we have few constructions of hyperholomorphic connections on ${\mathcal M}$ so verifying the predictions is an important issue.

Since the fibration structure of $\M$ gives a natural setting for the Strominger-Yau-Zaslow approach to mirror symmetry, it is the Fourier-Mukai transform approach which can be applied. In the original SYZ context of a special Lagrangian fibration Arinkin and Polishchuk \cite{arinkin} showed how a real Lagrangian can describe a holomorphic bundle on the mirror. This could be applied to the current situation but would give at most  a hyperholomorphic connection relative to the {\emph semi-flat} hyperk\"ahler metric which is defined over the regular locus in $ \calA$ and is only an approximation to the actual metric. We have described the Fourier-Mukai transform of the trivial bundle over an upward flow Lagrangian in terms of universal bundles over the mirror moduli space, so to verify the prediction we need to construct a hyperholomorphic connection, using the genuine hyperk\"ahler metric, on a universal bundle  over the mirror of the moduli space. Since we are working with $GL_n(\C)$ Higgs bundles here, and this group is isomorphic to its Langlands dual, the mirror is the same moduli space. Twistor theory allows us to express this by holomorphic means using the direct analogy of the Atiyah-Ward description of instantons \cite{AHS}, \cite{Sal} -- a hyperholomorphic connection on $\R^4$ is precisely an  anti-self-dual connection. 

Recall that the twistor space $Z$ of a hyperk\"ahler manifold $M$ is  the manifold $M\times S^2$ endowed with a complex structure such that the fibre over $(x_1,x_2,x_3)\in S^2$ is $M$ with complex structure $x_1I+x_2J+x_3K$. The projection onto the second factor is a holomorphic fibration $p: Z\rightarrow \PP^1$ with a real structure $\sigma$ covering the antipodal map on $S^2$ and each 2-sphere $(m,S^2)$ is a real holomorphic section, called a twistor line.   Then a unitary hyperholomorphic connection on a vector bundle is equivalent to a holomorphic bundle on $Z$ which is trivial on each twistor line and real with respect to $\sigma$.

A description of the twistor space for Higgs bundles due to Deligne and Simpson \cite{simpson3} is phrased in terms of $\lambda$-connections. A holomorphic $\lambda$-connection on a vector bundle $E$ is a differential operator $D:E\rightarrow E\otimes K$ satisfying 
$D(fs)=fDs+\lambda s\otimes df$ for a local holomorphic section $s$ and function $f$. If $\lambda\ne 0$ then $\lambda^{-1}D$ is a holomorphic, and therefore flat, connection and for $\lambda=0$ this is a Higgs field. There is an algebraic construction of a moduli space ${\M}^s_{\mathrm{Hod}}$ of stable $\lambda$-connections and the parameter $\lambda$ gives a projection to $\C$. 

A solution to the Higgs bundle equations gives rise to a $\bar\partial$-operator $\nabla^{0,1}_A+\zeta\Phi^* =\bar\partial_A+\zeta\Phi^*$ parametrized by $\zeta\in \C$ which commutes with $\zeta\partial _A+ \Phi$, so defining a $\lambda$-connection with $\lambda=\zeta$ for the holomorphic structure on $E$ defined by  the $\bar\partial$-operator. This is a holomorphic section of ${\mathcal M}_{\mathrm{Hod}}\rightarrow \C$ and the twistor space ${\mathcal Z}$ is constructed so as to extend this to $\zeta=\infty\in \PP^1$ and make this a twistor line. It  is defined by  taking the moduli space of $\lambda$-connections  and the corresponding one for the conjugate complex structure on $C$, and identifying with the moduli space of  flat connections over $\zeta\in \C^*\subset \PP^1$ by means of the holonomy, the representation of $\pi_1(C)$ in $GL_n(\C)$, which is independent of any holomorphic structure. Then identifying $\lambda=\zeta$ on one side with $\lambda=\zeta^{-1}$ on the other gives ${\mathcal Z}$.

When $\zeta\ne 0$ we have the flat connection
\begin{equation}
\nabla_A+\zeta^{-1}\Phi+\zeta \Phi^*.
\label{flatconn} 
\end{equation}
and as $\zeta$  tends to $\infty$, $\partial_A+\zeta^{-1}\Phi$ defines a holomorphic structure with respect to the conjugate complex structure on $C$ and $\zeta^{-1}\bar\partial_A+\Phi^*$  a $\lambda$-connection for $\lambda=\zeta^{-1}$.   This defines a section of the fibration $p: {\mathcal Z}\rightarrow \PP^1$.

Since $A$ is a unitary connection the antilinear involution $a\mapsto -a^*$ on $\End(E)$ coupled with the antipodal map $\zeta\mapsto -1/\bar \zeta$ acting on $\PP^1$ transforms $\nabla^{1,0} _A+ \zeta^{-1}\Phi$ to  $\nabla^{0,1}_A+\bar \zeta \Phi^*$ and so a solution to the Higgs bundle equations  defines a real section, and this corresponds to a point of $\M$.

\subsection{The connection on a universal bundle}
As described above, the Fourier-Mukai transform of the upward flow of a very stable fixed point  of type $(1,1,\dots,1)$ can be expressed as a tensor product of exterior powers of the universal bundle corresponding to  points $c\in C$. A hyperholomorphic connection on the universal bundle will then induce one on the associated tensor or exterior powers. 

In the  complex structure  $\zeta \ne 0,\infty$ there is a straightforward description of  the universal bundle: given a framing of a bundle $E$ at $c$, the holonomy of a flat connection on $E$ around closed curves with base point $c$ defines a homomorphism from $\pi_1(C,c)$ to $\GL_n(\C)$.   We denote $\calR_\B=\Hom(\pi_1(C,c)\to \GL_n(\C))$, and the open locus of irreducible representations $\calR^s_\B$. We have the action of $\GL_n(\C)$ on $\calR_\B$ by conjugation, and denote the affine quotient by $\M_\B$, the character variety. Over the open subset $\M^s_\B=\calR^s_\B/\GL_n(\C)\subset \M_\B$ the  variety $\calR^s_\B$ is a principal $\PGL_n(\C)$-bundle, representing irreducible local systems framed at the point $c\in C$.

As described in Section 6, there is an obstruction in  $H^2({\M^s_\B}, \Z_n)$ to lifting to a  principal $\GL_n(\C)$-bundle. Given local liftings over  open sets $U_i$, the obstruction class is represented by a flat line bundle $L_{ij}$ over $U_i\cap U_j$ with $L_{ij}^n$ trivial -- the data for a gerbe $\theta$. If the universal vector bundle ${\mathbb E}$  does not exist globally, it  does over the neighbourhoods $U_i$, and a connection is a locally defined notion. Moreover tensoring with a flat line bundle $L_{ij}$ acts trivially on the curvature and so takes a local hyperholomorphic connection to another one, so the notion of hyperholomorphic passes over to the situation where we need to twist by a flat gerbe.

For the complex structure $\zeta=0$ or $\infty$, the moduli space of framed Higgs bundles was constructed by Simpson in    \cite[Theorem 4.10]{simpson1}, and in \cite[Proposition 4.1]{simpson3} he  outlined, by the same approach, a construction for $\lambda$-connections. We denote by $\calR^s_\Hod$ the framed moduli space of stable $\lambda$-connections. This carries a $\GL_n(\C)$-action on the framing, so that it becomes a  principal $\PGL_n(\C)$-bundle over $\M^s_\Hod$.  Over $\lambda=0$ this principal bundle $\calR^s_\Hod\to \M^s_\Hod$ restricts to $\calR^s_\Dol\to \M_\Dol^s=\M^s$  the framed moduli space of stable rank $n$ degree $0$ Higgs bundles, which is a $\PGL_n(\C)$-principal bundle over $\M^s$. 

When $\lambda\ne 0$ we have the complex analytically equivalent alternative description $\calR^s_\B\to \M^s_\B$ of the framed moduli space in terms of the character variety, using the holonomy of the flat connection. This  is clearly compatible with the identification in the construction of  the twistor space ${\mathcal Z}$,  so we have a holomorphic $\PGL_n(\C)$-bundle over ${\mathcal Z}$. The final point is to show this is trivial on a real twistor line. 

A real twistor line is defined by  a  solution to the Higgs bundle equations which, as $\zeta$ varies,  is a pair $\bar\partial_A+\zeta\Phi^*$,  $\zeta\partial _A+ \Phi$ on the {\emph {same}} $C^{\infty}$ vector bundle $E$. A basis $e_1,\dots, e_n$ of the fibre $E_c$ then defines a framing for all holomorphic structures and we need this to be holomorphic in $\zeta$. But  from the integrability theorem of the holomorphic structure $\bar\partial_A+\zeta\Phi^*$ we can find a basis of local sections $s_1,\dots, s_n$ in a neighbourhood of $c$   with $(\bar\partial_A+\zeta\Phi^*)s_i=0$ and  these vary holomorphically with $\zeta$. 
Then evaluation at $c$ gives a frame which is holomorphic in $\zeta$ and similarly using the conjugate complex structure for $\zeta\ne 0$. .

\begin{remark} 
	
	A formal explanation for the  existence of the hyperholomorphic connection is provided by the infinite-dimensional hyperk\"ahler quotient construction of ${\mathcal M}$ itself \cite{hitchin}. This considers the action of the group ${\mathcal G}$ of unitary gauge transformations on the infinite-dimensional flat hyperk\"ahler manifold which is the cotangent bundle of the space ${\mathfrak A}$ of $\bar\partial$-operators $\bar\partial_A$ on the underlying $C^{\infty}$ vector bundle $E$. The zero set of the hyperk\"ahler moment map for this action consists of pairs $(A,\Phi)$ satisfying the Higgs bundle equations. This space is formally a principal ${\mathcal G}/U(1)$ bundle over the quotient ${\mathcal M}$ and the restriction of the flat $L^2$ metric on $T^*{\mathfrak A}$ defines a distribution orthogonal to the orbit directions, and this is a connection on the principal bundle. In any hyperk\"ahler quotient this is hyperholomorphic. Now given a framing at $c$, the evaluation map gives a homomorphism from ${\mathcal G}/U(1)$ to $PU(n)$ and an associated hyperholomorphic connection on the universal bundle.
\end{remark}


Thus we have proved the following.

\begin{theorem} There is a hyperholomorphic connection on the $\PGL_n(\C)$ bundle $\calR^s_\Dol$ on $\M^s$, which is associated to the projective bundle $\P(\E_c)$. Moreover, there is a hyperholomorphic structure on the  vector bundle $\E_c$ on $\M^s$ twisted by the gerbe  $\theta$. 	
\end{theorem}

\subsection{The universal Higgs field} 
For each point $c\in C$ we obtain the universal bundle $\E_c$ with a hyperholomorphic connection on ${\mathcal M}$, and so $C$ itself is the parameter space for  a family of such bundles with connection.  Equivalently, these correspond to holomorphic bundles on the twistor space ${\mathcal Z}$. In particular, considering first order deformations, there is a natural map from the tangent space of $C$ at $c$ to $H^1({\mathcal Z},\End (\E_c))$. This space can be interpreted  in terms of variations of hyperholomorphic connections  by using the Penrose transform in the current quaternionic context \cite{Sal}. It is isomorphic to the first cohomology group of an elliptic complex generalizing  the deformation complex for anti-self-dual connections in four dimensions in \cite{AHS}. 

Since $\M$ is the moduli space of pairs $(E,\Phi)$ the universal bundle should have a universal Higgs field to accompany it and indeed a global construction of the universal pair  can be found in  \cite{hausel-vanishing} for the $(d,n)=1$ case
and  \cite[Theorem 4.5]{simpson1}  for a local construction in general. 
Given the role of the universal bundle, the reader may wonder how the universal Higgs field at $c$ enters the picture. We shall see here that it defines the infinitesimal variation of the hyperholomorphic connection on the universal bundle.

The starting point is the observation that the space $\Hom(\pi_1(C,c), \PGL_n(\C))$ is independent of $c$. More accurately, choosing a path from $c$ to $y$, there is a biholomorphism between the two spaces which only depends on the homotopy class of the path. In our situation it is the holonomy of the flat connection along that path which identifies the two. We have a flat connection when $\zeta \ne 0$ or $\infty$ and this suggests that the deformation class in $H^1({\mathcal Z},\End (\E_c))$ is supported on the fibres $Z_0,Z_{\infty}\subset {\mathcal Z}$ over these two points in $\PP^1$.

Given a tangent vector  $X\in T_c(C)$, we can contract with the  Higgs field $\Phi_c\in (\End (E) \otimes K)_c$ to obtain $\Phi(X)$ which is a section of the universal bundle $\End (\E_c)$ over  the moduli space of Higgs bundles  ${\mathcal M}$ which we now think of as the fibre $Z_0$ over $\zeta=0$ in the twistor space. Similarly $-\Phi^*$ defines a holomorphic section $-\Phi^*(X)$ over $Z_{\infty}$, the moduli space  for the conjugate complex structure. Let $s$ be the pullback  to ${\mathcal Z}\rightarrow \PP^1$ of the holomorphic section $\zeta d/d\zeta$ of $T\PP^1\cong {\mathcal O}(2)$ on $\PP^1$, then we have an exact sequence of sheaves on  ${\mathcal Z}$:
$$0\rightarrow {\mathcal O}(\End (\E_c))\stackrel {s} \rightarrow {\mathcal O}(\End (\E_c)(2))\rightarrow {\mathcal O}_{Z_0+Z_{\infty}}(\End (\E_c)(2))\rightarrow 0.$$
The vector field  $d/d\zeta$ trivializes ${\mathcal O}(2)$ at $\zeta=0$ and similarly at infinity so the pair $(\Phi(X),-\Phi^*(X))$ may be regarded as a section of $\End (\E_c)(2)$ over $Z_0+Z_{\infty}$. Then we have: 
\begin{proposition} The map $TC_c\rightarrow H^1({\mathcal Z},\End (\E_c))$ defining the first order deformation of the universal bundle at $c$ is $\delta(\Phi(X),-\Phi^*(X))$ where $\Phi$ is the universal Higgs field at $c$ and  $\delta: H^0({Z_0+Z_{\infty}},\End (\E_c)(2))\rightarrow H^1({\mathcal Z},\End (\E_c))$   is the connecting homomorphism in the cohomology of the above exact sequence. 
\end{proposition} 
\begin{proof} For $y\in C$ in a neighbourhood of $c$,  parallel translation of a flat connection  takes the principal bundle $\Hom(\pi_1(C,c), \PGL_n(\C))$ to $\Hom(\pi_1(C,y), \PGL_n(\C))$. This is an isomorphism  of the universal bundle at $c$ to the universal bundle at $y$, and to first order at $c$ it is the horizontal lift of a tangent vector $X$ to the universal bundle over $\M\times C$, the lift defined by the connection. This is a trivial deformation of holomorphic  principal bundles where $\M$ has the complex structure of the moduli space of flat connections, so since $\zeta\ne 0,\infty$ describes the points of ${\mathcal Z}$ via flat connections, it gives a trivialization of the deformation of the universal bundle on this part of ${\mathcal Z}$. The flat connection is $\partial _A+ \Phi/\zeta$ and so the trivialization does not extend to $\zeta=0$. 
	
	We may consider the first order deformation of the universal bundle as the Kodaira-Spencer class for the deformation of complex structure of the principal bundle ${\mathcal P}_c\rightarrow {\mathcal Z}$ as $c\in C$ varies. 
	In \v Cech cohomology, the deformation class in $H^1({\mathcal Z},\End (\E_c))\subset H^1({\mathcal P}_c, T)$ is given using an open covering  $\{U_i\}$, taking  local lifts $X_i$ of the tangent vector $X$ and defining the cocycle $X_j-X_i$, a section of $\End (\E_c)$ over $U_i\cap U_j$. In our case, for $\zeta\ne 0$, there is a global lift given by the horizontal space of the connection $\partial _A+ \Phi/\zeta$. If we multiply by $\zeta$ then this is global for all $\zeta$. Repeating this process near $\zeta=\infty$ shows that the class maps to zero in $H^1({\mathcal Z},(\End (\E_c))\stackrel {s} \rightarrow H^1({\mathcal Z},\End (\E_c)(2))$ and consequently is in the image of the coboundary map $\delta:
	H^0({Z_0+Z_{\infty}},(\End (\E_c)(2))\rightarrow H^1({\mathcal Z},(\End (\E_c))$.
	
	But the polar part of $\partial _A+ \Phi/\zeta$ at $\zeta=0$ is $\Phi(X)$ and 
	we see that the class is given by extending $\Phi(X)$ and $-\Phi^*(X)$ as sections of $\End (\E_c)(2)$ and dividing by $s$, which is the coboundary map  in \v{C}ech cohomology.
\end{proof}

\section{Further issues}
\label{further}

\subsection{Simple Lie groups}\label{simple}
The simplest example of a very stable upward flow of type $(1,1,...1)$ is, as we have seen, the canonical section of the Hitchin fibration. In fact  the moduli space of Higgs bundles for any simple Lie group $\G$ admits   a corresponding section, constructed in \cite{hitchin-lie}. For the adjoint group this is unique. This points towards a generalization of the $(1,1,...1)$ type in this context, where the mirror should be a hyperholomorphic bundle on the moduli space for the Langlands dual group $\G^{\vee}$.

The key point in \cite{hitchin-lie} is to use the principal three-dimensional subgroup of $\G$ \cite{Kos}. The fixed point $\mathcal E$ is then obtained from the canonical uniformising $\SL_2$ Higgs bundle  via the homomorphism $\SL_2\rightarrow \G$. At a fixed point  the scalar action on the Higgs field arises from the action of a subgroup of the adjoint  group of  $\G$ and in this case it is the  action of the image of $\C^\times\subset \SL_2$. 

The characteristic feature of the subgroup is that it breaks up the Lie algebra $\lie{g}$ into $\ell$ irreducible components where $\ell$ is the rank of $\G$ and the dimensions are $2d_i-1$ where the $d_i$ are the degrees of a basis of invariant polynomials. These are odd-dimensional and hence representations of $\SO_3$. Therefore $t   \in \T$ acts through $(t^{-1},1,t)\subset \SO_3$.

The  Higgs field  has homogeneity one with respect to the $\T$-action. Each  irreducible representation of $\SO_3$  has a single weight one subspace so the Higgs field must take values  in an $\ell$-dimensional subspace of $\lie{g}$. From \cite[Lemma 5.2]{Kos}, the action has  weight $1$ on the root spaces of the $\ell$ simple roots $\alpha_1,\dots, \alpha_{\ell}$ which therefore form a basis for this subspace.  In the case of the general linear group  treated in this paper   the simple roots are $x_{i+1}-x_i$ which correspond to line bundles  $L_{i+1}L_i^*$ and we dealt with  sections $b_i$ of $L_{i+1}L_i^*K$. The direct analogue for the group $\G$  of the $(1,1,\dots,1)$ type  fixed point is then a line bundle $L_{\alpha_i}$ for each simple root and a section $b_i$ of $L_{\alpha_i}K$. Writing any root $\alpha$ in terms of simple roots then defines by tensor product a line bundle $L_{\alpha}$ and a Lie algebra bundle 
$$H\oplus \bigoplus_{\alpha\in \Delta}L_{\alpha}$$
where $H$ is a trivial rank $\ell$ bundle associated to the Cartan subalgebra. For the canonical fixed point $\calE_0$  in the moduli space of $\G$-bundles each $b_i$ is nonvanishing and the upward flow is a section of the Hitchin fibration, and in particular is closed. We shall discuss next the case of a Higgs bundle $\calE$ when $b_i$ is of degree $m_i\ge 0$ by studying the virtual multiplicity $$m_\calE(t)=\frac{\chi_{\T}(\Sym ({T_\calE^+}^*))}{\chi_{\T}(\Sym ({\mathcal A}^*))}.$$ 
Here $T_\calE^+$ is the sum of root space line bundles for roots of positive weight under the $\C^{\times}$-action. Since the weight is $1$ on the simple roots, these are the positive roots. If  $\alpha=c_1\alpha_1+\dots+c_{\ell}\alpha_{\ell}$ in terms of the simple roots then $w(\alpha)=c_1+\dots+c_{\ell}$.   Since $m_i=\deg L_{\alpha_i}K$ we have 
$$\deg L_{\alpha}=\sum_{i=1}^{\ell} c_i(\alpha) (m_i-(2g-2)). $$

Following the procedure of  Remark 5.12 we obtain 
$$\chi_\T(T_\calE^+)=\sum_{\alpha>0} t^{-w(\alpha)}\left(\sum_{w(\alpha)=k}-(\deg L_{\alpha}+(1-g))+\sum_{w(\alpha)=k-1}\deg(L_{\alpha}+(g-1))\right)$$
or
$$\chi_\T(T_\calE^+)=\sum_{\alpha>0} t^{-w(\alpha)}\left(\sum_{w(\alpha)=k}\sum_{i=1}^{\ell} -m_ic_i(\alpha) +\sum_{w(\alpha)=k-1}\sum_{i=1}^{\ell} m_ic_i(\alpha) \right)+(g-1)N(t).$$ 
When all the $m_i=0$, the multiplicity is $1$ because then we have an equivariant section of ${\mathcal A}\rightarrow {\mathcal M}$ (one may also easily do the calculation -- see below the generating function for $N_j$) so the $(g-1)N(t)$ term above cancels the denominator ${\chi_{\T}(\Sym ({\mathcal A}^*))}$ and the multiplicity reduces to 
\begin{equation}
m_\calE(t)=\prod_{i=1}^{\ell}\prod_{\alpha>0}\left(\frac{1-t^{w(\alpha)+1}}{1-t^{w(\alpha)}}\right)^{m_ic_i(\alpha)}. 
\label{multG}
\end{equation} 
By experimentation one finds that this is rarely a polynomial. In the case of ${\G}_2$ this happens independently of the $m_i$, supposing they are not identically zero, for other groups there are constraints on the $m_i$. So there is a distinct difference from the case of $\GL_n$ or $\SL_n$. One case however always gives a polynomial as we see next. 

A simple root $\alpha_i$ is called {\it cominuscule} if the coefficient $c_i(\alpha)=0,\pm1$ for every root $\alpha$.  Every simple root of $\SL_n$ has this property but for other types  they are special or non-existent. The coroot $\alpha_i^{\vee}$ in the dual root system, the root system of the Langlands dual group, is called {\it minuscule} (see for example \cite{Buch}) and this corresponds to an irreducible representation: the exterior powers of the vector representation for $\SL_n$, the vector representation for $\mathrm{Sp}_{2n}$, the spin representation for $\mathrm{Spin}_{2n+1}$ and the vector and two spin representations for $\mathrm{Spin}_{2n}$.

Suppose then that $\alpha_i$ is cominuscule and the Higgs field has $m_j=0$ for $j\ne i$ and $m_i=1$. The formula (\ref{multG}) is multiplicative in $m_i$ so this case is sufficient. Since $c_i(\alpha)=0,1$ for a positive root, the multiplicity is 
$$\prod_{\alpha>0, c_i(\alpha)\ne 0}\left(\frac{1-t^{w(\alpha)+1}}{1-t^{w(\alpha)}}\right). $$
Consider this product over all the positive roots of $\G$:
$$\prod_{\alpha>0}\left(\frac{1-t^{w(\alpha)+1}}{1-t^{w(\alpha)}}\right). $$
Let $N_j$ be the number of root spaces of weight $j$, then this expression is
$$q_\G(t)=\prod_{j\ge 1}\left(\frac{1-t^{j+1}}{1-t^{j}}\right)^{N_j}= \prod_{j\ge 1}(1-t^{j})^{N_{j-1}-N_j}.$$

The Lie algebra $\lie{g}$ splits under the $\SO_3$-action into $\ell$  irreducibles of dimension $2d_i-1$ where $d_i$ is the degree of a basis element of the ring of invariant polynomials. In each irreducible of dimension $2d_i-1$, there is a one-dimensional space of weight $1,2,\dots d_i-1$. Then $$N_1x+N_2x^2+\cdots =\sum_{i=1}^{\ell}(x+x^2+\cdots + x^{d_i-1})=\frac{x}{(1-x)}\sum_1^{\ell}(1-x^{d_i-1}).$$
So the generating function for $N_{j-1}-N_j$ is 
$-\ell x+\sum_1^{\ell}x^{d_i}$, hence 
\begin{equation}
q_\G(t)=(1-t)^{-{\ell}}  \prod_{j=1}^{\ell}(1-t^{d_j}).
\label{qG}
\end{equation} 

Associated to each cominuscule root is a maximal parabolic subgroup $\rm P$ whose Levi subgroup $\rm L$  is obtained by deleting $\alpha_i$ from the Dynkin diagram of $\G$ to produce a semisimple group and  has an additional $\C^\times$ factor corresponding to the coroot ${\alpha_i}^{\vee}$ (see \cite{Buch} for example). All simple roots of type $A_{\ell}$ are minuscule and we have dealt with those in the main body of the paper. For  the other groups the deleted Dynkin diagram is connected and so we get a simple group. The roots for which $c_i(\alpha)=0$ are the roots of $\rm L$.  The $\T$-action has weight one on  the simple root spaces  of $\G$ and hence also  on those of $\rm L$. This means that $\T$ acts with weight one on a linear combination $\sum_{j\ne i} \lambda_je_{\alpha_j}$ of root vectors. From \cite{Kos} if $\lambda_j\ne 0$, this is a principal nilpotent in the Levi subalgebra and so $\T$ acts as a semisimple element in a principal three-dimensional subgroup of the Levi group $\rm L$. 

Let $n_j$ be the degrees of  generators of the invariant polynomials on $\rm L$. The rank of $\rm L$ is $\ell$ and the simple component $\ell-1$.  Then 
$$q_{\rm L}(t)=(1-t)^{-{\ell}}  \prod_{j=1}^{\ell-1}(1-t^{n_j})$$
and 
\begin{equation}
m_\calE(t)=\prod_{\alpha>0, c_i(\alpha)\ne 0}\left(\frac{1-t^{w(\alpha)+1}}{1-t^{w(\alpha)}}\right)=\frac{q_\G(t)}{q_{\rm L}(t)}= \prod_{j=1}^{\ell}\frac{(1-t^{d_j})}{(1-t^{n_j})}.
\label{mminusc}
\end{equation}
\begin{proposition} \label{cominint} The virtual multiplicity $m_\calE(t)$ in equation (\ref{mminusc}) is a polynomial.
\end{proposition} 
\begin{proof} From \cite[Theorem 26.1]{Borel}  if $K,U$ are compact Lie groups of the same rank then the Poincar\'e polynomial of $K/U$ is 
	$$P(K/U,t)=\frac{(1-t^{s_1})(1-t^{s_2})\cdots (1-t^{s_{\ell}})}{(1-t^{r_1})(1-t^{r_2})\cdots (1-t^{r_{\ell}})}.$$
	where $s_i-1, r_i-1$ are the degrees of generators of the cohomology of $K,U$. Taking $K$ and $U$ to be the maximal compact subgroups of ${\rm P},{\rm L}$ then $s_j=2d_j, r_j=2n_j$ and   $P(K/U,t)$ is a polynomial $Q(t^2)$. Then $m_\calE(t)=Q(t)$. 
\end{proof}
A theorem (which also follows from the geometric Satake theorem of \cite{mirkovic-villonen}, c.f. \cite[Theorem 1.6.3]{ginzburg}) of B.Gross \cite[Corollary 6.4]{Gross}  states that the Lefschetz action of $\SL_2$ on the cohomology of the flag variety $\G/{\rm P}$ is isomorphic to the representation of the principal $\SL_2$ of the Langlands dual $\G^{\vee}$ on the  representation given by the minuscule weight dual to $\alpha_i$. In the light of our results on exterior powers for the linear group this suggests the following.

\begin{conjecture}  \label{cominmirror} Let $\alpha_{i_1},\dots,\alpha_{i_k}$ be distinct cominuscule roots and suppose that the corresponding sections $b_{i_j}$ of $L_{\alpha_{i_j}}K$ have $m$ distinct zeros $x_1,\dots, x_m$, and for other simple roots $b_i$ has no zero. 
	Then the fixed point  is very stable and the mirror of its upward flow is the tensor product of the vector bundles associated to the universal principal bundle for the Langlands dual group at $x_j$ in the corresponding minuscule representation. \end{conjecture}
	
	\begin{remark}
The referee has observed a further example beyond the general linear group supporting this conjecture. The cominuscule root for $\mathrm{Sp}_{2n}$ gives a fixed point which is a rank $2n$ symplectic Higgs bundle of the form 
$$L_1\stackrel{1}\rightarrow L_2\stackrel{1}\rightarrow\cdots \stackrel{1}\rightarrow L_n\stackrel{b_n}\rightarrow L_n^{-1}\stackrel{1}\rightarrow \cdots \stackrel{1}\rightarrow L_2^{-1}\stackrel{1}\rightarrow L_1^{-1}$$
	and if $b_n$ has no multiple zeros this is very stable as an $\mathrm {SL}_{2n}$-bundle and a fortiori as an $\mathrm{Sp}_{2n}$-bundle. 
	
	On the upward flow from this point, the Higgs field $\Phi$
	at a zero of $b_n$  preserves a Lagrangian subspace. At a regular point the eigenspaces of $\Phi$ lie in pairs  in $n$ $2$-dimensional symplectic subspaces so there are $2^n$ invariant Lagrangian subspaces which should correspond  to a basis for the spin representation of the Langlands dual group $\mathrm{Spin}_{2n+1}.$ 

\end{remark}

\subsection{Multiple zeros} \label{multiple}

We showed in Theorem 1.2 that a necessary and sufficient condition for a fixed point of type $(1,1,\dots,1)$ to be very stable is that $b$
should have no multiple zeros. It is natural to ask what happens when we do have a multiple zero. What is the closure of the upward flow, 
and is there a corresponding hyperholomorphic bundle on the mirror moduli space? 

We already know from Remark 5.15 that the intersection 
of the upward flow with a generic fibre has fewer points than the situation where $b$ has simple zeros. We also see in the proof of Lemma 4.13 
that a multiple zero leads to a Hecke curve in the closure -- the closure of the flow tangent to a nilpotent Higgs field. 

We illustrate this  here with the simplest 
example of this situation in rank $2$, for the group $\SL_2$. Let ${\mathcal E}$ be a fixed point of the $\T$-action defined by the vector bundle
$L\oplus L^*$ and the off-diagonal Higgs 
field $b: L\rightarrow L^*K$, defining  a section of $L^{-2}K$ which we assume has  a single double zero at $c\in C$, and hence $L\cong K^{1/2}(-c)$.
The upward
flow is given
by Higgs bundles $(E,\Phi)$ where $E$ is an extension $L\rightarrow E\stackrel p\rightarrow L^*$, $\tr \Phi=0$ and $b=p \Phi:L\rightarrow L^*K$.  

The spectral curve  $ C_a$ is defined by $\det(x-\Phi)=x^2+\det \Phi=0$ and $\sigma(x)= -x$ is the covering involution. Then if $ C_a$ is smooth, as we have seen, $E=\pi_*U$ where the line
bundle $U$ is the cokernel of $x-\Phi:\pi^*(E\otimes K^*)\rightarrow \pi^*E$. The inclusion $L\subset E\in H^0(C;\Hom(L,E))$ is canonically defined 
by the direct image construction as a section of the line bundle $\Hom (\pi^*L,U)$ on $C_a$.

Now $b$ vanishes at  $c\in C$ and suppose further that $\det \Phi(c)\ne 0$. Since  $L$ is preserved by $\Phi$ at $c$, it  maps to zero in the cokernel 
of $\Phi-x$ for one of the two eigenvalues $\pm x$. Then $p=(x,c)\in  C_a$ or $\sigma(p)=(-x,c)$ are zeros of the holomorphic section of $(\pi^*L)^*U$.
Since $\pi:C_a\rightarrow C$ is unramified at these points, these zeros have multiplicity $2$ since $b$ has a double zero at $c$.

The bundle $\Lambda^2 E$ is trivial which implies $U\cong M\pi^*K^{1/2}$ where $\sigma^*M\cong M^*$ (i.e. $M$ lies in the Prym variety) and so 
since $L\cong K^{1/2}(-c)$, $\pi^*L^* U\cong M(c)$ which has degree $2$. But we have a section which vanishes with multiplicity $2$ 
hence $M(c)\sim 2p$ or $2\sigma(p)$ as divisors, or, since $c\sim p+\sigma(p)$, $M\sim \pm(p-\sigma(p))$. Thus the 
intersection with the generic fibre consists of two points in the Prym variety whereas if $b$ had two distinct zeros the intersection 
would be four points. 

Note that if $q(c)=0$ and $C_a$ is smooth, there is a single point $p$ with $\pi(p)=c$ and $\pi$ locally has the form $z\mapsto z^2$. Since $b$ has a double zero at $c$, the section of $\pi^*L^* U$ has a simple zero which is a contradiction since it has degree $2$. Thus the upward flow does not intersect these fibres, but as we shall see next, the closure does. 

Suppose $(E,\Phi)$ is a limit point of the upward flow. There is still a non-zero homomorphism $K^{1/2}(-c)\rightarrow E$ but its image may not be an
embedding. It will generate a subbundle $L_0$ of larger degree. Since it has positive degree by  stability it cannot be invariant by the Higgs 
field $\Phi_0$ so there is a nonvanishing section of $L_0^{-2}K$. But $\deg L_0>\deg K^{1/2}(-c)=g-2$ hence $\deg L_0^{-2}K\le 0$. 
The existence of a section means that we have equality and $L_0=K^{1/2}$, so the bundle is an extension $K^{1/2}\rightarrow E\rightarrow 
K^{-1/2}$ and $b$ is an isomorphism. But from \cite[Proposition 3.3]{hitchin} the extension is trivial and the limit lies on the canonical section. Since 
$\det \Phi\in  H^0(C,K^2)$ parametrizes the section the limit lies on 
the $(3g-4)$-dimensional  linear subspace consisting of sections which vanish at $c$.

Consider now the uniformising Higgs bundle $K^{1/2}\oplus K^{-1/2}$ with Higgs field $1: K^{1/2}\rightarrow K^{-1/2}K$. The {\it downward}  
flow from this point consists of extensions $K^{-1/2}\rightarrow E\rightarrow K^{1/2}$ with the same nilpotent Higgs field mapping the quotient to the subbundle. This is a Lagrangian submanifold and the 
extensions are parametrized by a class $e\in H^1(C,K^*)$ (of  dimension $(3g-3)$ = $\dim {\mathcal M}/2$). The section $s$
of ${\mathcal O}(c)$ 
vanishing at $c$ defines a homomorphism $s:K^{1/2}(-c)\rightarrow K^{1/2}$ which lifts to $E$ if $se=0\in  H^1(C,K^*)$. Equivalently
$e=\delta(u)$ for
$u\in H^0(c,{\mathcal O}_c(K^*))\cong \C$ where $\delta$ is the connecting homomorphism in the cohomology exact sequence.  

In Dolbeault terms with a $C^{\infty}$ splitting $K^{1/2}\oplus K^{-1/2}$ we have 
$$\bar\partial+ \begin{pmatrix}0 & \bar\partial v/s\cr
0 & 0\end{pmatrix}\qquad \Phi= \begin{pmatrix}0 & 1\cr
0 & 0\end{pmatrix} $$
where (as in 4.1.4) $v$ is a local extension of $u$, zero outside a neighbourhood and holomorphic in a smaller one.  Then the embedding of $K^{1/2}(-c)\subset E$ is $(v,-s)$ so that $\Phi(-v,s)=(s,0)$ and 
the skew form 
on $E$ gives 
$\langle (v,-s),(s,0)\rangle= s^2$. This is equivalent to projecting onto the quotient and clearly has a double zero at $c$. Then each extension $E$ 
lies on the upward flow from ${\mathcal E}$ and has as limit the canonical fixed point. This is a projective line in the closure, a Hecke curve.

We see therefore that the union of the upward flow $W^+_c$ from $c$ and $W^+_0$ from the canonical point is a closed subspace, and the symplectic form vanishes at  smooth points. 

\begin{conjecture}  The mirror
	of the closed Lagrangian cycle $W^+_0\cup W^+_c$ is the universal adjoint bundle for $c$ in the $\SO_3$ moduli space. \end{conjecture}

Here is some motivation.  The subvariety $W^+_0\cup W^+_c$ intersects a generic fibre in three points and so the Fourier-Mukai transform is generically a rank $3$
vector bundle. The virtual multiplicity $m_\calE(t)$ for a double zero of $b$ is the same as for two distinct zeros and, from the point of view of the previous subsection this is  $(1+t)^2$, from the Poincar\'e polynomial of $\PP^1\times \PP^1$. Moreover the Lefschetz action of $\SL_2$ on the cohomology is the action on $V\otimes V$ where $V$ is the defining 2-dimensional representation. This is the direct sum of the adjoint  and the trivial one-dimensional 
representation so on the adjoint component the action is consistent with the conjecture.

Now the  adjoint bundle for $\SO_3$ is a rank $3$ bundle $V$ and $\Phi: K^*\rightarrow V$  is injective when $\Phi$ is regular at 
each point of $C$, i.e. when the spectral curve is smooth. Trivializing $K$
at $c$, the universal bundle admits a homomorphism from the trivial bundle. It follows that the canonical section, which corresponds to the trivial line bundle on the spectral curve,  forms a component of the 
support of the inverse transform. The other component clearly relates to the quotient sheaf which is generically a rank $2$ bundle. 

One point to note is that the hyperholomorphic connection on the universal bundle is not  reducible (its Pontryagin class is a generator of $H^4({\mathcal M},\R)$) and this casts a question mark over the
issue of whether the mirror of the closure of $W^+_c$ actually has a hyperholomorphic connection. 

More generally, there is a Hecke correspondence for every irreducible representation of the Langlands dual group, which has a natural compactification (see \cite[\S 9]{kapustin-witten}).  In our case, the image of the Hitchin section in the $\SL_2$ Higgs moduli space under this compactified Hecke  correspondence, attached to the adjoint representation of $\SO_3$, is precisely our Lagrangian cycle $W_0^+\cup W_c^+$. Its mirror then should be the universal adjoint bundle at $c$. The $\T$-equivariant index of its structure sheaf then should be $1+t+t^2$, which is also the $\T$-character of the mirror adjoint bundle at the canonical uniformising Higgs bundle. Incidently, it also matches the intersection Poincar\'e polynomial of a certain weighted projective space, which is the closure of the space of Hecke modifications at a point in our case (c.f. \cite[9.2]{kapustin-witten}). In fact, from the geometric Satake correspondence \cite{mirkovic-villonen,ginzburg} this generalises to the general case of simple (even reductive) groups, and we expect the analoguous results to hold there too. 

\subsection{The cotangent fibre} \label{cotfib}

The upward flow from a very stable bundle $(E,0)$ is an important example. 
We have calculated the rank of its transform in Remark~\ref{multN} using the multiplicity formula. The alternative method, as in \cite{beauville-etal} is to argue
that if there are no nilpotent Higgs fields then locally the moduli space is given by the vanishing of invariant polynomials. For a simple Lie group this is therefore the vanishing of  $(2d_i-1)(g-1)$ functions of degree $d_i$ for $i=1,\dots,\ell$ and so the rank of this bundle is
$$\prod_1^{\ell}d_i^{(2d_i-1)(g-1)},$$  which Laumon in \cite[Remark 6.2.3]{laumon2} called the global analogue of the order of the Weyl group. 
The structure of this bundle even for the $2^{3g-3}$ case of $SL_2$ is largely unknown, and less so a way to construct a hyperholomorphic connection on it.  In \eqref{cotmir}  we offered evidence from mirror symmetry for a conjecture of the $\T$-character of the  fibre of this bundle with a $\T$-equivariant structure at the fixed points of type $(1,1)$. 

It is an important example as  it relates to the work of Donagi--Pantev \cite{DP} on the 
geometric Langlands programme, where they proposed a geometric approach of constructing automorphic sheaves on the moduli stack of bundles, which was, in fact, also the motivation of  Drinfeld \cite{drinfeld} and Laumon \cite{laumon} for the introduction of very stable bundles, see also \cite[Conjecture 6.2.1]{laumon2}. This programme aims to associate to a point  $m\in {\mathcal M}^{\vee}$, the mirror of the Higgs bundle moduli space ${\mathcal M}$, a flat connection on ${\mathcal N}$, the moduli space of stable bundles, with singularities on the ``wobbly locus". A generic point in the mirror consists of a spectral curve $C_a$ and (according to the SYZ procedure) a line bundle $L$ over its Jacobian. In ${\mathcal M}$, the spectral curve defines a fibre of the Hitchin fibration. This fibre over $a\in {\mathcal A}$ intersects $T^*{\mathcal N}\subset {\mathcal M}$ in the complement of a subvariety and the projection to ${\mathcal N}$ extends to a rational map as in \cite{beauville-etal}. The idea is that the vector bundle $V$ on ${\mathcal N}$ is associated to a direct image of $L$, and the canonical 
1-form on the cotangent bundle $T^*{\mathcal N}$ defines a Higgs field on $V$, a section of $\End (V)\otimes T^*$. Then, using the higher dimensional nonlinear Hodge theory due to Simpson and  Mochizuki \cite{mochizuki}, one hopes to obtain a flat connection.  

What is the link with this paper? The complement of the wobbly locus in ${\mathcal N}$ is the set of very stable vector bundles $E$. The fibre of $V$  at a generic $[E]\in {\mathcal N}$ is the direct sum of the fibres of $L$ over the intersection of the upward flow of $(E,0)$ with the Hitchin fibre. Put another way, if ${\mathcal N}^{vs}\subset {\mathcal N}$ denotes the moduli space of very stable bundles, then the Fourier-Mukai transform of the upward flows over  ${\mathcal N}^{vs}$
defines a bundle on ${\mathcal N}^{vs}\times {\mathcal M}^{\vee}$. In this paper we have been concerned with the restriction to $[E]\times {\mathcal M}^{\vee}$, for Donagi and Pantev it is the restriction to ${\mathcal N}^{vs}\times \{m\}$.

In fact what we have encountered is a miniature version of this in the case of very stable Higgs bundles of type $(1,1,\dots, 1)$. Instead of the $\T$-fixed point set ${\mathcal N}$ we have another component: a subvariety defined by the zero set of $(b_1,\dots, b_{n-1})$ in the product of symmetric products $C^{[m_1]}\times C^{[m_2]}\times \cdots \times C^{[m_{n-1}]}$. Moreover we have explicitly  identified  the wobbly locus, namely where the product $b_1\dots b_{n-1}$ has multiple zeroes. This  includes the various diagonals. The hyperholomorphic bundle we identified as a tensor product of exterior powers of universal bundles corresponding to the zeros of the $b_i$. By nonlinear Hodge theory on $C$, a point $m\in {\mathcal M}^{\vee}$ defines a flat connection on the curve $C$ and hence also on the product $C^{m_1}\times C^{m_2}\times \cdots \times C^{m_{n-1}}$, invariant by the action of the product of symmetric groups, hence descending to a flat connection, but with singularities over the diagonals.

    \end{document}